\documentclass[reqno,a4paper,draft]{amsart}

\usepackage{enumitem}
\setenumerate{label=\textnormal{(\arabic*)}}

\usepackage{amsmath,amssymb,dsfont,verbatim,bm}
\usepackage{mathtools}
\usepackage[latin1]{inputenc}

\usepackage[raggedright]{titlesec}

\usepackage{tikz}
\usepackage{float,subfig}
\usepackage{amsrefs}

\titleformat{\chapter}[display]
{\normalfont\huge\bfseries}{\chaptertitlename\\thechapter}{20pt}{\Huge}
\titleformat{\section}
{\normalfont\Large\bfseries\center}{\thesection}{1em}{}
\titleformat{\subsection}
{\normalfont\large\bfseries}{\thesubsection}{1em}{}
\titleformat{\subsubsection}[runin]
{\normalfont\normalsize\bfseries}{\thesubsubsection}{1em}{}
\titleformat{\paragraph}[runin]
{\normalfont\normalsize\bfseries}{\theparagraph}{1em}{}
\titleformat{\subparagraph}[runin]
{\normalfont\normalsize\bfseries}{\thesubparagraph}{1em}{}

\titlespacing*{\chapter} {0pt}{50pt}{40pt}
\titlespacing*{\section} {0pt}{3.5ex plus 1ex minus .2ex}{2.3ex plus .2ex}
\titlespacing*{\subsection} {0pt}{3.25ex plus 1ex minus .2ex}{1.5ex plus .2ex}
\titlespacing*{\subsubsection}{0pt}{3.25ex plus 1ex minus .2ex}{1.5ex plus .2ex}
\titlespacing*{\paragraph} {0pt}{3.25ex plus 1ex minus .2ex}{1em}
\titlespacing*{\subparagraph} {\parindent}{3.25ex plus 1ex minus .2ex}{1em}


\newtheorem{theorem}{Theorem}[section]
\newtheorem{lemma}[theorem]{Lemma}
\newtheorem{proposition}[theorem]{Proposition}
\newtheorem{corollary}[theorem]{Corollary}

\theoremstyle{definition}
\newtheorem{definition}[theorem]{Definition}
\newtheorem{notations}[theorem]{Notations}
\newtheorem{notation}[theorem]{Notation}
\newtheorem{example}[theorem]{Example}

\theoremstyle{remark}
\newtheorem{remark}[theorem]{Remark}

\DeclareMathOperator{\Aut}{Aut}

\DeclareMathOperator{\ide}{id}

\DeclareMathOperator{\Supp}{Supp}

\DeclareMathOperator{\en}{en}

\DeclareMathOperator{\factors}{factors}

\DeclareMathOperator{\st}{st}

\DeclareMathOperator{\PE}{PE}
\DeclareMathOperator{\val}{val}
\DeclareMathOperator{\Valsup}{Valsup}
\DeclareMathOperator{\Valinf}{Valinf}
\DeclareMathOperator{\Succ}{Succ}
\DeclareMathOperator{\Pred}{Pred}
\DeclareMathOperator{\Val}{Val}

\DeclareMathOperator{\lcm}{lcm}

\newcommand{\ov}{\overline}

\begin{document}

\title{The Dixmier conjecture  and the shape of possible counterexamples II}

\author{Jorge A. Guccione}
\address{Departamento de Matem\'atica\\ Facultad de Ciencias Exactas y Naturales-UBA, Pabell\'on~1-Ciudad Universitaria\\ Intendente Guiraldes 2160 (C1428EGA) Buenos Aires, Argentina.}
\address{Instituto de Investigaciones Matem\'aticas ``Luis A. Santal\'o"\\ Facultad de Ciencias Exactas y Naturales-UBA, Pabell\'on~1-Ciudad Universitaria\\ Intendente Guiraldes 2160 (C1428EGA) Buenos Aires, Argentina.}
\email{vander@dm.uba.ar}

\author{Juan J. Guccione}
\address{Departamento de Matem\'atica\\ Facultad de Ciencias Exactas y Naturales-UBA\\ Pabell\'on~1-Ciudad Universitaria\\ Intendente Guiraldes 2160 (C1428EGA) Buenos Aires, Argentina.}
\address{Instituto Argentino de Matem\'atica-CONICET\\ Savedra 15 3er piso\\ (C1083ACA) Buenos Aires, Argentina.}
\email{jjgucci@dm.uba.ar}

\thanks{Jorge A. Guccione and Juan J. Guccione research were supported by PIP 112-200801-00900 (CONICET)}

\author{Christian Valqui}
\address{Pontificia Universidad Cat\'olica del Per\'u - Instituto de Matem\'atica y Ciencias Afi\-nes, Secci\'on Matem\'aticas, PUCP, Av. Universitaria 1801, San Miguel, Lima 32, Per\'u.}

\address{Instituto de Matem\'atica y Ciencias Afines (IMCA) Calle Los Bi\'ologos 245. Urb San C\'esar. La Molina, Lima 12, Per\'u.}
\email{cvalqui@pucp.edu.pe}

\thanks{Christian Valqui research was supported by PUCP-DGI-2011-0206 and Lucet 90-DAI-L005.}


\begin{abstract} We continue with the investigation began in ``The Dixmier conjecture  and the shape of possible counterexamples''. In that paper we introduced the notion of an irreducible pair $(P,Q)$ as the image of the pair $(X,Y)$ of the canonical generators of $W$ via an endomorphism which is not an automorphism, such that it cannot be made ``smaller'', we let $B$ denote the minimum of the greatest common divisor of the total degrees of $P$ and $Q$, where $(P,Q)$ runs on the irreducible pairs and we prove that $B\ge 9$. In the present work we improve this lower bound by proving that $B\ge 15$. In order to do this we need to show the the main results of our previous paper remain valid for a family of algebras $(W^{(l)})_{l\in \mathds{N}}$ that extend~$W$.
\end{abstract}

\maketitle

\section*{Introduction}

In this paper $K$ is a characteristic zero field, $W$ is the Weyl algebra on $K$, that is the unital associative $K$-algebra generate by elements $X,Y$ and the relation $[Y,X]=1$. In~\cite{D} Dixmier posed a question, nowadays known as the Dixmier conjecture: is an algebra endomorphism of the Weyl algebra $W$ on a characteristic zero field, necessarily an automorphism? Currently, the Dixmier conjecture remains open. In 2005 the stable equivalence between the Dixmier and Jacobian conjectures was established in~\cite{T} by Yoshifumi Tsuchimoto. In~\cite{G-G-V2} we introduced the notion of an irreducible pair $(P,Q)$ as the image of the pair $(X,Y)$ of the canonical generators of $W$ via an endomorphism which is not an automorphism, such that it cannot be made ``smaller''. Following the strategy of describing the generators of possible counterexamples, we proved the following result: If the Dixmier conjecture is false, then there exist and irreducible pair $(P,Q)$ such that the support of both $P$ and $Q$ is subrectangular. We also made a first ``cut'' at the lower right edge of the support of such pairs, which gave us a lower bound for
$$
B:=\min\{\gcd(v_{1,1}(P),v_{1,1}(Q)),\quad \text{where $(P,Q)$ is an irreducible pair}\}.
$$
We managed to prove that $B\ge 9$.

In the present work we will start with an irreducible subrectangular pair $(P,Q)$, and cut further the lower right edge of the support. For this we need to embed $P$ and $Q$ in a bigger algebra $W^{(l)}$, basically adjoining fractional powers of $X$. As a $K$-linear space $W^{(l)}$ is $K[X,X^{-1/l},Y]$ and the relation $[Y,X]=1$ is preserved.

In the first four sections we carry over the results of~\cite{G-G-V2} from $W$ to $W^{(l)}$. This comprehends basically the leading terms associated to a valuation, the corresponding polynomials $f_{P,\rho,\sigma}$ and the $(\rho,\sigma)$-bracket. Using the same differential equation found in~\cite{G-G-V2} we arrive at the Theorem~\ref{th tipo irreducibles}, which asserts the existence of a $(\rho,\sigma)$-homogeneous element $F$ with $[P,F] = \ell_{\rho,\sigma}(P)$. As in~\cite{G-G-V2}, it is a powerful tool to restrict the possible geometric shapes of irreducible pairs and of some pairs constructed out of them.

The central technical result in this article is Proposition~\ref{preparatoria} which generalizes~\cite[Prop.6.2]{G-G-V2}. It allows to ``cut'' the right lower edge of the support of a given pair in $W^{(l)}$. Starting from an irreducible pair $(P_0,Q_0)$, in Theorem~\ref{familia} we generate a finite chain of pairs $(P_i,Q_i)$, with $[P_i,Q_i]_{\rho_i,\sigma_i}=0$ for all but the last pair.

Proposition~\ref{preparatoria} also allows to increase the lower bound for $B$ from~$9$ to~$15$. For this we only have analyze extensively two cases
$$
\frac{1}{m}\en_{\rho,\sigma}(P)\notin \{(3,6),(4,6)\},
$$
which is done in Proposition~\ref{pr no (3,6) ni (4,6)}. To eliminate the possibility of $B=15$ we would have to analyze the cases $(5,10)$ and $(6,9)$, and for $B=16$ we need to analyze $(6,10)$ and $(4,12)$. Instead of making such an extensive analysis, we will develop an algorithm to carry out an exhaustive search for the smallest possible complete chain $(S_j)$ as in Proposition~\ref{chains}, probably with the use of computers.

The present work yields the necessary tools for that purpose, which will increase significatively the lower bound for $B$. Additionally, it also indicates precisely the exact location of the corners of the possible counterexamples, which simplifies the search.

\section{Preliminaries}\label{preliminares}

\setcounter{equation}{0}

In order to continue with the study of the irreducible pairs introduced in~\cite{G-G-V2} it is convenient to consider some extensions of the Weyl algebra $W$. On these extensions we will make similar constructions as in~\cite{G-G-V2}, thus extending several results of that paper.

\smallskip

For each $l\in \mathds{N}$, we define an algebra $W^{(l)}$ as the Ore extension $A[Y,\ide,\delta]$, where $A$ is the algebra of Laurent polynomials $K[Z_l,Z_l^{-1}]$ and $\delta\colon A\to A$ is the derivation, defined by $\delta(Z_l) = \frac{1}{l} Z_l^{1-l}$. Suppose that $l,h\in \mathds{N}$ such that $l|h$ and let $d := h/l$. We have
$$
[Y,Z_h^d] = \sum_{i=0}^{d-1} Z_h^i[Y,Z_h]Z_h^{d-i-1} = \frac{d}{h} Z_h^{d-h} = \frac{1}{l} (Z_h^d)^{1-l}.
$$
Hence there is an inclusion $\iota_l^h\colon W^{(l)}\to W^{(h)}$, given by $\iota_l^h(Z_l) := Z_h^d$ and $\iota_l^h(Y) := Y$.

\smallskip

We will write $X^{\frac{1}{l}}$ and $X^{\frac{-1}{l}}$ instead of $Z_l$ and $Z_l^{-1}$, respectively. With this notation the map $\iota_l^h$ satisfies $\iota_l^h(X^{\frac{1}{l}}) = (X^{\frac{1}{h}})^d$. We will consider $W^{(l)}\subseteq W^{(h)}$ via this inclusion. Note that $W\subseteq W^{(1)}$.

\smallskip

Similarly, for each $l\in \mathds{N}$, we consider the commutative $K$-algebra $L^{(l)}$, generated by variables $x^{\frac{1}{l}}$, $x^{\frac{-1}{l}}$ and $y$, subject to the relation $x^{\frac{1}{l}} x^{\frac{-1}{l}} = 1$. In other words $L^{(l)} = K[x^{\frac{1}{l}}, x^{\frac{-1}{l}},y]$. Obviously, there is a canonical inclusion $L^{(l)}\subseteq L^{(h)}$, for each $l,h\in \mathds{N}$ such that $l|h$. We let $\Psi^{(l)}\colon W^{(l)}\to L^{(l)}$ denote the $K$-linear map defined by $\Psi^{(l)}\bigl(X^{\frac{i}{l}}Y^j\bigr) := x^{\frac{i}{l}} y^j$. As in~\cite{G-G-V2}, let
\begin{align*}
& \ov{\mathfrak{V}} := \{(\rho,\sigma)\in \mathds{Z}^2: \text{$\gcd(\rho,\sigma) = 1$ and $\rho+\sigma\ge 0$}\}
\shortintertext{and}
&\mathfrak{V} := \{(\rho,\sigma)\in \ov{\mathfrak{V}}: \rho+\sigma> 0 \}.
\end{align*}

Now we extend to the algebras $W^{(l)}$ and $L^{(l)}$ some well know definitions and notations given in~\cite{G-G-V2} for $W$ and the polynomial algebra $L:=K[x,y]$.

\begin{definition} For all $(\rho,\sigma)\in \ov{\mathfrak{V}}$ and $(i/l,j)\in \frac{1}{l}\mathds{Z}\times \mathds{Z}$ we write
$$
v_{\rho,\sigma}(i/l,j):= \rho i/l+\sigma j.
$$
\end{definition}

\begin{notations}\label{not valuaciones para polinomios} Let $(\rho,\sigma)\in \ov{\mathfrak{V}}$. For $P = \sum a_{\frac{i}{l},j} x^{\frac{i}{l}} y^j\in L^{(l)}\setminus\{0\}$, we define:

\begin{itemize}

\smallskip

\item[-] The {\em support} of $P$ as
$$
\Supp(P) := \left\{\left(\frac{i}{l},j\right): a_{\frac{i}{l},j}\ne 0\right\}.
$$

\smallskip

\item[-] The {\em $(\rho,\sigma)$-degree} of $P$ as $v_{\rho,\sigma}(P):= \max\left\{ v_{\rho,\sigma} \bigl(\frac{i}{l},j\bigr): a_{\frac{i}{l},j} \ne 0\right\}$.

\smallskip

\item[-] The {\em $(\rho,\sigma)$-leading term} of $P$ as
$$
\ell_{\rho,\sigma}(P):= \displaystyle \sum_{\{\rho \frac{i}{l} + \sigma j = v_{\rho,\sigma}(P)\}} a_{\frac{i}{l},j} x^{\frac{i}{l}} y^j.
$$

\smallskip

\item[-] $w(P):= \left(\frac{i_0}{l},\frac{i_0}{l}-v_{1,-1}(P)\right)$ such that
$$
\frac{i_0}{l} = \max \left\{\frac{i}{l}: \left(\frac{i}{l},\frac{i}{l}-v_{1,-1}(P) \right) \in \Supp(\ell_{1,-1}(P))\right\},
$$

\smallskip

\item[-] $\ell_c(P):= a_{\frac{i_0}{l}j_0}$, where $\left(\frac{i_0}{l},j_0\right)= w(P)$.

\smallskip

\item[-] $\ell_t(P):= a_{\frac{i_0}{l}j_0} x^{\frac{i_0}{l}}y^{j_0}$, where $\left(\frac{i_0}{l},j_0\right)= w(P)$.

\smallskip

\item[-] $\ov{w}(P):= \left(\frac{i_0}{l}-v_{-1,1}(P),\frac{i_0}{l}\right)$ such that
$$
\frac{i_0}{l} = \max \left\{\frac{i}{l}: \left(\frac{i}{l}-v_{-1,1}(P),\frac{i}{l} \right) \in \Supp(\ell_{-1,1}(P))\right\},
$$

\smallskip

\item[-] $\ov{\ell}_c(P):= a_{\frac{i_0}{l}j_0}$, where $\left(\frac{i_0}{l},j_0\right)= \ov{w}(P)$.

\smallskip

\item[-] $\ov{\ell}_t(P):= a_{\frac{i_0}{l}j_0} x^{\frac{i_0}{l}}y^{j_0}$, where $\left(\frac{i_0}{l},j_0\right)= \ov{w}(P)$.

\end{itemize}
\end{notations}

\begin{notations}\label{not valuaciones para alg de Weyl} Let $(\rho,\sigma)\in \ov{\mathfrak{V}}$. For $P \in W^{(l)}\setminus\{0\}$, we define:

\begin{itemize}

\smallskip

\item[-] The {\em support} of $P$ as $\Supp(P) := \Supp\bigl(\Psi^{(l)}(P)\bigr)$.

\smallskip

\item[-] The {\em $(\rho,\sigma)$-degree} of $P$ as $v_{\rho,\sigma}(P):= v_{\rho,\sigma} \bigl(\Psi^{(l)}(P)\bigr)$.

\smallskip

\item[-] The {\em $(\rho,\sigma)$-leading term} of $P$ as $\ell_{\rho,\sigma}(P):= \ell_{\rho,\sigma} \bigl(\Psi^{(l)}(P)\bigr)$.

\smallskip

\item[-] $w(P):= w\bigl(\Psi^{(l)}(P)\bigr)$.

\smallskip

\item[-] $\ell_c(P):= \ell_c\bigl(\Psi^{(l)}(P)\bigr)$.

\smallskip

\item[-] $\ell_t(P):= \ell_c(P) X^{\frac{i_0}{l}}Y^{j_0}$, where $\left(\frac{i_0}{l},j_0 \right) = w(P)$.

\smallskip

\item[-] $\ov{w}(P):= \ov{w}\bigl(\Psi^{(l)}(P)\bigr)$.

\smallskip

\item[-] $\ov{\ell}_c(P):= \ov{\ell}_c\bigl(\Psi^{(l)}(P)\bigr)$.

\smallskip

\item[-] $\ov{\ell}_t(P):= \ov{\ell}_c(P) X^{\frac{i_0}{l}}Y^{j_0}$, where $\left(\frac{i_0}{l},j_0\right)= \ov{w}(P)$.

\end{itemize}
\end{notations}

\begin{notation}\label{not elementos rho-sigma homogeneos} We say that $P\in L^{(l)}$ is {\em $(\rho,\sigma)$-homogeneous} if $P = 0$ or $P = \ell_{\rho,\sigma}(P)$. Moreover we say that $P\in W^{(l)}$ is {\em $(\rho,\sigma)$-homogeneous} if $\Psi^{(l)}(P)$ is so.
\end{notation}

\begin{definition}\label{Comienzo y Fin de un elemento de W^{(l)}} Let $P\in W^{(l)}\setminus\{0\}$. We define
$$
\st_{\rho,\sigma}(P) = w(\ell_{\rho,\sigma}(P))\quad\text{for $(\rho,\sigma)\in \ov{\mathfrak{V}}\setminus (1,-1)$}
$$
and
$$
\en_{\rho,\sigma}(P) = \ov{w}(\ell_{\rho,\sigma}(P)) \quad\text{for $(\rho,\sigma)\in \ov{\mathfrak{V}}\setminus (-1,1)$}.
$$
\end{definition}

\begin{lemma}\label{le conmutacion en Wl} For each $l\in \mathds{N}$, we have
$$
Y^jX^{\frac{i}{l}} = \sum_{k=0}^j k!\binom{j}{k}\binom{i/l}{k} X^{\frac{i}{l}-k} Y^{j-k}.
$$
\end{lemma}

\begin{proof} It follows easily using that
$$
[Y,X^{\frac{i}{l}}] = \frac{i}{l} X^{\frac{i}{l}-1},\qquad [Y^j,X^{\frac{i}{l}}] = [Y,X^{\frac{i}{l}}] Y^{j-1}+Y[Y^{j-1},X^{\frac{i}{l}}]
$$
and an induction argument.
\end{proof}

For $l\in \mathds{N}$ and $j\in \mathds{Z}$, we set
$$
W^{(l)}_{j/l}:=\left\{P\in W^{(l)}\setminus\{0\} : P \text{ is $(1,-1)$-homogeneous and } v_{1,-1}(P) = \frac{j}{l}\right\}\cup \{0\}.
$$

\begin{remark}\label{graduacion} It is easy to see that $W^{(l)}_{j/l}$ is a subvector space of $W^{(l)}$. Moreover, by Lemma~\ref{le conmutacion en Wl}, we know that $W^{(l)}$ is a $\frac{1}{l}\mathds{Z}$-graded algebra with $W^{(l)}_{j/l}$ the $(1,-1)$-homogeneous component of degree $\frac{j}{l}$, and by~\cite[Lemma~2.1]{G-G-V1}, we know that $W^{(l)}_0 = K[XY]$, and hence commutative.
\end{remark}

\begin{proposition}\label{pr v de un producto} Let $P,Q\in W^{(l)}\setminus\{0\}$. The following assertions hold:

\begin{enumerate}

\smallskip

\item $w(PQ) = w(P) + w(Q)$ and $\ov{w}(PQ) = \ov{w}(P) + \ov{w}(Q)$. In particular $PQ\ne 0$.

\smallskip

\item $\ell_{\rho,\sigma}(PQ) = \ell_{\rho,\sigma}(P)\ell_{\rho,\sigma}(Q)$ for all $(\rho,\sigma)\in \mathfrak{V}$.

\smallskip

\item $v_{\rho,\sigma}(PQ) = v_{\rho,\sigma}(P) + v_{\rho,\sigma}(Q)$ for all $(\rho,\sigma)\in \ov{\mathfrak{V}}$.

\smallskip

\item $\st_{\rho,\sigma}(PQ)= \st_{\rho,\sigma}(P)+\st_{\rho,\sigma}(Q)$  for all $(\rho,\sigma)\in \mathfrak{V}$.

\smallskip

\item $\en_{\rho,\sigma}(PQ)= \en_{\rho,\sigma}(P)+\en_{\rho,\sigma}(Q)$ for all $(\rho,\sigma)\in \mathfrak{V}$.

\end{enumerate}
The same properties hold for $P,Q\in L^{(l)}\setminus\{0\}$.
\end{proposition}

\begin{proof} For $P,Q\in W^{(l)}\setminus\{0\}$ this follows easily from Lemma~\ref{le conmutacion en Wl} using that $\rho+\sigma >0$ if $(\rho,\sigma)\in \mathfrak{V}$. The proof for $P,Q\in L^{(l)}\setminus\{0\}$ is easier.
\end{proof}

Let $A=(a_1,a_2)$ and $B=(b_1,b_2)$ in $\mathds{R}^2$. As in~\cite{G-G-V2} we say that $A$ and $B$ are {\em aligned} if $A\times B := \det\left(\begin{smallmatrix} a_1 & a_2\\ b_1 & b_2\end{smallmatrix}\right)$ is zero.

\begin{definition}\label{def alineados} Let $P,Q\in L^{(l)}\setminus\{0\}$. We say that $P$ and $Q$ are {\em aligned} and write $P\sim Q$, if $w(P)$ and $w(Q)$ are so. Moreover we say that $P,Q\in W^{(l)}\setminus\{0\}$ are aligned if $\Psi^{(l)}(P)\sim \Psi^{(l)}(Q)$. Note that
\begin{itemize}

\smallskip

\item[-] By definition $P\sim Q$ if and only if $\ell_{1,-1}(P)\sim \ell_{1,-1}(Q)$.

\smallskip

\item[-] $\sim$ is not an equivalence relation (it is so restricted to $\{P:w(P)\ne (0,0)\}$).

\smallskip

\item[-] If $P\sim Q$ and $w(P)\ne (0,0)\ne w(Q)$, then $w(P)=\lambda w(Q)$ with $\lambda\ne 0$.

\end{itemize}
\end{definition}

\begin{proposition}\label{pr conmutadores con exponentes fraccinarios} Let $P,Q\in W^{(l)}\setminus \{0\}$. The following assertions hold:

\begin{enumerate}

\smallskip

\item If $P \nsim Q$, then
$$
[P,Q]\ne 0\quad\text{and}\quad w\bigl([P,Q]\bigr) = w(P)+w(Q)-(1,1).
$$

\smallskip

\item If $\ov{w}(P) \nsim \ov{w}(Q)$, then
$$
[P,Q]\ne 0\quad\text{and}\quad \ov w\bigl([P,Q]\bigr) = \ov w(P)+\ov w(Q)-(1,1).
$$
\end{enumerate}
\end{proposition}

\begin{proof} We only prove item~(1) since item~(2) is similar. Let $w(P) = \left(\frac{r}{l},s\right)$ and $w(Q) = \left(\frac{u}{l},v\right)$. Since $\binom{s}{1}\binom{u/l}{1} - \binom{v}{1} \binom{r/l}{1} = (r/l,s) \times (u/l,v)\ne 0$, using Lemma~\ref{le conmutacion en Wl} one can check that
$$
\ell_t\bigl([P,Q]\bigr) = \Biggl(\binom{s}{1}\binom{u/l}{1} - \binom{v}{1} \binom{r/l}{1} \Biggr) \ell_c(P)\ell_c(Q)X^{\frac{r+u}{l}-1}Y^{s+v-1}.
$$
So, $w\bigl([P,Q]\bigr) = w(P)+w(Q)-(1,1)$.
\end{proof}

\begin{remark}\label{re v de un conmutador} For all $P,Q\in W^{(l)}\setminus\{0\}$ and each $(\rho,\sigma)\in \ov{\mathfrak{V}}$, we have
$$
[P,Q] = 0\quad\text{or}\quad v_{\rho,\sigma}([P,Q])\le v_{\rho,\sigma}(P) + v_{\rho,\sigma}(Q) - (\rho+\sigma).
$$
\end{remark}

\section{The bracket associated with a valuation}

\setcounter{equation}{0}

\begin{definition}\label{def rho-sigma proporcionales} Let $(\rho,\sigma)\in\ov{\mathfrak{V}}$ and $P,Q\in W^{(l)}\setminus\{0\}$. We say that $P$ and $Q$ are {\em $(\rho,\sigma)$-proportional} if $[P,Q]=0$ or $v_{\rho,\sigma}([P,Q]) < v_{\rho,\sigma}(P) + v_{\rho,\sigma}(Q) - (\rho+\sigma)$.
\end{definition}

\begin{definition}\label{def rho-sigma corchete} Let $l\in \mathds{N}$ and $(\rho,\sigma)\in \ov{\mathfrak{V}}$. We define
$$
[-,-]_{\rho,\sigma}\colon \bigl(W^{(l)}\setminus\{0\}\bigr)\times \bigl(W^{(l)}\setminus\{0\}\bigr)\to L^{(l)}_{\rho,\sigma},
$$
by
$$
[P,Q]_{\rho,\sigma} = \begin{cases} 0 &\text{if $P$ and $Q$ are $(\rho,\sigma)$-proportional,}\\ \ell_{\rho,\sigma}([P,Q]) &\text{if $P$ and $Q$ are not $(\rho,\sigma)$-proportional.}
\end{cases}
$$
\end{definition}

\begin{lemma}\label{pr conmutadores de elementos homogeneos con exponentes fraccinarios} Let $(\rho,\sigma)\in \mathfrak{V}$ and let $P$ and $Q$ be $(\rho,\sigma)$-homogeneous elements of $W^{(l)}\setminus \{0\}$. Assume that $\sigma\le 0$.

\begin{enumerate}

\smallskip

\item If $w(P)\nsim w(Q)$, then $[P,Q] \ne 0$ and $w([P,Q]) = w\bigl(\ell_{\rho,\sigma}([P,Q])\bigr)$.

\smallskip

\item If $\ov{w}(P)\nsim \ov{w}(Q)$, then $[P,Q]\ne 0$ and $\ov{w}([P,Q])= \ov{w}\bigl(\ell_{\rho,\sigma} ([P,Q])\bigr)$.
\end{enumerate}

\end{lemma}

\begin{proof} We only prove item~(1) since item~(2) is similar. Write
$$
\qquad P = \sum_{i=0}^{\alpha} \lambda_i X^{\frac{r}{l}-\frac{i\sigma} {\rho}} Y^{s+i} \quad\text{and}\quad Q = \sum_{j=0}^{\beta} \mu_j X^{\frac{u}{l}-\frac{j\sigma}{\rho}} Y^{v+j},
$$
with $\lambda_0,\lambda_{\alpha},\mu_0,\mu_{\beta}\ne 0$. Since, by Lemma~\ref{le conmutacion en Wl},
$$
X^{\frac{i}{l}}Y^j X^{\frac{i'}{l}}Y^{j'} = \sum_{k=0}^j  k!\binom{j}{k} \binom{i'/l}{k} X^{\frac{i+i'}{l}-k}Y^{j+j'-k},
$$
we obtain that
$$
[P,Q] = \sum_{i=0}^{\alpha}\sum_{j=0}^{\beta}\sum_{k=0}^{\max\{s+i,v+j\}} \lambda_i\mu_j c_{ijk} X^{\frac{r+u}{l}-\frac{(i+j)\sigma} {\rho}-k} Y^{s+v+i+j-k},
$$
where
$$
c_{ijk} = k!\binom{s+i}{k} \binom{u/l-j\sigma/\rho}{k} - k!\binom{v+j}{k} \binom{r/l-i\sigma/\rho}{k}.
$$
Note that $c_{ij0} = 0$. Furthermore $c_{001}\ne 0$, because $w(P)\nsim w(Q)$. Consequently, since $\rho+\sigma>0$,
$$
\ell_{\rho,\sigma}([P,Q]) = \sum_{i=0}^{\alpha}\sum_{j=0}^{\beta} \lambda_i\mu_j c_{ij1} x^{\frac{r+u}{l}-\frac{(i+j)\sigma} {\rho}-1} y^{s+v+i+j-1}.
$$
Using again that $c_{001}\ne 0$, we obtain that
$$
w([P,Q]) = \left(\frac{r+u}{l}-1,s+v-1\right) = w\bigl(\ell_{\rho,\sigma}([P,Q])\bigr),
$$
as desired.
\end{proof}

\begin{proposition}\label{extremosnoalineados} Let $P,Q,R\in W^{(l)}\setminus \{0\}$ such that $[P,Q]_{\rho,\sigma} = \ell_{\rho,\sigma}(R)$, where $(\rho,\sigma)\in \mathfrak{V}$. Assume that $\sigma\le 0$. We have

\begin{enumerate}

\smallskip

\item If $\st_{\rho,\sigma}(P)\nsim \st_{\rho,\sigma}(Q)$, then
$$
\st_{\rho,\sigma}(P)+\st_{\rho,\sigma}(Q)-(1,1)=\st_{\rho,\sigma}(R).
$$

\smallskip

\item If $\en_{\rho,\sigma}(P)\nsim \en_{\rho,\sigma}(Q)$, then
$$
\en_{\rho,\sigma}(P)+\en_{\rho,\sigma}(Q)-(1,1)=\en_{\rho,\sigma}(R).
$$

\end{enumerate}
\end{proposition}

\begin{proof} We only prove item~(1) and leave the proof of item~(2), which is similar, to the reader. Let $P_1$ and $Q_1$ be $(\rho,\sigma)$-homogeneous elements of $W^{(l)}\setminus\{0\}$, such that
\begin{equation}
v_{\rho,\sigma}(P-P_1) < v_{\rho,\sigma}(P_1)\quad\text{and}\quad v_{\rho,\sigma}(Q-Q_1) < v_{\rho,\sigma}(Q_1).\label{eq13}
\end{equation}
Since
$$
[P,Q] = [P_1,Q_1] + [P_1,Q-Q_1] + [P-P_1,Q],
$$
and, by Remark~\ref{re v de un conmutador}, we have
\begin{align*}
v_{\rho,\sigma}([P_1,Q-Q_1])&\le v_{\rho,\sigma}(P_1) + v_{\rho,\sigma}(Q-Q_1) - (\rho+\sigma)\\
& < v_{\rho,\sigma}(P_1) + v_{\rho,\sigma}(Q_1) - (\rho+\sigma)\\
& = v_{\rho,\sigma}(P) + v_{\rho,\sigma}(Q) - (\rho+\sigma)
\shortintertext{and}
v_{\rho,\sigma}([P-P_1,Q]) & < v_{\rho,\sigma}(P) + v_{\rho,\sigma}(Q) - (\rho+\sigma),
\end{align*}
it follows from the fact that $P$ and $Q$ are not $(\rho,\sigma)$-proportional, that
\begin{equation}
v_{\rho,\sigma}([P,Q]-[P_1,Q_1]) < v_{\rho,\sigma}([P,Q]).\label{eq14}
\end{equation}
Note that~\eqref{eq13} and~\eqref{eq14} imply
$$
\ell_{\rho,\sigma}(P) = \ell_{\rho,\sigma}(P_1),\quad \ell_{\rho,\sigma}(Q) = \ell_{\rho,\sigma}(Q_1)\quad \text{and} \quad \ell_{\rho,\sigma}([P,Q]) = \ell_{\rho,\sigma}([P_1,Q_1]).
$$
Consequently, by item~(1) of Proposition~\ref{pr conmutadores con exponentes fraccinarios} and item~(1) of Lemma~\ref{pr conmutadores de elementos homogeneos con exponentes fraccinarios},
\begin{align*}
\st_{\rho,\sigma}(P)+\st_{\rho,\sigma}(Q)-(1,1) &= \st_{\rho,\sigma}(P_1)+\st_{\rho,\sigma}(Q_1) -(1,1)\\
&= w(P_1)+w(Q_1)-(1,1)\\
&= w([P_1,Q_1])\\
&= w\bigl(\ell_{\rho,\sigma}([P_1,Q_1])\bigr)\\
&= w\bigl(\ell_{\rho,\sigma}([P,Q])\bigr)\\
&= w\bigl(\ell_{\rho,\sigma}(R)\bigr)\\
&= \st_{\rho,\sigma}(R),
\end{align*}
as desired.
\end{proof}

\begin{proposition}\label{calculo del corchete} Let $(\rho,\sigma)\in \mathfrak{V}$ and $P,Q\in W^{(l)}\setminus \{0\}$. Assume that $\sigma\le 0$. If
$$
\qquad \ell_{\rho,\sigma}(P) = \sum_{i=0}^{\alpha} \lambda_i x^{\frac{r}{l}-\frac{i\sigma} {\rho}} y^{s+i}\quad\text{and}\quad \ell_{\rho,\sigma}(Q) = \sum_{j=0}^{\beta} \mu_j x^{\frac{u}{l}-\frac{j\sigma}{\rho}} y^{v+j},
$$
with $\lambda_0,\lambda_{\alpha},\mu_0,\mu_{\beta}\ne 0$, then
$$
[P,Q]_{\rho,\sigma} = \sum \lambda_i\mu_j c_{ij} x^{\frac{r+u}{l}-\frac{(i+j)\sigma} {\rho}-1} y^{s+v+i+j-1}.
$$
where $c_{ij} = \bigl(\frac{u}{l}-\frac{j\sigma}{\rho},v+j\bigr) \times \bigl(\frac{r}{l} -\frac{i\sigma}{\rho},s+i\bigl)$.
\end{proposition}

\begin{proof} Write
$$
\qquad P = \sum_{i=0}^{\alpha} \lambda_i X^{\frac{r}{l}-\frac{i\sigma} {\rho}} Y^{s+i} + R_P\quad\text{and}\quad Q = \sum_{j=0}^{\beta} \mu_j X^{\frac{u}{l}-\frac{j\sigma}{\rho}} Y^{v+j} + R_Q.
$$
Since $R_P = 0$ or $v_{\rho,\sigma}(R_P) < v_{\rho,\sigma}(P)$, and $R_Q = 0$ or $v_{\rho,\sigma}(R_Q) < v_{\rho,\sigma}(Q)$, from Re\-mark~\ref{re v de un conmutador} it follows that
\begin{equation}
[P,Q] = \sum_{i=0}^{\alpha} \sum_{j=0}^{\beta}\lambda_i \mu_j \left[X^{\frac{r}{l} -\frac{i\sigma}{\rho}} Y^{s+i}, X^{\frac{u}{l}-\frac{j\sigma}{\rho}} Y^{v+j} \right] + R,
\label{eq1}
\end{equation}
where $R = 0$ or $v_{\rho,\sigma}(R) < v_{\rho,\sigma}(P) + v_{\rho,\sigma}(Q) - (\rho + \sigma)$. Now, since $\rho+\sigma>0$ and by Lemma~\ref{le conmutacion en Wl},
$$
X^{\frac{i}{l}}Y^j X^{\frac{i'}{l}}Y^{j'} = \sum_{k=0}^j  k!\binom{j}{k} \binom{i'/l}{k} X^{\frac{i+i'}{l}-k}Y^{j+j'-k},
$$
we obtain that
\begin{equation}
\left[X^{\frac{r}{l} -\frac{i\sigma}{\rho}} Y^{s+i}, X^{\frac{u}{l}-\frac{j\sigma}{\rho}} Y^{v+j} \right] = c_{ij} X^{\frac{r+u}{l}-\frac{(i+j)\sigma} {\rho}-1} Y^{s+v+i+j-1} + R_{ij}, \label{eq2}
\end{equation}
with $R_{ij} = 0$ or $v_{\rho,\sigma}(R_{ij}) < v_{\rho,\sigma}(P) + v_{\rho,\sigma}(Q) - (\rho + \sigma)$. Combining~\eqref{eq1} with~\eqref{eq2}, we obtain that
$$
[P,Q]= \sum_{i=0}^{\alpha}\sum_{j=0}^{\beta} \lambda_i\mu_j c_{ij} X^{\frac{r+u}{l}-\frac{(i+j)\sigma} {\rho}-1} Y^{s+v+i+j-1} + R_{PQ},
$$
where $R_{PQ} = 0$ or $v_{\rho,\sigma}(R_{PQ}) < v_{\rho,\sigma}(P) + v_{\rho,\sigma}(Q) - (\rho + \sigma)$. Now, since
$$
v_{\rho,\sigma}\left(X^{\frac{r+u}{l}-\frac{(i+j)\sigma} {\rho}-1} Y^{s+v+i+j-1}\right) = v_{\rho,\sigma}(P) + v_{\rho,\sigma}(Q) - (\rho + \sigma),
$$
the result follows immediately.
\end{proof}

\begin{corollary}\label{ell depende del ell} Let $(\rho,\sigma)\in \mathfrak{V}$ and $P,Q,P_1,Q_1\in W^{(l)}\setminus\{0\}$. Assume that $\sigma\le 0$. If $\ell_{\rho,\sigma}(P) = \ell_{\rho,\sigma}(P_1)$ and $\ell_{\rho,\sigma}(Q) = \ell_{\rho,\sigma}(Q_1)$, then $[P,Q]_{\rho,\sigma}=[P_1,Q_1]_{\rho,\sigma}$.
\end{corollary}

\begin{proof} By Proposition~\ref{calculo del corchete}.
\end{proof}

\begin{corollary}\label{extremos alineados} Let $(\rho,\sigma)\in \mathfrak{V}$ and $P,Q\in W^{(l)}\setminus \{0\}$. If $[P,Q]_{\rho,\sigma}= 0$ and $\sigma\le 0$, then
$$
\st_{\rho,\sigma}(P)\sim \st_{\rho,\sigma}(Q)\quad\text{and}\quad \en_{\rho,\sigma}(P)\sim \en_{\rho,\sigma}(Q).
$$
\end{corollary}

\begin{proof} This follows immediately from Proposition~\ref{calculo del corchete}, since
$$
\st_{\rho,\sigma}(P)\times \st_{\rho,\sigma}(Q)=c_{00}\quad\text{and}\quad \en_{\rho,\sigma}(P) \times \en_{\rho,\sigma}(Q)=c_{\alpha\beta},
$$
where we are using the same notations as in the statement of that result.
\end{proof}

\begin{remark} Until now all definitions and notations we have introduced, when applied to $P\in W\subseteq W^{(1)}$, coincide with those given in~\cite{G-G-V2}. This is not the case with the following definition.
\end{remark}

\begin{definition}\label{polinomio asociado f^{(l)}} Given $P\in L^{(l)}\setminus\{0\}$ and $(\rho,\sigma)\in \mathfrak{V}$ with $\sigma\le 0$, we write
$$
f^{(l)}_{P,\rho,\sigma} := \sum_{i=0}^{\gamma} a_i x^i\in K[x],
$$
if
$$
\ell_{\rho,\sigma}(P) = \sum_{i=0}^{\gamma} a_i x^{\frac{r}{l} -\frac{i\sigma}{\rho}} y^{s+i}\quad\text{with $a_0\ne 0$ and $a_{\gamma}\ne 0$.}
$$
Now, for $P\in W^{(l)}$ we set $f^{(l)}_{P,\rho,\sigma}:=f^{(l)}_{\Psi^{(l)}(P),\rho,\sigma}$. Note that
\begin{align}
&\st_{\rho,\sigma}(P) = \Bigl(\frac{r}{l},s\Bigr),\qquad \en_{\rho,\sigma}(P) = \Bigl(\frac{r}{l}- \frac{\gamma\sigma}{\rho},s+\gamma\Bigr)\label{eq57}
\shortintertext{and}
&\ell_{\rho,\sigma}(P)= x^{\frac{r}{l}}y^s f^{(l)}_{P,\rho,\sigma}(x^{-\frac{\sigma} {\rho}}y).\label{eq58}
\end{align}
\end{definition}

\begin{remark}\label{f y f^{(l)}} Let $(\rho,\sigma) \in \mathfrak{V}$ with $\sigma\!\le\! 0$ and let $P\!\in W\setminus\{0\}$. Comparing Definition~\ref{polinomio asociado f^{(l)}} and~\cite[Def.~1.20]{G-G-V2}, we obtain that
$$
f^{(l)}_{P,\rho,\sigma}(x) = f_{P,\rho,\sigma}(x^{\rho}).
$$
The same formula is valid for $P\in L\setminus\{0\}$.
\end{remark}

\begin{remark}\label{f de un producto} Let $(\rho,\sigma)\in \mathfrak{V}$ with $\sigma\le 0$. From Proposition~\ref{pr v de un producto} it follows immediately that
\begin{equation*}
f^{(l)}_{PQ,\rho,\sigma}=f^{(l)}_{P,\rho,\sigma}f^{(l)}_{Q,\rho,\sigma}\quad\text{for $P,Q \in W^{(l)}\setminus\{0\}$.}
\end{equation*}
The same result holds for $P,Q\in L^{(l)}\setminus\{0\}$.
\end{remark}

\smallskip

Item~(2) of the following theorem justifies the terminology ``$(\rho,\sigma)$-proportional'' introduced in Definition~\ref{def rho-sigma proporcionales}.

\begin{theorem}\label{f[] en W^{(l)}} Let $P,Q\in W^{(l)}\setminus \{0\}$ and $(\rho,\sigma)\in \mathfrak{V}$ with $\sigma\le 0$. Set $a:=\frac{1} {\rho} v_{\rho,\sigma}(Q)$ and $b:=\frac{1}{\rho} v_{\rho,\sigma}(P)$.

\begin{enumerate}

\smallskip

\item If $[P,Q]_{\rho,\sigma}\ne 0$, then there exist $h\in \mathds{N}_0$ and $c\in \mathds{Z}$, such that
$$
x^h f_{[P,Q]} = cf_P f_Q+ ax f'_P f_Q-bxf'_Qf_P,
$$
where $f_P:=f^{(l)}_{P,\rho,\sigma}$, $f_Q:=f^{(l)}_{Q,\rho,\sigma}$ and $f_{[P,Q]}:=f^{(l)}_{[P,Q],\rho, \sigma}$.

\smallskip

\item If $[P,Q]_{\rho,\sigma}= 0$ and $a,b>0$, then there exist $\lambda_P,\lambda_Q\in K^{\times}$, $m,n\in \mathds{N}$ and a $(\rho,\sigma)$-homogeneous polynomial $R\!\in\! L^{(l)}$, with $\gcd(m,n)\! =\! 1$ and $m/n\! =\! b/a$, such that
$$
\ell_{\rho,\sigma}(P) = \lambda_P R^m\quad \text{and}\quad \ell_{\rho,\sigma}(Q) = \lambda_Q R^n.
$$

\end{enumerate}

\end{theorem}

\begin{proof} Write
$$
\qquad \ell_{\rho,\sigma}(P) = \sum_{i=0}^{\alpha} \lambda_i x^{\frac{r}{l}-\frac{i\sigma} {\rho}} y^{s+i}\quad\text{and}\quad \ell_{\rho,\sigma}(Q) = \sum_{j=0}^{\beta} \mu_j x^{\frac{u}{l}-\frac{j\sigma}{\rho}} y^{v+j},
$$
with $\lambda_0,\lambda_\alpha,\mu_0,\mu_\beta\ne 0$. By Proposition~\ref{calculo del corchete},
$$
[P,Q]_{\rho,\sigma} = \sum \lambda_i\mu_j c_{ij} x^{\frac{r+u}{l}-\frac{(i+j)\sigma} {\rho}-1} y^{s+v+i+j-1},
$$
where $c_{ij} := \bigl(\frac{u}{l}-\frac{j\sigma}{\rho},v+j\bigr)\times \bigl(\frac{r}{l}- \frac{i\sigma}{\rho},s+i\bigr)$. Set
$$
F(x):=\sum_{i,j} \lambda_i\mu_j c_{ij} x^{i+j}.
$$
Note that if $[P,Q]_{\rho,\sigma} = 0$, then $F=0$, and if $[P,Q]_{\rho,\sigma}\ne 0$, then $F = x^h f_{[P,Q]}$, where $h$ is the multiplicity of $x$ in $F$. Note that
$$
a=\left(\frac{u}{l},v\right)\times \left(-\frac{\sigma}{\rho},1\right)\quad\text{and}\quad b= -\left(-\frac{\sigma}{\rho},1\right)\times \left(\frac{r}{l},s\right).
$$
Let
$$
c:=\left(\frac{u}{l},v\right)\times \left(\frac{r}{l},s\right).
$$
Clearly $c_{ij} = c+ia-jb$. Since
$$
\sum_{i,j} \lambda_i\mu_j x^{i+j} = f_Pf_Q ,\quad \sum_{i,j} i\lambda_i\mu_j x^{i+j}=xf_P'f_Q \quad \text{and} \quad \sum_{i,j} j\lambda_i\mu_j x^{i+j} = xf_Q'f_P.
$$
we obtain
\begin{align}
F = cf_Pf_Q+axf_P'f_Q-bxf_Q'f_P.\label{eq8}
\end{align}
Item~(1) follows immediately from this fact. Assume now that $[P,Q]_{\rho,\sigma}= 0$ and that $a,b>0$. In this case $F = 0$ and, in particular, $c = c_{00} = \frac{F(0)}{\lambda_0\mu_0} = 0$. Hence,~\eqref{eq8} becomes
\begin{equation}
af_P'f_Q-bf_Q'f_P = 0.\label{eq9}
\end{equation}
Let $\bar{l}\in \mathds{N}$ be such that $\bar a:=\bar{l}a$ and $\bar b:=\bar{l}b$ are natural numbers. Since~\eqref{eq9} implies $(f_P^{\bar a}/f_Q^{\bar b})'=0$, there exists $\lambda\in K^\times$, such that $f_P^{\bar a} = \lambda f_Q^{\bar b}$. Hence, there are $g\in K[x]$ and $\lambda_P,\lambda_Q\in K^\times$, such that
\begin{equation}
f_P = \lambda_P g^m\quad\text{and}\quad f_Q=\lambda_Q g^n,\label{eq10}
\end{equation}
where $m: = \bar b/\gcd(\bar a,\bar b)$ and $n :=\bar a/\gcd(\bar a,\bar b)$. Now, note that $\left\{\left(s,-\frac{r}{l}\right),(\rho,\sigma)\right\}$ is a basis of $\mathds{Q}\times \mathds{Q}$ as a $\mathds{Q}$-vector space, since
$$
\left(s,-\frac{r}{l}\right)\times (\rho,\sigma) = \left(\frac{r}{l},s\right).(\rho,\sigma) = v_{\rho,\sigma}(P)= \rho b > 0,
$$
where the dot denotes the usual inner product. Hence, from
\begin{align*}
& \rho b \left(\frac{u}{l},v\right). \left(s,-\frac{r}{l}\right) =  \rho b \left(\frac{u}{l},v \right)\times \left(\frac{r}{l},s\right) = \rho b c = 0 = \rho a \left(\frac{r}{l},s \right). \left(s,-\frac{r}{l}\right)
\shortintertext{and}
& \rho b\left(\frac{u}{l},v\right).(\rho,\sigma) = v_{\rho,\sigma}(P)v_{\rho,\sigma}(Q) = v_{\rho,\sigma}(Q) v_{\rho,\sigma}(P) = \rho a \left(\frac{r}{l},s\right).(\rho,\sigma),
\end{align*}
it follows that $\bar b(u,l v) = \bar a (r,l s)$. Consequently $m(u,l v) = n(r,l s)$, and so there exists $(p,\bar q)\in \mathds{Z}\times \mathds{N}_0$, such that
\begin{equation*}
(u,l v) = n(p,\bar q)\quad\text{and}\quad (r,l s) = m(p,\bar q).
\end{equation*}
In particular $l|n \bar q$ and $l|m \bar q$, and so $l|\bar q$, since $m$ and $n$ are coprime. Hence,
\begin{equation}
\left(\frac{u}{l}, v\right) = n\left(\frac{p}{l},q\right)\quad\text{and}\quad \left(\frac{r}{l},s\right) = m \left(\frac{p}{l}, q\right),\label{eq11}
\end{equation}
where $q:=\bar q/l$. If $g = \sum_{i=0}^{\gamma} \nu_i x^i$ with $\nu_{\gamma}\ne 0$, then, by~\eqref{eq10} and~\eqref{eq11},
$$
R:=\sum_{i=0}^{\gamma} \nu_i x^{\frac{p}{l}-i\frac{\sigma}{\rho}} y^{q+i}
$$
fulfills
$$
\ell_{\rho,\sigma}(P)= x^{\frac{r}{l}}y^s f_P(x^{-\frac{\sigma}{\rho}}y) = \lambda_P \bigl(x^{\frac{p}{l}}y^q g(x^{-\frac{\sigma}{\rho}}y)\bigr)^m = \lambda_P R^m
$$
and
$$
\ell_{\rho,\sigma}(Q)= x^{\frac{u}{l}}y^v f_Q(x^{-\frac{\sigma}{\rho}}y) = \lambda_Q \bigl(x^{\frac{p}{l}}y^q g(x^{-\frac{\sigma}{\rho}}y)\bigr)^n = \lambda_Q R^n.
$$
In order to finish the proof it suffices to check that $R\in L^{(l)}$. First note that $R$ belongs to the field of fractions of $L^{(l)}$, because $R^m,R^n\in L^{(l)}$ and $\gcd(m,n) = 1$. But then $R\in L^{(l)}$, since $R^m\in L^{(l)}$.
\end{proof}

\begin{lemma}\label{ell del conmutador de una potencia} Let $C,E\in W^{(l)}\setminus\{0\}$, $m\in \mathds{N}$ and $(\rho,\sigma)\in \mathfrak{V}$. If $[C,E]\ne 0$, then
$$
[C^m,E]\ne 0\quad\text{and}\quad \ell_{\rho,\sigma}([C^m,E]) = m\ell_{\rho,\sigma}(C)^{m-1} \ell_{\rho,\sigma}([C,E]).
$$
Moreover $[C^m,E]_{\rho,\sigma}\ne 0$ if and only if $[C,E]_{\rho,\sigma}\ne 0$.
\end{lemma}

\begin{proof} Since
$$
[C^m,E] = \sum_{i=0}^{m-1} C^i[C,E]C^{m-i-1}
$$
and, by Proposition~\ref{pr v de un producto},
$$
\ell_{\rho,\sigma}(C^i[C,E]C^{m-i-1}) = \ell_{\rho,\sigma}(C)^{m-1}\ell_{\rho,\sigma}([C,E]),
$$
we have
$$
\ell_{\rho,\sigma}([C^m,E]) = m\ell_{\rho,\sigma}(C)^{m-1}\ell_{\rho,\sigma} ([C,E]),
$$
In order to prove the last assertion note that, again by Proposition~\ref{pr v de un producto},
$$
v_{\rho,\sigma}([C^m,E]) = v_{\rho,\sigma}(C^m) + v_{\rho,\sigma}([C,E]) - v_{\rho,\sigma}(C),
$$
and so
$$
v_{\rho,\sigma}([C^m,E]) = v_{\rho,\sigma}(C^m) + v_{\rho,\sigma}(E) - (\rho+\sigma)
$$
if and only if
$$
v_{\rho,\sigma}([C,E]) = v_{\rho,\sigma}(C) + v_{\rho,\sigma}(E) - (\rho+\sigma),
$$
which by Definition~\ref{def rho-sigma corchete}, means that $[C^m,E]_{\rho,\sigma}\ne 0 \Leftrightarrow [C,E]_{\rho,\sigma}\ne 0$.
\end{proof}

\begin{lemma}\label{Esta en Ll} Let $A,B\in L^{(l)}$ and $C\in L^{(l')}$, where $l|l'$. If $AC=B$, then $C\in L^{(l)}$.
\end{lemma}

\begin{proof} Consider the $K(y)$-polynomial algebras $K(y)[x^{\frac{1}{l}}]\subseteq K(y)[x^{\frac{1}{l'}}]$. By the division algorithm there exists $D,R\in K(y)[x^{\frac{1}{l}}]$ such that
$$
CA = B = DA+R \quad\text{and}\quad R=0\text{ or } \deg_l(R) < \deg_l(A),
$$
where $\deg_l$ denotes usual the degree in $x^{\frac{1}{l}}$. Let $\deg_{l'}$ be the degree in $x^{\frac{1}{l'}}$. Since $\deg_{l'} = \frac{l'}{l}\deg_{l'}$ the result follows from the uniqueness of the division algorithm.
\end{proof}

\begin{theorem}\label{te conmutadores1} Let $(\rho,\sigma)\!\in\! \mathfrak{V}$ with $\sigma\!\le\! 0$ and let $C,D\!\in\! W^{(l)}\setminus \{0\}$ with $v_{\rho,\sigma}(C)\!>\!0$. If
\begin{equation}
[C^k,D]_{\rho,\sigma}=\ell_{\rho,\sigma}(C^{k+j})\quad\text{for some $k\in \mathds{N}$ and $j\in \mathds{N}_0$,}\label{eq conmutador igual a potencia}
\end{equation}
then there exists a $(\rho,\sigma)$-homogeneous element $E\in W^{(l)}$, such that
$$
[C^t,E]_{\rho,\sigma}= t\ell_{\rho,\sigma}(C^t)\quad\text{for all $t\in \mathds{N}$.}
$$
\end{theorem}

\begin{proof} By equality~\eqref{eq conmutador igual a potencia} and items~(2) and~(3) of Proposition~\ref{pr v de un producto}, we have
$$
(k+j)v_{\rho,\sigma}(C)=kv_{\rho,\sigma}(C)+v_{\rho,\sigma}(D)-(\rho+\sigma)\quad\text{and}\quad \ell_{\rho,\sigma}([C^k,D]) = \ell_{\rho,\sigma}(C^{k+j}),
$$
and so,
\begin{equation}
v_{\rho,\sigma}(D)=jv_{\rho,\sigma}(C)+\rho+\sigma\quad\text{and}\quad f^{(l)}_{[C^k,D],\rho,\sigma} = f^{(l)}_{C^{k+j},\rho,\sigma}.\label{f1}
\end{equation}
Hence, by item~(1) of Theorem~\ref{f[] en W^{(l)}} and Remark~\ref{f de un producto}, there exist $h\!\in\! \mathds{N}_0$ and $c\!\in\! \mathds{Z}$, such that
$$
x^hf^{k+j} = cf^kg+ax(f^k)'g-bxf^kg',
$$
where
$$
f:=f^{(l)}_{C,\rho,\sigma},\quad g:=f^{(l)}_{D,\rho,\sigma},\quad a:=\frac{1}{\rho}v_{\rho,\sigma}(D)\quad\text{and}\quad b:=\frac{1}{\rho}v_{\rho,\sigma}(C^k)=\frac{k}{\rho}v_{\rho,\sigma}(C).
$$
Note that, since
$$
a = \frac{j}{k}b + \varepsilon,
$$
with $\varepsilon:=1 + \frac{\sigma}{\rho}$, the pair $(f,g)$ fulfills the condition $\PE(k,j,\varepsilon,b,c)$, introduced in~\cite[Def.~1.23]{G-G-V2}.  By~\cite[Prop.~1.24]{G-G-V2} there exists $\bar{g}\in K[x]$ such that $g=f^j\bar{g}$. Set $\alpha:=\deg f$, $\beta:=\deg g$ and write
$$
\ell_{\rho,\sigma}(C) = \sum_{i=0}^{\alpha}\lambda_i x^{\frac{r}{l}-i\frac{\sigma}{\rho}}y^{s+i}\quad\text{and} \quad \ell_{\rho,\sigma}(D) = \sum_{i=0}^{\beta} \mu_i x^{\frac{u}{l}-i\frac{\sigma}{\rho}} y^{v+i},
$$
with $\lambda_0,\lambda_{\alpha},\mu_0,\mu_{\beta}\ne 0$. By item~(2) of Proposition~\ref{pr v de un producto},
$$
\ell_{\rho,\sigma}(C^k) = \ell_{\rho,\sigma}(C)^k  = \left(\sum_{i=0}^{\alpha}\lambda_i x^{\frac{r}{l}-i\frac{\sigma}{\rho}}y^{s+i}\right)^k = \sum_{i=0}^{k\alpha} \bar{\lambda}_i x^{k\frac{r}{l}-i\frac{\sigma}{\rho}} y^{ks+i}
$$
with each $\bar{\lambda}_i\in K$. Let $\gamma:=\beta-j\alpha$ be the degree of $\ov g$ and write $\bar{g} = \sum_{i=0}^\gamma \eta_i x^i$. Note that $\eta_0\ne 0$, since $f^j(0)\eta_0 =f^j(0)\bar{g}(0) = g(0)\ne 0$. We define
$$
E:=k \sum_{i=0}^\gamma \eta_i X^{\frac{u}{l}-j\frac{r}{l}-i\frac{\sigma}{\rho}}Y^{v-js+i}.
$$
We claim that $E\in W^{(l')}$, where $l':=\lcm(l,\rho)$. For this it suffices to check that
$$
 v-js\ge 0.
$$
We consider the two cases
$$
\st_{\rho,\sigma}(D)\sim\st_{\rho,\sigma}(C)\quad\text{and}\quad \st_{\rho,\sigma}(D)\nsim \st_{\rho,\sigma}(C).
$$
Note that $(\frac{r}{l},s)=\st_{\rho,\sigma}(C)\ne (0,0)$, since $v_{\rho,\sigma}(\st_{\rho,\sigma}(C)) = v_{\rho,\sigma}(C) >0$. Hence, if $\st_{\rho,\sigma}(D)\sim\st_{\rho,\sigma}(C)$, then there exists $\lambda\in K$ such that
$$
\left(\frac{u}{l},v\right)=\st_{\rho,\sigma}(D)=\lambda \st_{\rho,\sigma}(C)=\lambda \left(\frac{r}{l},s\right).
$$
Consequently, by the first equality in~\eqref{f1},
$$
jv_{\rho,\sigma}(C)+\rho+\sigma=v_{\rho,\sigma}(D)=v_{\rho,\sigma}(\st_{\rho,\sigma}(D)) =\lambda v_{\rho,\sigma}(\st_{\rho,\sigma}(C)) =\lambda v_{\rho,\sigma}(C),
$$
which implies $\lambda> j$, since $v_{\rho,\sigma}(C)>0$ and $\rho+\sigma>0$. But then $v-js=(\lambda-j)s\ge 0$ as we want. Assume now that $\st_{\rho,\sigma}(D)\nsim \st_{\rho,\sigma}(C)$. Then
$$
\st_{\rho,\sigma}(D)\nsim \st_{\rho,\sigma}(C^k),
$$
since $\st_{\rho,\sigma}(C^k) = k\st_{\rho,\sigma}(C)$ by item~(4) of Proposition~\ref{pr v de un producto}. Hence, equality~\eqref{eq conmutador igual a potencia}, Proposition~\ref{extremosnoalineados} and item~(4) of Proposition~\ref{pr v de un producto} yields
$$
\st_{\rho,\sigma}(D)+k\st_{\rho,\sigma}(C)-(1,1)=(k+j)\st_{\rho,\sigma}(C),
$$
which implies $\left(\frac{u}{l},v\right)=j\left(\frac{r}{l},s\right)+(1,1)$, and so $v-js=1>0$, which ends the proof of the claim.

\smallskip

Now, since $\eta_0\ne 0$ and $\eta_{\gamma}\ne 0$, we have $\bar{g}=f^{(l)}_{\frac{1}{k}E,\rho,\sigma}$. Thus, by Proposition~\ref{pr v de un producto} and equality~\eqref{eq58},
\begin{align*}
\ell_{\rho,\sigma}\left(\frac{1}{k} EC^j\right)& =\ell_{\rho,\sigma}\left(\frac{1}{k} E\right) \ell_{\rho,\sigma}(C)^j\\
& = x^{\frac{u-jr}{l}}y^{v-js}\ov g(x^{-\frac{\sigma}{\rho}} y) \left(x^\frac{r}{l} y^s f(x^{-\frac{\sigma}{\rho}}y)\right)^j\\
& = x^{\frac{u}{l}} y^v g(x^{-\frac{\sigma}{\rho}} y)\\
& = \ell_{\rho,\sigma}(D).
\end{align*}
Consequently, by Lemma~\ref{Esta en Ll}, $\ell_{\rho,\sigma}(E) \in L^{(l)}$ and, since $\Psi^{(l)}(E) = \ell_{\rho,\sigma}(E)$, we have $E\!\in\! W^{(l)}$. Moreover, by Corollary~\ref{ell depende del ell}, equality~\eqref{eq conmutador igual a potencia} and item~(2) of Proposition~\ref{pr v de un producto},
\begin{equation}
\left[C^k,\frac{1}{k} EC^j\right]_{\rho,\sigma} = [C^k,D]_{\rho,\sigma} = \ell_{\rho,\sigma}(C^{k+j}) = \ell_{\rho,\sigma}(C)^{k+j}.\label{eq15}
\end{equation}
Hence $\left[C^k,\frac{1}{k} EC^j\right]_{\rho,\sigma}\ne 0$, and so, by items (2) and~(3) of Proposition~\ref{pr v de un producto},
\begin{equation}
\begin{aligned}
v_{\rho,\sigma}\left(\left[C^k,\frac{1}{k} EC^j\right]\right) & = v_{\rho,\sigma}(C^k) + v_{\rho,\sigma} \left(\frac{1}{k} EC^j\right) - (\rho+\sigma)\\
& = (k+j)v_{\rho,\sigma}(C) + v_{\rho,\sigma}(E) - (\rho+\sigma)
\end{aligned}\label{eq16}
\end{equation}
and
\begin{equation}
\begin{aligned}
\left[C^k,\frac{1}{k} EC^j\right]_{\rho,\sigma} & = \ell_{\rho,\sigma}\left(\left[C^k, \frac{1}{k} EC^j\right]\right)\\
& = \ell_{\rho,\sigma}\left(\left[C^k, \frac{1}{k} E\right]C^j\right)\\
& =\frac{1}{k} \ell_{\rho,\sigma}([C^k,E])\ell_{\rho,\sigma}(C)^j\\
& = \ell_{\rho,\sigma}([C,E])\ell_{\rho,\sigma}(C)^{k+j-1},
\end{aligned}\label{eq17}
\end{equation}
where the last equality follows from Lemma~\ref{ell del conmutador de una potencia}. Hence, again by Pro\-position 2.7
\begin{equation}
v_{\rho,\sigma}\left(\left[C^k,\frac{1}{k} EC^j\right]\right) = (k+j-1)v_{\rho,\sigma}(C) + v_{\rho,\sigma}([C,E]).\label{eq18}
\end{equation}
Combining now~\eqref{eq16} with~\eqref{eq18}, and~\eqref{eq15} with~\eqref{eq17}, we obtain
\begin{equation*}
v_{\rho,\sigma}([C,E]) = v_{\rho,\sigma}(C) + v_{\rho,\sigma}(E) - (\rho+\sigma) \quad \text{and}\quad \ell_{\rho,\sigma}([C,E]) = \ell_{\rho,\sigma}(C).
\end{equation*}
Hence $[C,E]_{\rho,\sigma} = \ell_{\rho,\sigma}(C)\ne 0$, and thus, by Lemma~\ref{ell del conmutador de una potencia} and item (2) of Proposition~\ref{pr v de un producto}, we have
$$
[C^t,E]_{\rho,\sigma} = \ell_{\rho,\sigma}([C^t,E]) = t\ell_{\rho,\sigma}(C)^{t-1}\ell_{\rho,\sigma}([C,E]) = t\ell_{\rho,\sigma}(C)^t = t\ell_{\rho,\sigma}(C^t),
$$
for all $t\in \mathds{N}$.
\end{proof}

Recall from~\cite{G-G-V2} that $\ov{\mathfrak{V}}$ is endowed with an order relation such that
$$
(1,-1)<(\rho,\sigma)<(-1,1)
$$
for all $(\rho,\sigma)\in \mathfrak{V}$ and that
$$
(\rho_1,\sigma_1) \le (\rho,\sigma)\Leftrightarrow (\rho_1,\sigma_1)\times (\rho,\sigma)\ge 0\quad\text{for all $(\rho_1,\sigma_1),(\rho,\sigma)\in \mathfrak{V}$.}
$$

\begin{definition}\label{forma debil} Let $P\in W^{(l)}\setminus \{0\}$. We define the set of {\em valuations associated with} $P$ as
$$
\Val(P):=\{(\rho,\sigma)\in\mathfrak{V}:\# \Supp(\ell_{\rho,\sigma}(P))>1\},
$$
and we set $\ov{\Val}(P):= \Val(P)\cup\{(1,-1),(-1,1)\}$. We make a similar definition for $P\in L^{(l)} \setminus \{0\}$.
\end{definition}

For each $(r/l,s)\in \frac{1}{l}\mathds{Z}\times \mathds{Z} \setminus \mathds{Z}(1,1)$ there exists a unique $(\rho,\sigma)\in \mathfrak{V}$ such that $v_{\rho,\sigma}(r/l,s)=0$. In fact clearly
$$
(\rho,\sigma) := \begin{cases}\left(-\frac{ls}{d},\frac{r}{d}\right)& \text{ if $r-ls>0$,}\\ \left(\frac{ls}{d},-\frac{r}{d}\right)& \text{ if $r-ls<0$,}\end{cases}
$$
where $d:=\gcd(r,ls)$, fulfill the required condition, and the uniqueness is evident.

\begin{definition} For $(r/l,s)\in\frac{1}{l}\mathds{Z}\times\mathds{Z}\setminus \mathds{Z}(1,1)$, we define $\val(r/l,s)$ to be the unique $(\rho,\sigma)\in \mathfrak{V}$ such that $v_{\rho,\sigma}(r/l,s)=0$.
\end{definition}

\begin{remark} Note that if $P\in W^{(l)}\setminus\{0\}$ and $(\rho,\sigma)\in \Val(P)$, then
$$
(\rho,\sigma) = \val\bigl(\en_{\rho,\sigma}(P) - \st_{\rho,\sigma}(P)\bigr).
$$
\end{remark}

Our aim is to prove Proposition~\ref{le basico}, and therefore we fix $P\in W^{(l)}\setminus \{0\}$ and $(\rho,\sigma)\in \mathfrak{V}$. We set $\en:=\en_{\rho,\sigma}(P)$ and $\st:= \st_{\rho,\sigma}(P)$ and we consider the following two sets of valuations
$$
\Valsup(\rho,\sigma)\!:=\!\left\{\val\!\left(\!\left(\frac{i}{l},j\right)\!-\!\en\!\right): \left(\!\frac{i}{l},j\!\right)\!\!\in\! \Supp(P)\text{ and } v_{-1,1}\!\left(\frac{i}{l},j\right) \!> \!v_{-1,1}(\en)\right\}
$$
and
$$
\Valinf(\rho,\sigma)\!:=\!\left\{\val\left(\!\left(\frac{i}{l},j\right)\!-\!\st\!\right) : \left(\!\frac{i}{l},j \!\right)\!\!\in\! \Supp(P) \text{ and } v_{1,-1}\!\left(\frac{i}{l},j\right)\!>\! v_{1,-1}(\st)\right\}.
$$

\begin{lemma}\label{separacion} The following assertions hold:

\begin{enumerate}

\smallskip

\item If $(\rho_1,\sigma_1)\in \Valsup(\rho,\sigma)$, then $(\rho_1,\sigma_1)>(\rho,\sigma)$.

\smallskip

\item If $(\rho_1,\sigma_1)\in \Valinf(\rho,\sigma)$, then $(\rho_1,\sigma_1)<(\rho,\sigma)$.

\end{enumerate}

\end{lemma}

\begin{proof} We only prove item~(1) and leave the other one to the reader. Clearly, if
$$
(i/l,j)\in \Supp(P)\quad \text{and}\quad v_{\rho,\sigma}(i/l,j) = v_{\rho,\sigma}(P),
$$
then $(i/l,j)\in \Supp\bigl(\ell_{\rho,\sigma}(P)\bigr)$, and so $v_{-1,1}(i/l,j)\le v_{-1,1}(\en)$. Consequently, if
$$
(i/l,j)\in \Supp(P)\quad \text{and}\quad v_{-1,1}(i/l,j) > v_{-1,1}(\en),
$$
then $v_{\rho,\sigma}(i/l,j) < v_{\rho,\sigma}(P) = v_{\rho,\sigma}(\en)$. This means
\begin{equation}
v_{\rho,\sigma}(a,b)<0, \label{eq ij menor que P}
\end{equation}
where $(a,b) := (i/l,j)-\en$. Note that $v_{-1,1}(i/l,j)> v_{-1,1}(\en)$ now reads
$$
b-a = v_{-1,1}(a,b)>0.
$$
But then
$$
(\rho_1,\sigma_1):=\val\bigl((i/l,j)-\en\bigr) = \val(a,b)=\lambda (b,-a),
$$
for some $\lambda>0$. Hence
\begin{align*}
0& > v_{\rho,\sigma}(a,b)\\
&= a\rho+b\sigma\\
&= -\frac {1}{\lambda}(\sigma_1 \rho-\rho_1 \sigma)\\
&= - \frac {1}{\lambda}(\rho,\sigma)\times (\rho_1,\sigma_1).
\end{align*}
This yields $(\rho,\sigma)\times(\rho_1,\sigma_1)>0$, and so $(\rho_1,\sigma_1)>(\rho,\sigma)$, as desired.
\end{proof}

\begin{lemma}\label{le recorte de dominio} Let $P$, $(\rho,\sigma)$, $\st$ and $\en$ be as before. We have:
\begin{enumerate}

\smallskip

\item If $(i/l,j)\in \Supp(P)$, $(\rho_1,\sigma_1)>(\rho,\sigma)$ and $v_{-1,1}(i/l,j)\le v_{-1,1}(\en)$, then
\begin{equation}
v_{\rho_1,\sigma_1}(i/l,j)\le v_{\rho_1,\sigma_1}(\en).\label{eqnu1}
\end{equation}
Moreover, if $(\rho_1,\sigma_1)\ne (-1,1)$, then equality holds if and only if $\left(\frac{i}{l},j\right)=\en$.

\smallskip

\item If $(i/l,j)\in \Supp(P)$, $(\rho_1,\sigma_1)<(\rho,\sigma)$ and $v_{1,-1}(i/l,j)\le v_{1,-1}(\st)$, then
$$
v_{\rho_1,\sigma_1}(i/l,j)\le v_{\rho_1,\sigma_1}(\st).
$$
Moreover, if $(\rho_1,\sigma_1)\ne (1,-1)$, then equality holds if and only if $\left(\frac{i}{l},j\right)=\st$.
\end{enumerate}
\end{lemma}

\begin{proof} We prove item~(1) and leave the proof of item~(2), which is similar, to the reader. Set $(a,b):=\left(i/l,j\right)-\en$. Then, by the hypothesis,
$$
\rho \sigma_1-\sigma \rho_1 >0\quad \text{and}\quad b-a\le 0.
$$
Hence
\begin{equation}
b\rho\sigma_1+\sigma \rho_1 a-a\rho \sigma_1-b\sigma \rho_1 \le 0,\label{eq condiciones del lemma}
\end{equation}
and the equality holds if and only if $b\!=\!a$. We also know that $v_{\rho,\sigma} (i/l,j)\le v_{\rho,\sigma}(\en)$, which means that $\rho a+\sigma b \le 0$. Since $\rho_1+\sigma_1\ge 0$, we obtain
\begin{equation}
\rho_1\rho a+\sigma_1\sigma b+\rho_1 \sigma b+\sigma_1\rho a =(\rho a+\sigma b)(\rho_1+\sigma_1) \le 0.\label{eq28}
\end{equation}
Summing up~\eqref{eq condiciones del lemma} and~\eqref{eq28}, we obtain
$$
0\ge \rho\rho_1 a+\sigma\sigma_1 b+\rho\sigma_1 b+\sigma\rho_1 a=(\rho+\sigma)(\rho_1 a+ \sigma_1 b),
$$
and so $v_{\rho_1,\sigma_1}(a,b)\le 0$, as desired. Moreover, if the equality is true, then~\eqref{eq condiciones del lemma} is also an equality, and so $b=a$. Hence $0 = v_{\rho_1,\sigma_1}(a,a) = (\rho_1+\sigma_1) a$, which implies that $a=0$ or $(\rho_1,\sigma_1) = (-1,1)$. Thus, is $(\rho_1,\sigma_1)\ne (-1,1)$ and equality holds in~\eqref{eqnu1}, then $(i/l,j) = \en$.
\end{proof}

\begin{definition} If $\Valsup(\rho,\sigma)\ne \emptyset$, then we define
$$
\Succ(\rho,\sigma):=\min \Valsup(\rho,\sigma)
$$
and if $\Valinf(\rho,\sigma)\ne \emptyset $, then we define
$$
\Pred(\rho,\sigma):=\max \Valinf(\rho,\sigma).
$$
\end{definition}

\begin{lemma}\label{predysucc} The following assertions hold:
\begin{enumerate}

\smallskip

\item $\Succ(\rho,\sigma)\in \Val(P)$ and $\en=\st_{\Succ(\rho,\sigma)}(P)$.

\smallskip

\item $\Pred(\rho,\sigma)\in \Val(P)$ and $\st=\en_{\Pred(\rho,\sigma)}(P)$.
\end{enumerate}
\end{lemma}

\begin{proof} We only prove~(1) since~(2) is similar. As above we set $\en:=\en_{\rho,\sigma}(P)$. Write $(\rho_1,\sigma_1):= \Succ(\rho,\sigma)$. By defi\-nition, there exists an $(i_0/l,j_0)\in \Supp(P)$, such that
$$
v_{-1,1}(i_0/l,j_0)>v_{-1,1}(\en)\quad\text{and}\quad (\rho_1,\sigma_1)= \val\bigl((i_0/l,j_0)-\en\bigr).
$$
Consequently,
\begin{equation}
(i_0/l,j_0)\ne \en\quad\text{and}\quad v_{\rho_1,\sigma_1}(\en) = v_{\rho_1,\sigma_1}(i_0/l,j_0).\label{eq29}
\end{equation}
Hence $(\rho_1,\sigma_1)\ne(-1,1)$, since, on the contrary, $v_{\rho_1,\sigma_1}(\en)< v_{\rho_1,\sigma_1}(i_0/l,j_0)$.  We claim that $v_{\rho_1,\sigma_1}(P) = v_{\rho_1,\sigma_1}(\en)$, which, by~\eqref{eq29}, proves that $(\rho_1,\sigma_1) \in \Val(P)$. In fact, assume on the contrary that there exists $(i/l,j)\in \Supp(P)$ with
\begin{equation}\label{ineq}
v_{\rho_1,\sigma_1}(i/l,j)>v_{\rho_1,\sigma_1}(\en)
\end{equation}
By item~(1) of Lemmas~\ref{separacion} and~\ref{le recorte de dominio}, necessarily $v_{-1,1}(i/l,j)>v_{-1,1}(\en)$, and so $(a,b):=(i/l,j)-\en$ fulfills $b-a>0$. Hence
$$
(\rho_2,\sigma_2):=\val\bigl((i/l,j)-\en\bigr) = \val(a,b)= \lambda (b,-a)
$$
with $\lambda>0$. Now~\eqref{ineq} leads to
\begin{align*}
0& < (\rho_1,\sigma_1).(a,b)\\
&= \frac{1}{\lambda}(\rho_2\sigma_1-\sigma_2 \rho_1)\\
&= \frac{1}{\lambda}(\rho_2,\sigma_2)\times(\rho_1,\sigma_1),
\end{align*}
which implies that $(\rho_2,\sigma_2)<(\rho_1,\sigma_1)$. But this fact is impossible, since $(\rho_1,\sigma_1)$ is minimal in $\Valsup(\rho,\sigma)$ and $(\rho_2,\sigma_2)\in \Valsup(\rho,\sigma)$. This proves the claim and so $\Succ(\rho,\sigma)\in \Val(P)$.

Finally we will check that $\en=\st_{\rho_1,\sigma_1}(P)$. For this, it suffices to prove that any $(i/l,j)\in \Supp(\ell_{\rho_1,\sigma_1}(P))$ fulfills $v_{1,-1}(i/l,j)\le v_{1,-1}(\en)$ or, equivalently, that $v_{-1,1}(i/l,j) \ge v_{-1,1}(\en)$. To do this we first note that by item~(1) of Lemma~\ref{separacion} we have $(\rho_1,\sigma_1) > (\rho,\sigma)$. Since, moreover $(i/l,j)\in \Supp(P)$ and $(\rho_1,\sigma_1)\ne (-1,1)$, using item~(1) of Lemma~\ref{le recorte de dominio}, it follows that if $v_{-1,1}(i/l,j)< v_{-1,1}(\en)$, then $v_{\rho_1,\sigma_1}(i/l,j)< v_{\rho_1,\sigma_1}(\en)$, which is a contradiction.
\end{proof}

\begin{proposition}\label{le basico} Let $P\in W^{(l)}\setminus\{0\}$ and let $(\rho_1,\sigma_1)> (\rho_2,\sigma_2)$ consecutive elements in $\ov{\Val}(P)$. The following assertions hold:

\begin{enumerate}

\smallskip

\item If $(\rho_1,\sigma_1)\in \Val(P)$ and $(\rho_1,\sigma_1)> (\rho,\sigma) \ge (\rho_2,\sigma_2)$, then
$$
(\rho_1,\sigma_1) = \Succ_{\rho,\sigma}(P).
$$

\smallskip

\item If $(\rho_2,\sigma_2)\in \Val(P)$ and $(\rho_1,\sigma_1)\ge (\rho,\sigma) > (\rho_2,\sigma_2)$, then
$$
(\rho_2,\sigma_2) = \Pred_{\rho,\sigma}(P).
$$

\smallskip

\item If $(\rho_1,\sigma_1)> (\rho,\sigma) >(\rho_2,\sigma_2)$, then
$$
\{\st_{\rho_1,\sigma_1}(P)\} = \Supp(\ell_{\rho,\sigma}(P)) = \{\en_{\rho_2,\sigma_2}(P)\}.
$$

\end{enumerate}

\end{proposition}

\begin{proof} (1)\enspace By item~(1) of Lemmas~\ref{separacion} and~\ref{predysucc} it is clear that the existence of $\Succ(\rho,\sigma)$ implies
$$
(\rho,\sigma)<\Succ(\rho,\sigma)\quad\text{and}\quad \Succ(\rho,\sigma)\in \Val(P).
$$
Hence $(\rho_1,\sigma_1)\le \Succ(\rho,\sigma)$. So we must prove that $\Succ(\rho,\sigma)$ exists and that $(\rho_1,\sigma_1)\ge \Succ(\rho,\sigma)$. For the existence it suffices to prove that $\Valsup(\rho,\sigma)\ne \emptyset$. Assume on the contrary that $\Valsup_{\rho,\sigma}(P) = \emptyset$. Then, by definition
$$
v_{-1,1}(i=l,j)\le v_{-1,1}(\en_{\rho,\sigma})(P)\quad\text{for all $(i/l,j)\in \Supp(P)$.}
$$
Consequently, since $(\rho,\sigma)<(\rho_1,\sigma_1)< (-1,1)$, by item~(1) of Lemma~\ref{le recorte de dominio},
$$
\Supp(\ell_{\rho_1,\sigma_1}(P)) = \{\en_{\rho,\sigma}(P)\},
$$
and so $(\rho_1,\sigma_1)\notin \Val(P)$, which is a contradiction.

Now we prove that $(\rho_1,\sigma_1)\ge \Succ(\rho,\sigma)$. Since $(\rho_1,\sigma_1)$ is the minimal element of $\Val(P)$ greater than $(\rho,\sigma)$, it suffices to prove that there exists no $(\rho_3,\sigma_3)\in \Val(P)$ such that $\Succ(\rho,\sigma)> (\rho_3,\sigma_3)>(\rho,\sigma)$. In other words that
\begin{equation*}
\Succ_{\rho,\sigma}(P)> (\rho_3,\sigma_3)>(\rho,\sigma) \Longrightarrow (\rho_3,\sigma_3)\notin \Val(P).
\end{equation*}
So let us assume $\Succ(\rho,\sigma)> (\rho_3,\sigma_3) > (\rho,\sigma)$ and let $(i/l,j)\in \Supp(\ell_{\rho_3,\sigma_3}(P))$. We assert that  $(i/l,j)=\en_{\rho,\sigma}(P)$, which shows that $\Supp(\ell_{\rho_3,\sigma_3}(P))=\{\en_{\rho,\sigma}(P)\}$, and consequently that $(\rho_3,\sigma_3)\notin \Val(P)$. In fact, if $v_{-1,1}(i/l,j)\le v_{-1,1}(\en_{\rho,\sigma}(P))$, this follows from item~(1) of Lemma~\ref{le recorte de dominio}, applied to $(\rho_3,\sigma_3)$ instead of $(\rho_1,\sigma_1)$. Assume now that $v_{-1,1}(i/l,j) \ge v_{-1,1}(\en_{\rho,\sigma}(P))$. Since, by item~(1) of Lemma~\ref{predysucc}, we know that $\st_{\Succ(\rho,\sigma)}(P) = \en_{\rho,\sigma}(P)$, we have
$$
v_{1,-1}(i/l,j) \le v_{1,-1}(\en) = v_{1,-1}\bigl(\st_{\Succ(\rho,\sigma)}(P)\bigr).
$$
Hence, applying item~(2) of Lemma~\ref{le recorte de dominio}, with $\Succ(\rho,\sigma)$ instead of $(\rho,\sigma)$ and $(\rho_3,\sigma_3)$ instead of $(\rho_1,\sigma_1)$, and taking into account that $(i/l,j)\in \Supp(\ell_{\rho_3,\sigma_3}(P))$, we ob\-tain
$$
v_{\rho_3,\sigma_3}(i/l,j) = v_{\rho_3,\sigma_3}\bigl(\st_{\Succ(\rho,\sigma)}(P)\bigr).
$$
Consequently, since $(\rho_3,\sigma_3) \ne (1,-1)$, it follows, again by item~(2) of Lemma~\ref{le recorte de dominio}, that $(i/l,j) = \st_{\Succ(\rho,\sigma)}(P) = \en_{\rho,\sigma}(P)$, which proves the assertion.

\smallskip

\noindent (2)\enspace It is similar to the proof of item~(1).

\smallskip

\noindent (3)\enspace We first prove that if $(\rho_1,\sigma_1)\in \Val(P)$, then
$$
\{\st_{\rho_1,\sigma_1}(P)\} = \Supp(\ell_{\rho,\sigma}(P)).
$$
Since $\{\en_{\rho,\sigma}(P)\} = \Supp(\ell_{\rho,\sigma}(P))$, this fact follows from item~(1) and item~(1) of Lemma~\ref{predysucc}. This conclude the proof of the first equality in the statement when $(\rho_1,\sigma_1)\in (-1,1)$. Now, a symmetric argument shows that if $(\rho_2,\sigma_2)>(1,-1)$, then
$$
\{\en_{\rho_2,\sigma_2}(P)\} = \Supp(\ell_{\rho,\sigma}(P)).
$$
Assume now that $(\rho_1,\sigma_1)=(-1,1)$ and $(\rho_2,\sigma_2)\ne (1,-1)$. Then, by item~(1) of Lemmas~\ref{separacion} and~\ref{predysucc}, we have $\Valsup(\rho_2,\sigma_2)= \emptyset$. Hence
$$
v_{-1,1}(i/l,j)\le v_{-1,1}(\en_{\rho_2,\sigma_2}(P)),
$$
for all $(i/l,j)\in \Supp(P)$. Consequently, $\en_{\rho_2,\sigma_2}(P)\in \Supp\bigl(\ell_{-1,1} (P)\bigr)$, and so
$$
\st_{-1,1}(P) = \en_{\rho_2,\sigma_2}(P)+(a,a),
$$
for some $a\ge 0$. But necessarily $a=0$, since $a>0$ leads to the contradiction
$$
v_{\rho_2,\sigma_2}(\st_{-1,1}(P))= v_{\rho_2,\sigma_2}(\en_{\rho_2,\sigma_2}(P)+(a,a))= v_{\rho_2,\sigma_2}(P)+ a(\rho_2+\sigma_2).
$$
Thus
$$
\{\st_{\rho_1,\sigma_1}(P)\} = \{\en_{\rho_2,\sigma_2}(P)\} = \Supp(\ell_{\rho,\sigma}(P)).
$$
Similarly, if $(\rho_1,\sigma_1)\ne (-1,1)$ and $(\rho_2,\sigma_2)= (1,-1)$, then
$$
\{\st_{\rho_1,\sigma_1}(P)\} = \{\en_{\rho_2,\sigma_2}(P)\} = \Supp(\ell_{\rho,\sigma}(P)).
$$
Finally we assume that $(\rho_1,\sigma_1) = (-1,1)$ and $(\rho_2,\sigma_2) = (1,-1)$. Since $\Val(P) = \emptyset$, it follows from  Lemma~\ref{predysucc} that
$$
\Valsup(\rho,\sigma) = \emptyset = \Valinf(\rho,\sigma).
$$
Hence
\begin{equation}
v_{-1,1}(P) = v_{-1,1}(\en_{\rho,\sigma}(P))\quad\text{and}\quad v_{1,-1}(P) = v_{1,-1}(\st_{\rho,\sigma}(P)). \label{eq30}
\end{equation}
But, since $\en_{\rho,\sigma}(P)=\st_{\rho,\sigma}(P)$, it follows easily from~\eqref{eq30} that $P = \ell_{-1,1}(P)$, and so,
$$
\{\en_{1,-1}(P)\}=\{w(P)\}=\{\ov{w}(P)\}=\{\st_{-1,1}(P)\} = \Supp\bigl(\ell_{\rho,\sigma}(P)\bigr),
$$
as desired.
\end{proof}

\begin{proposition}\label{le basico1} Let $P\in W^{(l)}\setminus\{0\}$ and let $(\rho',\sigma')\in Val(P)$. We have:

\begin{enumerate}

\smallskip

\item If $(\rho,\sigma)\in \ov{\mathfrak{B}}$ satisfy $\sigma <0$ and $(\rho',\sigma')<(\rho,\sigma)$, then
$$
v_{\rho,\sigma}\bigl(\st_{\rho',\sigma'}(P)\bigr) < v_{\rho,\sigma}\bigl(\en_{\rho',\sigma'}(P)\bigr),
$$

\smallskip

\item If $(\rho,\sigma)\in \ov{\mathfrak{B}}$ satisfy $(\rho,\sigma)<(\rho',\sigma')$, then
$$
v_{\rho,\sigma}\bigl(\st_{\rho',\sigma'}(P)\bigr) > v_{\rho,\sigma}\bigl(\en_{\rho',\sigma'}(P)\bigr).
$$

\smallskip

\end{enumerate}
The same properties hold for $P\in L^{(l)}\setminus\{0\}$.
\end{proposition}

\begin{proof} We prove item~(1) and leave the proof of item~(2), which is similar, to the reader. By definition
$$
(\rho',\sigma')<(\rho,\sigma)\Longleftrightarrow \rho'\sigma - \sigma'\rho >0 \Longleftrightarrow  \sigma-\frac{\rho\sigma'}{\rho'}>0,
$$
where the last equivalence follows from the fact that $\rho'>0$. Now, set
$$
\st_{\rho',\sigma'}(P) = \Bigl(\frac{r}{l},s\Bigr)\quad\text{and}\quad \en_{\rho',\sigma'}(P) = \Bigl(\frac{r}{l}- \frac{\gamma\sigma'}{\rho'},s+\gamma\Bigr)
$$
Since $(\rho',\sigma')\in Val(P)$ we have $\gamma\in \mathds{N}$. Consequently,
$$
v_{\rho,\sigma}\bigl(\en_{\rho',\sigma'}(P)\bigr) - v_{\rho,\sigma}\bigl(\st_{\rho',\sigma'}(P)\bigr) = \sigma\gamma -\frac{\rho\sigma'\gamma }{\rho'}>0,
$$
as desired.
\end{proof}

\section{Fixed points of $(\rho,\sigma)$-brackets}

\setcounter{equation}{0}

The aim of this section is to construct $F\in W$ such that $[F,P]_{\rho,\sigma} = \ell_{\rho,\sigma}(P)$ and $[F,P]_{\rho,\sigma} = \ell_{\rho,\sigma}(P)$ for suitable pairs $(P,Q)$ and some given $(\rho,\sigma)\in \mathfrak{V}$.

\begin{theorem}\label{tecnico1} Let $C\!\in\! W^{(l)}$ and let $(\rho,\sigma)\!\in\! \mathfrak{V}$ with $\sigma\le 0$. Suppose that
$$
v_{\rho,\sigma}(C)> 0\quad\text{and}\quad \ell_{\rho,\sigma}(C) \ne \zeta \ell_{\rho,\sigma}(D^h),\quad\! \text{for all $D\in W^{(l)}$, $\zeta\in K^{\times}$ and $h>1$.} $$
If there exist $n,m\in \mathds{N}$ and $A,B\in W^{(l)}$ such that
\begin{enumerate}

\smallskip

\item $\ell_{\rho,\sigma}(A)=\ell_{\rho,\sigma}(C^n)$,

\smallskip

\item $\ell_{\rho,\sigma}(B)=\ell_{\rho,\sigma}(C^m)$,

\smallskip

\item $c := \gcd(n,m)\notin \{n,m\}$,

\smallskip

\item $\ell_{\rho,\sigma}([A,B]) = \lambda \ell_{\rho,\sigma}(C^h)$, for some $h\in \mathds{N}_0$ and $\lambda\in K^{\times}$,

\smallskip

\end{enumerate}
then there exist $D\in W^{(l)}$, $\mu\in K^{\times}$ and $k,j_0\in \mathds{N}$, such that
$$
[D,C^k]_{\rho,\sigma} = \mu\ell_{\rho,\sigma}(C^{k+j_0}).
$$
\end{theorem}

\begin{proof} Take $A$ and $B$ satisfying the hypothesis of the statement with $c$ minimum. Set
\begin{align*}
& m_1:=m/c,\quad n_1:=n/c,\quad D_0:=A^{m_1}-B^{n_1}
\shortintertext{and}
& X:=\biggl\{D=D_0+\sum_{i,j\in \mathds{N}_0}\lambda_{ij}A^i B^j\in W^{(l)} : in+jm<cn_1 m_1 \text{ and }  \lambda_{ij}\in K \biggr\}.
\end{align*}
We claim that each element $D\in X$ satisfies
\begin{align}
&\ell_{\rho,\sigma}([D,B])= m_1 \lambda \ell_{\rho,\sigma}(C^{nm_1-n+h}) \label{commutatorDB1l}
\shortintertext{and}
&\ell_{\rho,\sigma}([D,A])= n_1 \lambda \ell_{\rho,\sigma}(C^{mn_1-m+h}). \label{commutatorDB2l}
\end{align}
In fact, this is true for $D_0$ since, by Proposition~\ref{pr v de un producto}, Lemma~\ref{ell del conmutador de una potencia} and items~(1) and~(4), we have
$$
\ell_{\rho,\sigma}([D_0,B]) = \ell_{\rho,\sigma}([A^{m_1},B])= m_1 \lambda \ell_{\rho,\sigma}(C^{nm_1-n+h}).
$$
and similarly
$$
\ell_{\rho,\sigma}([D_0,A]) = n_1 \lambda \ell_{\rho,\sigma}(C^{mn_1-m+h}).
$$
In particular $D_0\ne 0$. So, in order to establish~\eqref{commutatorDB1l} and~\eqref{commutatorDB2l}, it suffices to show that
\begin{align*}
& v_{\rho,\sigma}(\ell_{\rho,\sigma}([A^i B^j,B])) < (nm_1-n+h) v_{\rho,\sigma}(C)
\shortintertext{and}
&v_{\rho,\sigma}([A^i B^j,A]) < (mn_1-m+h) v_{\rho,\sigma}(C),
\end{align*}
for all $i,j$ such that $in+jm < n_1 m_1c$. But this follows from the fact that, again by Proposition~\ref{pr v de un producto}, Lemma~\ref{ell del conmutador de una potencia} and items~(1), (2) and~(4),
\begin{align*}
&\ell_{\rho,\sigma}([A^i B^j,B])= \ell_{\rho,\sigma}([A^i ,B]B^j) = i\lambda\ell_{\rho,\sigma}(C^{ni+mj-n+h})
\shortintertext{and}
&\ell_{\rho,\sigma}([A^i B^j,A]) = \ell_{\rho,\sigma}(A^i[B^j,A]) = j\lambda\ell_{\rho,\sigma}(C^{ni+mj-m+h}).
\end{align*}
Now, by Remark~\ref{re v de un conmutador}, equality~\eqref{commutatorDB1l} implies that for $D\in X$
$$
v_{\rho,\sigma}(D)+v_{\rho,\sigma}(B)-(\rho+\sigma) \ge v_{\rho,\sigma}(C^{nm_1-n+h}),
$$
and so, by item~(2),
\begin{equation}\label{eqq5l}
v_{\rho,\sigma}(D)> v_{\rho,\sigma}(C^{nm_1-n+h})-v_{\rho,\sigma}(B) = (nm_1-m-n+h)v_{\rho,\sigma}(C)
\end{equation}
for all $D\in X$. Hence, there exists $D_1\in X$ such that $v_{\rho,\sigma}(D_1)$ is minimum. We have two alternatives:
\begin{equation}
[D_1,B]_{\rho,\sigma}\ne 0\quad\text{or}\quad [D_1,B]_{\rho,\sigma}= 0.\label{eq20}
\end{equation}
Note that
$$
j_0: = nm_1-n-m+h\ge c (n_1m_1-n_1-m_1) =c\bigl((n_1-1)(m_1-1)-1\bigr) >0,
$$
since $\gcd(n_1,m_1) = 1$ and $n_1,m_1>1$ by item~(3). Hence, in the first case, the thesis holds with $k=m$ and $\mu = m_1\lambda$, because, by item~(2), Corollary~\ref{ell depende del ell} and~\eqref{commutatorDB1l},
$$
[D_1,C^m]_{\rho,\sigma} =  [D_1,B]_{\rho,\sigma} = m_1\lambda\ell_{\rho,\sigma}(C^{nm_1-n+h}).
$$
Assume now that $[D_1,B]_{\rho,\sigma} = 0$. We are going to show that this alternative is impossible, because it implies that $c$ is not minimum. In other words, that

\begin{itemize}

\item[(*)] there exist $\bar A, \bar B\in W^{(l)}$, $\bar \lambda\in K^{\times}$ and $\bar n,\bar m,\bar c,\bar h\in \mathds{N}$ with $\bar c<c$, such that~(1), (2), (3) and~(4) hold, with $\bar A$, $\bar B$, $\bar \lambda$, $\bar n$, $\bar m$, $\bar c$ and $\bar h$ instead of $A$, $B$, $\lambda$, $n$, $m$, $c$ and $h$ respectively.

\end{itemize}
With this purpose in mind, we claim that there exist $\lambda_1\in K^{\times}$ and $r\in \mathds{N}$, such that
\begin{equation}
\ell_{\rho,\sigma}(D_1) = \lambda_1 \ell_{\rho,\sigma}(C^r),\quad r<n_1 m_1 c,\quad r>c\quad\text{and}\quad c\nmid r\label{Al}
\end{equation}
In fact, by Corollary~\ref{ell depende del ell} and item~(2), we know that $[D_1,C^m]_{\rho,\sigma}  = [D_1,B]_{\rho,\sigma} = 0$, which by Lemma~\ref{ell del conmutador de una potencia} implies that $[D_1,C]_{\rho,\sigma} = 0$. Hence, by Theorem~\ref{f[] en W^{(l)}}, there exists $R = \ell_{\rho,\sigma}(R)\in L^{(l)}$, $\zeta,\xi \in K^{\times}$ and $r,s\in \mathds{N}$, such that
$$
\ell_{\rho,\sigma}(D_1) = \zeta R^r\quad\text{and}\quad \ell_{\rho,\sigma}(C) = \xi R^s.
$$
Besides, by the conditions required to $C$, it must be $s=1$ and so, Proposition~\ref{pr v de un producto}, $\ell_{\rho,\sigma}(D_1) = \frac{\zeta}{\xi^r}\ell_{\rho,\sigma}(C^r)$, which proves the equality in~\eqref{Al} with $\lambda_1 = \frac{\zeta}{\xi^r}$. Moreover
$$
r v_{\rho,\sigma}(C) = v_{\rho,\sigma}(D_1)\le v_{\rho,\sigma}(D_0)= v_{\rho,\sigma}(A^{m_1} - B^{n_1})<n_1 m_1 c v_{\rho,\sigma}(C),
$$
where the last inequality follows from the fact that, by items~(1) and~(2),
$$
\ell_{\rho,\sigma}(A^{m_1}) = \ell_{\rho,\sigma}(C^{cn_1m_1}) = \ell_{\rho,\sigma}(B^{n_1}).
$$
Thus $r<n_1 m_1 c$. Note that by~\eqref{eqq5l} and the equality in~\eqref{Al},
$$
rv_{\rho,\sigma}(C) = v_{\rho,\sigma}(D_1) > (nm_1-m-n+h)v_{\rho,\sigma}(C) \ge c(m_1n_1-m_1-n_1)v_{\rho,\sigma}(C).
$$
Hence
\begin{equation}
r>c(m_1n_1-m_1-n_1) = c\bigl((m_1-1)(n_1-1)-1\bigr)\ge c,\label{eqq9l}
\end{equation}
where the last equality holds, as before, since $m_1,n_1\ge 2$ and $m_1\ne n_1$. Next we will prove that $c$ does not divide $r$. Assume on the contrary that $c|r$. By~\cite[Lemma~4.1]{G-G-V2} and the first inequality in~\eqref{eqq9l} there exist $a_1,b_1\ge 0$ such that
$$
a_1 n_1 +b_1 m_1=\frac{r}{c}.
$$
Consequently
$$
a_1 n +b_1 m = r < cn_1 m_1,
$$
and so $D_2:=D_1-\lambda_1 A^{a_1}B^{b_1}\in X$. Moreover, since by items~(1) and~(2), and the equality in~\eqref{Al},
$$
\lambda_1 \ell_{\rho,\sigma}(A^{a_1}B^{b_1}) = \lambda_1 \ell_{\rho,\sigma}(C^{a_1 n +b_1 m})= \lambda_1 \ell_{\rho,\sigma}(C^r)= \ell_{\rho,\sigma}(D_1),
$$
we get $v_{\rho,\sigma}(D_2)<v_{\rho,\sigma}(D_1)$, which contradicts the minimality of $v_{\rho,\sigma}(D_1)$. Thus $c$ does not divide $r$.

Set $\bar c:=\gcd(c,r)$ and $\bar A:=\frac{1}{\lambda_1}D_1$. By~\eqref{Al}, we know that
$$
\ell_{\rho,\sigma}(\bar A) = \ell_{\rho,\sigma}(C^{\bar n}),
$$
where $\bar n := r>c>\bar c$. Moreover, by~\cite[Lemma~4.2]{G-G-V2} there exist $a,b\ge 0$, such that $\gcd(r,an+bm) = \bar c$. Note that $a>0$ or $b>0$, because $\bar c\ne r$. In particular $\bar c < c\le \min(n,m)\le an+bm$. Let $\bar B:=A^a B^b$. By Proposition~\ref{pr v de un producto} and items~(1) and~(2),
$$
\ell_{\rho,\sigma}(\bar B)= \ell_{\rho,\sigma}(C^{\bar m}),
$$
where $\bar m :=an+bm$.  So, in order to verify that~(*) holds, it only remains to establish~(4). But, since,
$$
\lambda_1[\bar A,\bar B] = [D_1,A^a B^b] = [D_1,A^a]B^b + A^a[D_1,B^b].
$$
and, by Proposition~\ref{pr v de un producto}, Lemma~\ref{ell del conmutador de una potencia}, items~(1) and~(2), and equalities~\eqref{commutatorDB1l} and~\eqref{commutatorDB2l}
\begin{align*}
& \ell_{\rho,\sigma}([D_1,A^a]B^b) = a n_1 \lambda \ell_{\rho,\sigma}(C^{\bar m +c(m_1 n_1-m_1-n_1)+h})
\shortintertext{and}
& \ell_{\rho,\sigma}(A^a[D_1,B^b]) = b m_1 \lambda \ell_{\rho,\sigma}(C^{\bar m +c(m_1 n_1-m_1 -n_1)+h}),
\end{align*}
we have
$$
\ell_{\rho,\sigma}([\bar A,\bar B]) = \bar \lambda C^{\bar h},
$$
with
$$
\bar \lambda:= \frac{\lambda}{\lambda_1}(a n_1+b m_1)\ne 0\quad\text{and}\quad\bar h:=\bar m +c (m_1 n_1-m_1 -n_1)+h>0,
$$
as desired.
\end{proof}

\begin{corollary}\label{diagonalcornerfreel} Let $(\rho,\sigma)\in \mathfrak{V}$ with $\sigma\le 0$, and let $C\in W^{(l)}$ such that
$$
v_{\rho,\sigma}(C)> 0\quad\text{and}\quad \ell_{\rho,\sigma}(C) \ne \zeta \ell_{\rho,\sigma}(D^h)\quad\!\text{for all $D\in W^{(l)}$, $\zeta\in K^{\times}$ and $h\in \mathds{N}$.}
$$
If there exist $n,m\in \mathds{N}$ and $A,B\in W^{(l)}$ such that
\begin{enumerate}

\smallskip

\item $\ell_{\rho,\sigma}(A)=\ell_{\rho,\sigma}(C^n)$,

\smallskip

\item $\ell_{\rho,\sigma}(B)=\ell_{\rho,\sigma}(C^m)$,

\smallskip

\item $c := \gcd(n,m)\notin \{n,m\}$,

\smallskip

\item $\ell_{\rho,\sigma}([A,B]) = \lambda \ell_{\rho,\sigma}(C^h)$, for some $h\in \mathds{N}_0$ and $\lambda\in K^{\times}$,

\smallskip

\end{enumerate}
then
$$
\Supp(\ell_{\rho,\sigma}(C))\ne \{(j,j)\}\quad\text{for all $j$.}
$$
\end{corollary}

\begin{proof} Assume on the contrary that $\ell_{\rho,\sigma}(C) = dX^jY^j$ for some $d\in K^{\times}$ and some~$j$. By Theorem~\ref{tecnico1} there exist $D\in W^{(l)}$, $\mu\in K^{\times}$, $k\in \mathds{N}$ and $j_0\in \mathds{N}_0$, such that $[D,C^k]_{\rho,\sigma} = \mu C^{k+j_0}$. Now, by Corollary~\ref{ell depende del ell}, we can assume that $D$ is $(\rho,\sigma)$-homogeneous. Write $D = \sum d_{rs} X^{r/l}Y^s$. Since, by Remark~\ref{re v de un conmutador},
$$
[d_{rl,r} X^rY^r,d^k X^{kj}Y^{kj}] = 0
$$
and, by Lemma~\ref{le conmutacion en Wl},
$$
\ell_{\rho,\sigma}([d_{rs} X^{r/l}Y^s,d^k X^{kj}Y^{kj}]) = (s-r/l)kj d_{rs} d^kx^{r/l+kj-1}y^{s+kj-1}\quad \text{for all $r\ne sl$,}
$$
we have
\begin{equation}
\ell_{\rho,\sigma}([D,d^k X^{kj}Y^{kj}]) = \sum_{r\ne sl} (s-r/l)kj d_{rs} d^k x^{r/l+kj-1} y^{s+kj-1}\label{eqq8l}
\end{equation}
On the other hand, by Proposition~\ref{pr v de un producto} and Corollary~\ref{ell depende del ell},
$$
\ell_{\rho,\sigma}([D,d^k X^{kj}Y^{kj}]) = \ell_{\rho,\sigma}([D,C^k]) = \mu \ell_{\rho,\sigma}(C^{k+j_0}) = \mu d^{k+j_0} x^{(k+j_0)j}y^{(k+j_0)j},
$$
which contradicts~\eqref{eqq8l}.
\end{proof}

\begin{lemma}\label{casi continuidad} Let $(\rho,\sigma)\in \mathfrak{V}$ and $(i/l,j)\in \frac{1}{l}\mathds{Z}\times \mathds{Z}$ such that $v_{\rho,\sigma}(i/l,j)>0$. Then there exists $(\rho'',\sigma'')<(\rho,\sigma)$ such that
$$
v_{\rho',\sigma'}(i/l,j)>0 \quad \text{for all $(\rho'',\sigma'')<(\rho',\sigma') <(\rho,\sigma)$}.
$$
\end{lemma}

\begin{proof} Suppose first that $i/l<j$ and take $(\rho'',\sigma''):= \lambda(jl,-i)$ with $\lambda = \frac{1}{\gcd(jl,i)}$. Since
$$
(\rho'',\sigma'')\times (\rho',\sigma') = \lambda(\rho' i + \sigma' j l) = l\lambda v_{\rho',\sigma'}(i/l,j),
$$
we have
$$
(\rho'',\sigma'')<(\rho,\sigma)\quad\text{and}\quad v_{\rho',\sigma'}(i/l,j)>0\quad\text{\!\! for all $(\rho'',\sigma'')<(\rho',\sigma')$}.
$$
Suppose now that $i/l\ge j$ and take $(\rho'',\sigma''):= (1,-1)$. Let $(\bar{\rho},\bar{\sigma}) = \lambda(-jl,i)$ with $\lambda = \frac{1}{\gcd(jl,i)}$. Since
$$
(\rho',\sigma')\times (\bar{\rho},\bar{\sigma}) = \lambda(\rho' i + \sigma' j l) = \lambda l v_{\rho',\sigma'}(i/l,j),
$$
we have
$$
(\rho,\sigma)<(\bar{\rho},\bar{\sigma})\quad\text{and}\quad  v_{\rho',\sigma'}(i/l,j)>0\quad\text{for all $(\rho',\sigma') < (\bar{\rho},\bar{\sigma})$}.
$$
The thesis follows immediately from these facts.
\end{proof}

\begin{corollary}\label{cont proporcionalidad} Let $(\rho,\sigma)\in \mathfrak{V}$ and $P,Q\in W^{(l)}$ such that $[Q,P]\in K^{\times}$. If  $[Q,P]_{\rho,\sigma}=0$, then there exists $(\rho'',\sigma'')<(\rho,\sigma)$ such that
$$
[Q,P]_{\rho',\sigma'}=0\quad\text{for all $(\rho'',\sigma'')<(\rho',\sigma') <(\rho,\sigma)$}.
$$
\end{corollary}

\begin{proof} By hypothesis
$$
v_{\rho,\sigma}(1,1) < v_{\rho,\sigma}(P) + v_{\rho,\sigma}(Q).
$$
Let $(i/l,j)\in \Supp(P)$ and $(i'/l,j')\in \Supp(Q)$ be such that $v_{\rho,\sigma}(P) = v_{\rho,\sigma}(i/l,j)$ and $v_{\rho,\sigma}(Q) = v_{\rho,\sigma}(i'/l,j')$. By Lemma~\ref{casi continuidad}, there exists $(\rho'',\sigma'') < (\rho,\sigma)$ such that
$$
v_{\rho',\sigma'}\bigl((i/l,j) + (i'/l,j') - (1,1)\bigr) > 0\quad \text{for all $(\rho'',\sigma'') < (\rho',\sigma') <(\rho,\sigma)$}.
$$
Hence,
$$
v_{\rho',\sigma'}(1,1) < v_{\rho',\sigma'}(P) + v_{\rho',\sigma'}(Q),
$$
from which the thesis follows immediately.
\end{proof}

\begin{theorem}\label{th tipo irreducibles} Let $P,Q\in W^{(l)}$ and $(\rho,\sigma)\in \mathfrak{V}$ with $\sigma\le 0$. Suppose that
\begin{alignat*}{3}
&[Q,P]=1,\qquad && v_{\rho,\sigma}(P)>0, \qquad && v_{\rho,\sigma}(Q)>0,\\
&[P,Q]_{\rho,\sigma}=0,\qquad && \frac{v_{\rho,\sigma}(P)}{v_{\rho,\sigma}(Q)}\notin \mathds{N},\qquad &&\frac{v_{\rho,\sigma}(Q)}{v_{\rho,\sigma}(P)}\notin \mathds{N}.
\end{alignat*}
Then

\begin{enumerate}

\smallskip

\item If $(\rho,\sigma)\notin \Val(P)$, then $v_{1,-1}(\ell_{\rho,\sigma}(P))\ne 0$.

\smallskip

\item $v_{1,-1}(\st_{\rho,\sigma}(P))\ne 0$.

\smallskip

\item There exists a $(\rho,\sigma)$-homogeneous element $F\in W^{(l)}$, such that
$$
[P,F]_{\rho,\sigma}=\ell_{\rho,\sigma}(P)\quad\text{and}\quad [Q,F]_{\rho,\sigma}= \frac{v_{\rho,\sigma}(Q)}{v_{\rho,\sigma}(P)} \ell_{\rho,\sigma}(Q).
$$
Furthermore
\begin{equation*}
v_{\rho,\sigma}(F) =\rho+\sigma\quad\text{and}\quad f^{(l)}_{[P,F],\rho,\sigma} = f^{(l)}_{P,\rho,\sigma},
\end{equation*}
where $f^{(l)}_{[P,F],\rho,\sigma}$ and $f^{(l)}_{P,\rho,\sigma}$ are the polynomials introduced in Definition~\ref{polinomio asociado f^{(l)}}.

\end{enumerate}

\end{theorem}

\begin{proof} By item~(2) of Theorem~\ref{f[] en W^{(l)}}, there exist $\lambda_P,\lambda_Q\in K^{\times}$, $m,n\in \mathds{N}$ and a $(\rho,\sigma)$-homogeneous polynomial $R\in L^{(l)}$, such that
$$
\ell_{\rho,\sigma}(P) = \lambda_P R^m\quad \text{and}\quad \ell_{\rho,\sigma}(Q) = \lambda_Q R^n.
$$
Clearly we can assume that
$$
R \ne \zeta S^h,\quad\text{for all $S\in L^{(l)}$, $\zeta\in K^{\times}$ and $h>1$.}
$$
Let $C\in W^{(l)}$ such that $\Psi^{(l)}(C) = R$. By Proposition~\ref{pr v de un producto},
$$
\ell_{\rho,\sigma}(P) = \lambda_P \ell_{\rho,\sigma}(C^m),\quad \ell_{\rho,\sigma}(Q) = \lambda_Q \ell_{\rho,\sigma}(C^n)
$$
and
$$
\ell_{\rho,\sigma}(C) \ne \zeta \ell_{\rho,\sigma}(D^h)\quad\text{for all $D\in W^{(l)}$, $\zeta\in K^{\times}$ and $h>1$.}
$$
Moreover, since, by item~(3) of Proposition~\ref{pr v de un producto},
\begin{equation}
v_{\rho,\sigma}(P) = m v_{\rho,\sigma}(C)\quad\text{and}\quad v_{\rho,\sigma}(Q) = n v_{\rho,\sigma}(C),\label{eqnue2}
\end{equation}
we deduce that $v_{\rho,\sigma}(C)\!>\!0$ and $\gcd(m,n)\!\notin\! \{m,n\}$. Thus, the conditions of Theo\-rem~\ref{tecnico1} and Corollary~\ref{diagonalcornerfreel} are fulfilled with
$$
A:=\frac{1}{\lambda_Q}Q,\quad B:=\frac{1}{\lambda_P}P,\quad h:=0\quad\text{and}\quad \lambda:= \frac{1}{\lambda_P \lambda_Q}.
$$
Consequently, if $(\rho,\sigma)\notin\Val(P)$, then $\Supp(\ell_{\rho,\sigma}(C)) = \{(i,j)\}$ with $i\ne j$, and so, $\Supp(\ell_{\rho,\sigma}(P)) = \{(m i,m j)\}$. Item~(1) follows immediately from this fact. In order to prove item~(2) we know that by Proposition~\ref{le basico}, Lemma~\ref{casi continuidad} and Corollary~\ref{cont proporcionalidad} there exists $(\rho',\sigma')<(\rho,\sigma)$ such that
\begin{align*}
& \Supp(\ell_{\rho',\sigma'}(P)) = \{\st_{\rho,\sigma}(P)\},\qquad\Supp(\ell_{\rho',\sigma'}(Q)) = \{\st_{\rho,\sigma}(Q)\},\\
& v_{\rho',\sigma'}(P)>0,\qquad v_{\rho',\sigma'}(Q)>0 \qquad\text{and}\qquad  [P,Q]_{\rho',\sigma'} = 0.
\end{align*}
We claim that $P$, $Q$ and $(\rho',\sigma')$ satisfy the hypothesis required $P$, $Q$ and $(\rho,\sigma)$ in the statement to of this theorem. By the above discussion, to do this, it only remains to prove that
$$
\sigma'\le 0 \qquad\text{and}\qquad\frac{v_{\rho',\sigma'}(P)}{v_{\rho',\sigma'}(Q)} \notin \mathds{N} \cup \left\{\frac{1}{n}:n\in \mathds{N}\right\}.
$$
The inequality follows directly from the fact that
$$
(\rho',\sigma')<(\rho,\sigma)\le (1,0),
$$
since $\sigma\le 0$. Let us verify the second condition. By item~(2) of Theorem~\ref{f[] en W^{(l)}}, there exist $\bar{\lambda}_P,\bar{\lambda}_Q\in K^{\times}$, $\bar{m},\bar{n}\in \mathds{N}$ and a $(\rho',\sigma')$-homoge\-neous polynomial $\bar{R}\in L^{(l)}$, such that
$$
\ell_{\rho',\sigma'}(P) = \bar{\lambda}_P \bar{R}^{\bar{m}}\quad \text{and}\quad \ell_{\rho',\sigma'}(Q) = \bar{\lambda}_Q \bar{R}^{\bar{n}}.
$$
Hence
$$
\{\st_{\rho,\sigma}(P)\} = \Supp(\ell_{\rho',\sigma'}(P)) = \bar{m} \Supp(\bar{R}) = \frac{\bar{m}}{\bar{n}} \Supp(\ell_{\rho',\sigma'}(Q)) = \frac{\bar{m}}{\bar{n}}\{\st_{\rho,\sigma}(Q)\},
$$
Thus,
$$
\frac{\bar{m}}{\bar{n}} v_{\rho,\sigma}(Q) = v_{\rho,\sigma}\Bigl(\frac{\bar{m}}{\bar{n}}\st_{\rho,\sigma}(Q)\Bigr) = v_{\rho,\sigma}(\st_{\rho,\sigma}(P)) = \frac{m}{n} v_{\rho,\sigma}(Q),
$$
where the last equality follows form~\eqref{eqnue2}. Consequently,
$$
\frac{v_{\rho',\sigma'}(P)}{v_{\rho',\sigma'}(Q)} = \frac{v_{\rho',\sigma'}(\ell_{\rho',\sigma'}(P))} {v_{\rho',\sigma'}(\ell_{\rho',\sigma'}(Q)} = \frac{\bar{m}}{\bar{n}} = \frac{m}{n}\notin \mathds{N} \bigcup \left\{\frac{1}{n}:n\in \mathds{N}\right\}.
$$
as desired. Applying item~(1) to $P$, $Q$ and $(\rho',\sigma')$, we obtain
$$
v_{1,-1}(\st_{\rho,\sigma}(P)) = v_{1,-1}(\ell_{\rho',\sigma'}(P))\ne 0
$$
which proves item~(2). Now, by Theorem~\ref{tecnico1}, there exist $D\in W^{(l)}$, $\mu\in K^{\times}$ and $k,j_0\in \mathds{N}$, such that
$$
[D,C^k]_{\rho,\sigma} = \mu C^{k+j_0},
$$
and so, by Theorem~\ref{te conmutadores1}, there exists a $(\rho,\sigma)$-homogeneous element $E\in W^{(l)}$, such that
$$
[C^t,E]_{\rho,\sigma}= t\ell_{\rho,\sigma}(C^t)\quad\text{for all $t\in \mathds{N}$.}
$$
Hence, by Corollary~\ref{ell depende del ell},
\begin{align*}
&[P,E]_{\rho,\sigma} = [\lambda_P C^m,E]_{\rho,\sigma} = m\ell_{\rho,\sigma}(\lambda_P C^m) = m \ell_{\rho,\sigma}(P)
\shortintertext{and}
&[Q,E]_{\rho,\sigma} = [\lambda_Q C^n,E]_{\rho,\sigma} = n\ell_{\rho,\sigma}(\lambda_Q C^n) = n \ell_{\rho,\sigma}(Q).
\end{align*}
If we set $F:=\frac{1}{m}E$, then we have
$$
[P,F]_{\rho,\sigma} = \ell_{\rho,\sigma}(P)\quad\text{and}\quad [Q,F]_{\rho,\sigma} = \frac{n}{m} \ell_{\rho,\sigma}(Q).
$$
Note now that $v_{\rho,\sigma}(P) = m v_{\rho,\sigma}(C)$ and $v_{\rho,\sigma}(Q) = n v_{\rho,\sigma}(C)$, and so
$$
\frac{n}{m} = \frac{v_{\rho,\sigma}(Q)}{v_{\rho,\sigma}(P)}.
$$
Finally it is clear that $[P,F]_{\rho,\sigma} = \ell_{\rho,\sigma}(P)$ implies
$$
v_{\rho,\sigma}(P) = v_{\rho,\sigma}(P) + v_{\rho,\sigma}(F) -(\rho+\sigma)\quad\text{and} \quad \ell_{\rho,\sigma}([P,F]) = \ell_{\rho,\sigma}(P).
$$
Consequently the last assertions in item~(3) are true.
\end{proof}

\begin{proposition}\label{pavadass} Let $(\rho,\sigma)\in \mathfrak{V}$ such that $\sigma\le 0$, let $P,F\in W^{(l)}\setminus\{0\}$ and let $f^{(l)}_{F,\rho,\sigma}$ and $f^{(l)}_{P,\rho,\sigma}$ be as in Definition~\ref{polinomio asociado f^{(l)}}. Assume that $F$ is $(\rho,\sigma)$-homogeneous and that $[P,F]_{\rho,\sigma}=\ell_{\rho,\sigma}(P)$. Then $f^{(l)}_{F,\rho,\sigma}$ is separable and every irreducible factor of $f^{(l)}_{P,\rho,\sigma}$ divides $f^{(l)}_{F,\rho,\sigma}$.
\end{proposition}

\begin{proof} By item~(1) of Theorem~\ref{f[] en W^{(l)}}, there exist $h\ge 0$ and $c\in \mathds{Z}$, such that
$$
x^h f_P = cf_P f_F+ ax f'_P f_F-bxf'_Ff_P,
$$
where $a:=\frac{1}{\rho} v_{\rho,\sigma}(F)$, $b:=\frac{1}{\rho} v_{\rho,\sigma}(P)$, $f^{(l)}_P:= f_{P,\rho,\sigma}$ and $f_F:= f^{(l)}_{F,\rho,\sigma}$. Hence, the pair of polynomials $(f_P,f_F)$ satisfies the condition $\PE(1,0,a,b,c)$ introduced in~\cite[Def.~1.23]{G-G-V2}. Since $f_P\ne 0\ne f_F$ the result follows from~\cite[Prop.~1.24]{G-G-V2}.
\end{proof}

\section{Cutting the right lower edge}

\setcounter{equation}{0}

The central result of this section is Proposition~\ref{preparatoria}, which loosely spoken cuts the right lower edge of the support of an irreducible subrectangular pair.

\begin{figure}[H]
\centering

\begin{tikzpicture}
\draw [->] (0,0) -- (0,4.5) node[anchor=east]{$Y$};
\draw (0,3.5) -- (1,4) -- (2,4) -- (3,3) -- (3.5,2) -- (2.5,0);
\draw (3,1) -- (2,0);

\fill[gray!50!white] (2,0) -- (3,1) -- (2.5,0) -- (2,0);
\draw [->] (0,0) -- (4.5,0) node[anchor=north]{$X$};
\draw (1.5,2) node{$P$};
\draw[<-, very thin] (2.6,0.40) .. controls (3,0.70) ..  (3.5,0.8);
\draw[xshift=3.4cm,yshift=0.8cm]
node[right,text width=6.5cm]{We cut a piece of the support};
\end{tikzpicture}
\caption{}
\end{figure}

This result is used in Section~\ref{computinglowerbounds} to determine lower bounds for
$$
B:=\min\{\gcd(v_{1,1}(P),v_{1,1}(Q)), \text{ where $(P,Q)$ is an irreducible pair}\}.
$$
In the last section we prove the existence of a complete chain of corners for a given an irreducible pair, that are compatible with Proposition~\ref{preparatoria}.

\begin{lemma}\label{le ell de la suma} Let $A_i\in W^{(l)}\setminus\{0\}$ ($i=0,\dots,n$) and let $(\rho,\sigma) \in \mathfrak{V}$. Suppose that there exists $q\in \mathds{Q}$ such that $v_{\rho,\sigma}(A_i)=q$ for all $i$ and set $A:= \sum_{i=0}^n A_i$. Then
\begin{align*}
A\ne 0\,\text{ and }\, v_{\rho,\sigma}(A) = q & \Longleftrightarrow \sum_i\ell_{\rho,\sigma}(A_i) \ne 0 \\
&\Longleftrightarrow A\ne 0\,\text{ and }\,\ell_{\rho,\sigma}(A) = \sum_i \ell_{\rho,\sigma}(A_i).
\end{align*}
\end{lemma}

\begin{proof} This is clear since the isomorphism of $K$-vector spaces $\Psi^{(l)}\colon W^{(l)}\to L^{(l)}$, introduced at the beginning of Section~\ref{preliminares}, preserves the $(\rho,\sigma)$-degree.
\end{proof}

For $\varphi\in \Aut(W^{(l)})$, we will denote by $\varphi_L$ the automorphism of $L^{(l)}$ given by $\varphi_L(x^{1/l}):=\Psi^{(l)}(\varphi(X^{1/l}))$ and $\varphi_L(y):= \Psi^{(l)} (\varphi(Y))$.

\begin{proposition}\label{pr ell por automorfismos} Let $(\rho,\sigma)\in \mathfrak{V}$ and $\lambda\in K$. Assume that $\sigma\le 0$ and $\rho|l$. Consider the automorphism of $W^{(l)}$ defined by $\varphi(X^{1/l})=X^{1/l}$ and $\varphi(Y) = Y+\lambda X^{\sigma/\rho}$. Then
$$
\ell_{\rho,\sigma}(\varphi(P)) = \varphi_L(\ell_{\rho,\sigma}(P))\quad\text{and}\quad v_{\rho,\sigma}(\varphi(P)) = v_{\rho,\sigma}(P)\quad\text{for all $P\in W^{(l)}\setminus\{0\}$}.
$$
Furthermore,
$$
\ell_{\rho_1,\sigma_1}(\varphi(P))=\ell_{\rho_1,\sigma_1}(P)\quad\text{for all $(\rho,\sigma)< (\rho_1,\sigma_1) < (-1,1)$.}
$$
\end{proposition}

\begin{proof} By item~3) of Proposition~\ref{pr v de un producto},
$$
v_{\rho,\sigma}\bigl(\varphi(X^{\frac{i}{l}} Y^j)\bigr) = i v_{\rho,\sigma}(X^{\frac{1}{l}})+ j v_{\rho,\sigma}(Y+\lambda X^{\frac{\sigma}{\rho}}) = \frac{i}{l} \rho + j\sigma = v_{\rho,\sigma}(X^{\frac{i}{l}} Y^j),
$$
for all $i\in \mathds{Z}$ and $j\in \mathds{N}_0$. Hence,
\begin{equation}
v_{\rho,\sigma}(\varphi(P)) = v_{\rho,\sigma}(P)\quad \text{for all $P\in W^{(l)}\setminus\{0\}$},\label{eq23}
\end{equation}
since $\varphi$ is bijective and therefore induce and isomorphism between the gradations associate with the $(\rho,\sigma)$-filtrations. We fix now a $P\in W^{(l)}\setminus\{0\}$ and write
$$
P = \sum_{i=0}^n \lambda_i X^{\frac{r}{l}-i\frac{\sigma}{\rho}}Y^{s+i} + R,
$$
where $R = 0$ or $v_{\rho,\sigma}(R) < v_{\rho,\sigma}(P)$. By~\eqref{eq23}, Lemma~\ref{le ell de la suma} and item~(2) of Proposition~\ref{pr v de un producto}
\begin{align*}
\ell_{\rho,\sigma}(\varphi(P)) &= \ell_{\rho,\sigma}\left(\sum_{i=0}^n \lambda_i \varphi\bigl(X^{\frac{r}{l}-i\frac{\sigma}{\rho}}Y^{s+i}\bigr) \right)\\
& = \ell_{\rho,\sigma}\left(\sum_{i=0}^n \lambda_i X^{\frac{r}{l}-i \frac{\sigma}{\rho}}(Y+\lambda X^{\frac{\sigma}{\rho}})^{s+i}\right)\\
& = \sum_{i=0}^n \lambda_i \ell_{\rho,\sigma}\bigl(X^{\frac{r}{l}-i \frac{\sigma}{\rho}}(Y+\lambda X^{\frac{\sigma}{\rho}})^{s+i}\bigr)\\
& = \sum_{i=0}^n \lambda_i x^{\frac{r}{l}-i \frac{\sigma}{\rho}}(y+\lambda x^{\frac{\sigma}{\rho}})^{s+i}\\
& = \varphi_L\left(\sum_{i=0}^n \lambda_i x^{\frac{r}{l}-i \frac{\sigma}{\rho}} y^{s+i}\right)\\
& = \varphi_L(\ell_{\rho,\sigma}(P)),
\end{align*}
as desired. Let $(\rho_1,\sigma_1)\in \mathfrak{V}$ such that $(\rho,\sigma)<(\rho_1,\sigma_1)$. Then $\rho_1\sigma < \rho \sigma_1$, and so
$$
\ell_{\rho_1,\sigma_1}(Y+\lambda X^{\frac{\sigma}{\rho}}) = y.
$$
Hence, by Proposition~\ref{pr v de un producto},
\begin{align*}
& \ell_{\rho_1,\sigma_1}\bigl(\varphi(X^{\frac{i}{l}} Y^j)\bigr)= \ell_{\rho_1,\sigma_1}\bigl(X^{\frac{i}{l}} (Y+\lambda X^{\frac{\sigma}{\rho}})^j\bigr)= x^{\frac{i}{l}} y^j
\shortintertext{and}
& v_{\rho_1,\sigma_1}\bigl(\varphi(X^{\frac{i}{l}} Y^j)\bigr) = \frac{i}{l} \rho_1 + j \sigma_1 = v_{\rho_1,\sigma_1}(X^{\frac{i}{l}} Y^j),
\end{align*}
which implies that
\begin{equation}
v_{\rho_1,\sigma_1}(\varphi(R)) = v_{\rho_1,\sigma_1}(R)\quad\text{for all $R\in W^{(l)}\setminus \{0\}$.} \label{eq24}
\end{equation}
Fix now $P\in W^{(l)}\setminus\{0\}$ and write
$$
P = \sum_{\{(i/l,j):\rho_1 i/l + \sigma_1 j = v_{\rho_1,\sigma_1}(P)\}}  \lambda_{i/l,j} X^{\frac{i}{l}}Y^j + R,
$$
with $R = 0$ or $v_{\rho_1,\sigma_1}(R)<v_{\rho_1,\sigma_1}(P)$. Again by~\eqref{eq23}, Lemma~\ref{le ell de la suma} and item~(2) of Proposition~\ref{pr v de un producto}
\begin{align*}
\ell_{\rho_1,\sigma_1}(\varphi(P)) &= \ell_{\rho_1,\sigma_1}\biggl(\sum_{i/l,j}\lambda_{i/l,j} \varphi(X^{\frac{i}{l}}Y^j)\biggr)\\
& = \ell_{\rho_1,\sigma_1}\biggl(\sum_{i/l,j}\lambda_{i/l,j} X^{\frac{i}{l}}(Y+\lambda X^{\frac{\sigma}{\rho}})^j\biggr)\\
& = \sum_{i/l,j}\lambda_{i/l,j} \ell_{\rho_1,\sigma_1}\bigl(X^{\frac{i}{l}}(Y+\lambda X^{\frac{\sigma}{\rho}})^j\bigr)\\
& = \sum_{i/l,j}\lambda_{i/l,j} x^{\frac{i}{l}}y^j\\
& = \ell_{\rho_1,\sigma_1}(P),
\end{align*}
as desired.
\end{proof}

For the rest of this section we assume that $K$ is algebraically closed.

\smallskip

Let $P,Q\in W^{(l)}$ and let $(\rho,\sigma)\in \mathfrak{V}$. Write
$$
\st_{\rho,\sigma}(P) = \Bigl(\frac{r}{l},s\Bigr)\quad\text{and}\quad \mathfrak{f}(x):= x^s f^{(l)}_{P,\rho,\sigma}(x).
$$
Let $\varphi\in \Aut(W^{(l')})$ be the automorphism defined by
$$
\varphi(X^{\frac{1}{l'}}) := X^{\frac{1}{l'}}\quad\text{and}\quad \varphi(Y) := Y + \lambda X^{\frac{\sigma}{\rho}},
$$
where $l':=\lcm(l,\rho)$ and $\lambda$ is any element of $K$ such that the multiplicity $m_{\lambda}$ of $x-\lambda$ in $\mathfrak{f}(x)$ is maximum.

\begin{proposition}\label{preparatoria} If
\begin{itemize}

\smallskip

\item[\rm{(a)}] $\sigma \le 0$,

\smallskip

\item[\rm{(b)}] $[Q,P]=1$,

\smallskip

\item[\rm{(c)}] $(\rho,\sigma)\in \Val(P)\cap    \Val(Q)$,

\smallskip

\item[\rm{(d)}] $v_{\rho,\sigma}(P)>0$ and $v_{\rho,\sigma}(Q)>0$,

\smallskip

\item[\rm{(e)}] $[P,Q]_{\rho,\sigma}=0$.

\smallskip

\item[\rm{(f)}] $\frac{v_{\rho,\sigma}(Q)}{v_{\rho,\sigma}(P)} \notin \mathds{N}$ and $\frac{v_{\rho,\sigma}(P)}{v_{\rho,\sigma}(Q)} \notin \mathds{N}$,

\smallskip

\item[\rm{(g)}] $v_{1,-1}\bigl(\en_{\rho,\sigma}(P)\bigr)<0$ and $v_{1,-1}\bigl(\en_{\rho,\sigma}(Q)\bigr)<0$,

\smallskip

\end{itemize}
then, there exists $(\rho',\sigma')\in \mathfrak{V}$ such that
\begin{enumerate}

\smallskip

\item $(\rho',\sigma')<(\rho,\sigma)$ and $(\rho',\sigma')\in \Val(\varphi(P))\cap    \Val(\varphi(Q))$,

\smallskip

\item $v_{1,-1}\bigl(\en_{\rho',\sigma'}(\varphi(P))\bigr)<0$ and $v_{1,-1} \bigl(\en_{\rho',\sigma'}(\varphi(Q))\bigr)<0$,

\smallskip

\item $v_{\rho',\sigma'}(\varphi(P))>0$ and $v_{\rho',\sigma'}(\varphi(Q))>0$,

\smallskip

\item $\frac{v_{\rho',\sigma'}(\varphi(P))}{v_{\rho',\sigma'}(\varphi(Q))}= \frac{v_{\rho,\sigma}(P)}{v_{\rho,\sigma}(Q)}$,

\smallskip

\item For all $(\rho,\sigma) <(\rho'',\sigma'') < (-1,1)$ the equalities
$$
\ell_{\rho'',\sigma''}(\varphi(P))=\ell_{\rho'',\sigma''}(P)\quad\text{and}\quad \ell_{\rho'',\sigma''}(\varphi(Q))=\ell_{\rho'',\sigma''}(Q)
$$
hold,

\smallskip

\item $\en_{\rho',\sigma'}(\varphi(P)) = \st_{\rho,\sigma}(\varphi(P)) = \Bigl(\frac{r}{l}+\frac{s\sigma}{\rho}-m_{\lambda} \frac{\sigma}{\rho},m_{\lambda}\Bigr)$,

\smallskip

\item $\en_{\rho',\sigma'}(\varphi(Q)) = \st_{\rho,\sigma}(\varphi(Q)))$ and $\en_{\rho',\sigma'}(\varphi(P)) = \frac{v_{\rho,\sigma}(P)}{v_{\rho,\sigma}(Q)} \en_{\rho',\sigma'}(\varphi(Q))$,

\smallskip

\item It is true that
$$
\qquad v_{0,1}(\en_{\rho',\sigma'}(\varphi(P))) <v_{0,1} (\en_{\rho,\sigma}(P)) \quad\text{or}\quad \en_{\rho',\sigma'}(\varphi(P)) = \en_{\rho,\sigma}(P).
$$
Furthermore, in the second case $\en_{\rho,\sigma}(P)+(\sigma/\rho,-1)\in \Supp(P)$,

\smallskip

\item $v_{\rho,\sigma}(\varphi(P)) = v_{\rho,\sigma}(P)$ and $v_{\rho,\sigma}(\varphi(Q)) = v_{\rho,\sigma}(Q)$,

\smallskip

\item $[\varphi(Q),\varphi(P)]_{\rho,\sigma} = 0$.

\smallskip

\item There exists  a $(\rho,\sigma)$-homogeneous element $F\in W^{(l)}$, which is not a monomial, such that
$$
[P,F]_{\rho,\sigma} = \ell_{\rho,\sigma}(P)\quad\text{and}\quad v_{\rho,\sigma}(F) =\rho+\sigma.
$$
Furthermore, if $\en_{\rho,\sigma}(F) = (1,1)$, then $\st_{\rho,\sigma}(\varphi(P)) = \en_{\rho,\sigma}(P)$,

\smallskip

\end{enumerate}
Note that

\begin{itemize}

\smallskip

\item[-] if $l'=1$, then $\varphi$ induces an automorphism of $W$,

\smallskip

\item[-] $\sigma'<0$, since $\sigma\le 0$ means $(\rho,\sigma)\le (1,0)$, and so $(\rho',\sigma')<(\rho,\sigma)\le (1,0)$ implies $\sigma'<0$,

\smallskip

\item[-] $v_{\rho,\sigma}(F) =\rho+\sigma>0$ implies that $\st_{\rho,\sigma}(F) \ne (0,0)\ne \en_{\rho,\sigma}(F)$.

\end{itemize}

\end{proposition}

\begin{figure}[h]
\centering

\begin{tikzpicture}[scale=0.75]
\draw[step=.5cm,gray,very thin] (0,0) grid (6.3,8.3);
\draw[step=.5cm,gray,very thin] (8,0) grid (14.3,8.3);
\draw [->] (1,0) -- (4,6) -- (4.5,8);
\draw (5.2,6) node[fill=white]{$\en_{\rho,\sigma}(P)$};
\draw (5.3,3.1) node[fill=white]{$\ell_{\rho,\sigma}$};
\draw[<-, very thin] (2.8,3.5) .. controls (4,3.1) ..  (4.8,3.1);
\draw[<-, very thin] (11.3,4.4) .. controls (12.5,4.7) ..  (13.8,4.7);
\fill[gray!15!white] (0,0) -- (1,0) -- (4,6) -- (4.4375,7.75)-- (0,7.75) -- (0,0);
\draw (1.5,4.3) node{$P$};
\draw [->] (8.5,0) -- (10.5,3) -- (12,6)-- ((12.5,8);
\fill[gray!15!white] (8,0) -- (8.5,0) -- (10.5,3) -- (12,6)-- ((12.4375,7.75) -- (8,7.75) -- (8,0);
\draw (9.5,4.3) node{$\varphi(P)$};
\draw [->] (0,0) -- (6.4,0) node[anchor=north]{$X$};
\draw [->] (8,0) -- (14.4,0) node[anchor=north]{$X$};
\draw [->] (0,0) -- (0,8.4) node[anchor=east]{$Y$};
\draw [->] (8,0) -- (8,8.4) node[anchor=east]{$Y$};
\draw [dashed] (9,0) -- (10.5,3);
\draw[<-, very thin] (9.4,1.3) .. controls (10.25,0.7) ..  (11.2,0.7);
\draw (12,0.7) node[fill=white]{$\ell_{\rho',\sigma'}$};
\draw (13.55,3.2) node[fill=white]{$\en_{\rho',\sigma'}\bigl(\varphi(P)\bigr)$};
\draw (13.6,2.4) node[fill=white] {$=\st_{\rho,\sigma}\bigl(\varphi(P)\bigr)$};
\draw[<-, very thin] (10.6,3) .. controls (11,2.8) ..  (11.9,2.75);
\draw[<-, very thin] (12.1,6) .. controls (12.5,6.2) ..  (13,6.25);
\draw (13.8,6.7) node[fill=white]{$\en_{\rho,\sigma}(P)$};
\draw (14.4,5.9) node[fill=white] {$=\en_{\rho,\sigma}\bigl(\varphi(P)\bigr)$};
\draw (14.4,4.7) node[fill=white]{$\ell_{\rho,\sigma}$};
\draw[->] (0,0) -- (1,-0.5);
\draw (1.7,-0.45) node{$(\rho,\sigma)$};
\draw[->] (8,0) -- (9,-0.5);
\draw (9.7,-0.45) node{$(\rho,\sigma)$};
\draw[->] (8,0) -- (9.5,-1);
\draw (10.35,-1) node{$(\rho',\sigma')$};
\end{tikzpicture}

\caption{Ilustration of Proposition~\ref{preparatoria}}
\end{figure}
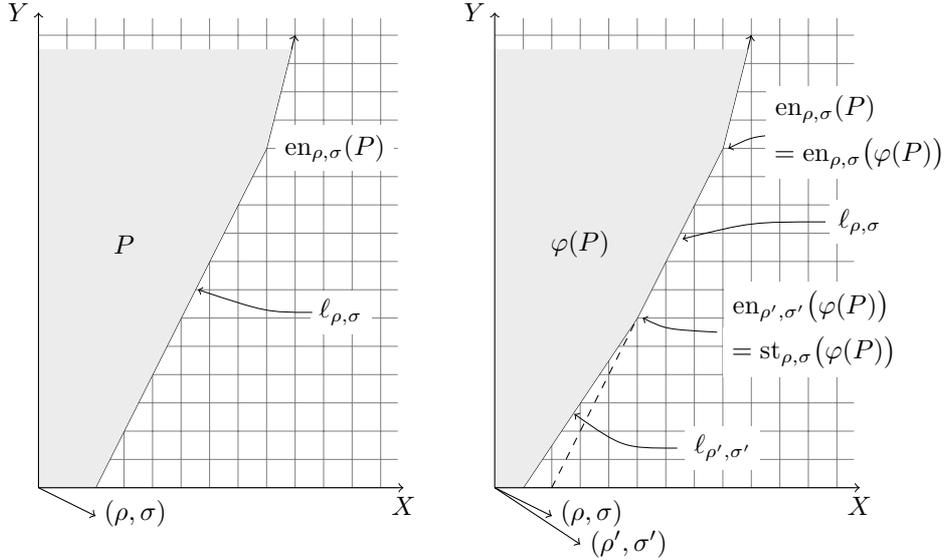

\begin{proof} Note that $\rho+\sigma>0$, $\rho>0$ and $\rho\ne -\sigma$ since $(\rho,\sigma)\in \mathfrak{V}$ and $\sigma\le 0$. We will use freely these facts. By item~(3) of Theorem~\ref{th tipo irreducibles} we can find a $(\rho,\sigma)$-homogeneous element $F\in W^{(l)}$ such that
\begin{equation}
[P,F]_{\rho,\sigma} = \ell_{\rho,\sigma}(P),\quad v_{\rho,\sigma}(F) =\rho+\sigma\quad\text{and}\quad f^{(l)}_{[P,F],\rho,\sigma} = f^{(l)}_{P,\rho,\sigma},\label{eqnue5}
\end{equation}
where $f^{(l)}_{[P,F],\rho,\sigma}$ and $f^{(l)}_{P,\rho,\sigma}$ are the polynomials introduced in Definition~\ref{polinomio asociado f^{(l)}}. We claim that
\begin{align}
1\le \#\factors(f^{(l)}_{P,\rho,\sigma})\le \deg(f^{(l)}_{F,\rho,\sigma}),\label{eq36}
\end{align}
where $\factors(f^{(l)}_{P,\rho,\sigma})$ denotes the number of linear different factors of $f^{(l)}_{P,\rho,\sigma}$. In fact, the first inequality follows from the fact that $(\rho,\sigma)\in \Val(P)$, while the second one follows from Proposition~\ref{pavadass}. Note that, by the very definition of $f^{(l)}_{F,\rho,\sigma}$, Condition~\eqref{eq36} implies that $F$ is not a monomial.

\smallskip

By~\eqref{eq57} and the definition of $\ell_{\rho,\sigma}(P)$ there exist $b_0,\dots,b_{\gamma}\in K$ with $b_0\ne 0$ and $b_{\gamma}\ne 0$, such that
$$
\ell_{\rho,\sigma}(P) = \sum_{i=0}^{\gamma} b_i x^{\frac{r}{l} -\frac{i\sigma}{\rho}} y^{s+i},
$$
and, again by~\eqref{eq57},
\begin{equation}
\en_{\rho,\sigma}(P) = \Bigl(\frac{r}{l}- \frac{\gamma\sigma}{\rho},s+\gamma\Bigr).\label{eqnue21}
\end{equation}
By Definition~\ref{polinomio asociado f^{(l)}},
$$
\mathfrak{f}(x) = \sum_{i=0}^{\gamma} b_i x^{i+s}.
$$
Let $(M_0,M) := \en_{\rho,\sigma}(F)$. From the second equality in~\eqref{eqnue5}, we obtain
\begin{equation}
M_0 = \frac{\rho+\sigma - \sigma M}{\rho}.\label{eqnue9}
\end{equation}
We assert that
\begin{equation}
1\le \#\factors(\mathfrak{f})\le M.\label{numero de factores}
\end{equation}
The first inequality if fulfilled because of~\eqref{eq36}. In order to prove the second one we begin by noting that, by the first equality in~\eqref{eqnue5} and item~(1) of Proposition~\ref{extremosnoalineados},
\begin{equation}
\st_{\rho,\sigma}(F)=(1,1)\quad\text{or}\quad \st_{\rho,\sigma}(F) \sim\st_{\rho,\sigma}(P),\label{eqnue6}
\end{equation}
and that $\st_{\rho,\sigma}(F)\ne (0,0)$ because of the second equality in~\eqref{eqnue5}. Hence, if $s>0$, then $\st_{\rho,\sigma}(F)\ne (u,0)$. Consequently, by~\eqref{eq57} and~\eqref{eq36},
$$
\#\factors(\mathfrak{f})=\#\factors(f^{(l)}_{P,\rho,\sigma})+1\le \deg(f^{(l)}_{F,\rho,\sigma})+1\le M.
$$
On the other hand, if $s=0$, then again by~\eqref{eq57} and~\eqref{eq36},
$$
\#\factors(\mathfrak{f})=\#\factors(f^{(l)}_{P,\rho,\sigma})\le \deg(f^{(l)}_{F,\rho,\sigma})\le M,
$$
as desired, proving the assertion.

\smallskip

For the sake of simplicity we set $N := \gamma+s$. Since $\deg \mathfrak{f} = N$, by~\eqref{numero de factores} there exists at least one factor $x-\lambda$ of $\mathfrak{f}$ with multiplicity $m_{\lambda}$ greater or equal to $N/M$. We take $\lambda\in K$ such that the multiplicity of $x-\lambda$ in $\mathfrak{f}(x)$ is maximum. We have
\begin{equation}
\mathfrak{f}(x) = \sum_{i = m_{\lambda}}^N a_i (x-\lambda)^i\quad\text{with $a_i\in K$, $a_{m_{\lambda}},a_N\ne 0$ and $m_{\lambda}\ge \frac{N}{M}$}.\label{eqnue8}
\end{equation}
Note that, since $\st_{\rho,\sigma}(P) = (r/l,s)$, by equality~\eqref{eq58} we have
$$
\ell_{\rho,\sigma}(P) = x^{\frac{r}{l}}y^s f^{(l)}_{P,\rho,\sigma}(x^{-\frac{\sigma} {\rho}}y)= x^{\frac{r}{l}+s\frac{\sigma}{\rho}}(x^{-\frac{\sigma}{\rho}}y)^s f^{(l)}_{P,\rho,\sigma}(x^{-\frac{\sigma} {\rho}}y)= x^{\frac{k}{l'}} \mathfrak{f}(x^{-\frac{\sigma} {\rho}}y),
$$
where $k:= \frac{rl'}{l}+\frac{l's\sigma}{\rho}$. So, by Proposition~\ref{pr ell por automorfismos},
$$
\ell_{\rho,\sigma}(\varphi(P)) = \varphi_L\bigl(\ell_{\rho,\sigma}(P)\bigr) = x^{\frac{k}{l'}} \mathfrak{f}(x^{-\frac{\sigma}{\rho}}y+\lambda) = \sum_{i=m_{\lambda}}^N a_i x^{\frac{k}{l'}} (x^{-\frac{\sigma}{\rho}}y)^i,
$$
since $\varphi_L(x^{-\sigma/\rho}y)=x^{-\sigma/\rho}y+\lambda$. But then, by the first equality in~\eqref{eq57},
\begin{equation}\label{starting de P}
\st_{\rho,\sigma}(\varphi(P))= \left(\frac{k}{l'}-m_{\lambda} \frac{\sigma}{\rho},m_{\lambda}\right) = \left(\frac{r}{l}+\frac{s\sigma}{\rho}-m_{\lambda} \frac{\sigma}{\rho},m_{\lambda}\right).
\end{equation}
Note also that by~\eqref{eqnue21},
\begin{equation}
\en_{\rho,\sigma}(P) = \Bigl(\frac{r}{l}- \frac{\gamma\sigma}{\rho},N\Bigr) = \Bigl(\frac{k}{l'}- \frac{N\sigma}{\rho},N\Bigr).\label{eqnue22}
\end{equation}
We claim that
\begin{equation}
v_{1,-1}(\st_{\rho,\sigma}(\varphi(P))<0.\label{eq50}
\end{equation}
First note that by Proposition~\ref{pr ell por automorfismos} (with $l$ replaced by $l'$),
\begin{equation}
v_{\rho,\sigma}(\varphi(P)) = v_{\rho,\sigma}(P)\qquad\text{and}\qquad v_{\rho,\sigma}(\varphi(Q)) = v_{\rho,\sigma}(Q).\label{eq49}
\end{equation}
So, item~(9) holds and the hypothesis of Theorem~\ref{th tipo irreducibles} are fulfilled for $\varphi(P)$, $\varphi(Q)$, $(\rho,\sigma)$ and $l'$. Item~(2) of that theorem gives
\begin{equation}\label{vunomenosunononulo}
 v_{1,-1}(\st_{\rho,\sigma}(\varphi(P))\ne 0.
\end{equation}
On the other hand, by the first equality in~\eqref{eqnue5} and item~(2) of Proposition~\ref{extremosnoalineados},
$$
\en_{\rho,\sigma}(F) = (1,1)\quad\text{or}\quad \en_{\rho,\sigma}(F) \sim\en_{\rho,\sigma}(P).
$$
In the first case
\begin{equation}
\Supp(F)\subseteq \{(1,1),(1+\sigma/\rho,0)\},\label{eqnue7}
\end{equation}
and so $\deg(f^{(l)}_{F,\rho,\sigma}) \le 1$. Hence, by~\eqref{eq36},
\begin{equation}
1\le \# \factors(f^{(l)}_{P,\rho,\sigma})\le \deg(f^{(l)}_{F,\rho,\sigma})\le 1,\label{eqnue24}
\end{equation}
and consequently in~\eqref{eqnue7} the equality holds. Thus, $\st_{\rho,\sigma}(F) = (1+\sigma/\rho,0)$, which implies that $l'=l$ and, by~\eqref{eqnue6}, also implies that $s=0$. Therefore $k=r$, $N=\gamma$ and $f^{(l)}_{P,\rho,\sigma} = \mathfrak{f}$. So, by~\eqref{eqnue24}
$$
\mathfrak{f} = a_N (x-\lambda)^N.
$$
where $a_N$ is as in~\eqref{eqnue8}. But then $m_{\lambda} = N = \gamma$ and so, by~\eqref{starting de P} and~\eqref{eqnue22},
$$
\st_{\rho,\sigma}(\varphi(P)) = \left(\frac{r}{l}-\gamma \frac{\sigma}{\rho},N\right) = \en_{\rho,\sigma}(P),
$$
which finishes the proof of item~(11) and yields~\eqref{eq50}, since $v_{1,-1}(\en_{\rho,\sigma}(P))<0$.

\smallskip

In the second case, by~\eqref{eqnue22}
$$
(M_0,M) := \en_{\rho,\sigma}(F) \sim \en_{\rho,\sigma}(P) = (N_0,N),
$$
where
\begin{equation}
N_0 := \frac{r}{l}- \frac{\gamma\sigma}{\rho} = \frac{k}{l'}- \frac{N\sigma}{\rho}.\label{eqnue11}
\end{equation}
Since, by~\eqref{numero de factores}
$$
M\ge 1\quad\text{and}\quad N :=\deg(\mathfrak{f})\ge \#\factors(\mathfrak{f})\ge 1,
$$
we have $\frac{N_0}{N} = \frac{M_0}{M}$. Hence, by~\eqref{eqnue9}, \eqref{eqnue8} and~\eqref{eqnue11},
$$
\frac{k\rho}{l'(\rho+\sigma)}= \frac{k/l'}{1+\sigma/\rho} = \frac{N_0+N\frac{\sigma}{\rho}} {M_0+M\frac{\sigma}{\rho}} =\frac{N(N_0/N+\sigma/\rho)}{M(M_0/M+\sigma/\rho)} = \frac{N}{M} \le m_{\lambda},
$$
which, combined with~\eqref{starting de P}, gives
$$
v_{1,-1}(\st_{\rho,\sigma}(\varphi(P)))=\frac{k}{l'}-m_{\lambda} \frac{\sigma}{\rho}-m_{\lambda} =\left(\frac{\sigma +\rho}{\rho}\right) \left(\frac{k\rho}{l'(\rho+\sigma)}-m_{\lambda}\right)\le 0.
$$
Taking into account~\eqref{vunomenosunononulo}, this yields~\eqref{eq50}, ending the proof of the claim. Now, by item~(2) of Theorem~\ref{f[] en W^{(l)}}, there exist relatively prime $\bar{m},\bar{n}\in \mathds{N}$, $\lambda_P,\lambda_Q\in K^{\times}$ and a $(\rho,\sigma)$-homogeneous $R\in L^{(l)}$ such that
\begin{equation}
\frac{\bar{n}}{\bar{m}} = \frac{v_{\rho,\sigma}(Q)}{v_{\rho,\sigma}(P)},\quad \ell_{\rho,\sigma}(P) = \lambda_P R^{\bar{m}}\quad\text{and}\quad \ell_{\rho,\sigma}(Q) = \lambda_Q R^{\bar{n}}.\label{eq46}
\end{equation}
Hence, again by Proposition~\ref{pr ell por automorfismos},
\begin{equation}
\ell_{\rho,\sigma}(\varphi(P))=\lambda_P \varphi_L (R)^{\bar{m}}\quad\text{and}\quad \ell_{\rho,\sigma}(\varphi(Q))=\lambda_Q \varphi_L(R)^{\bar{n}}.\label{eq43}
\end{equation}
Consequently, by items~(4) and (5) of Proposition~\ref{pr v de un producto},
\begin{align}
&\st_{\rho,\sigma}(\varphi(P))=\bar{m}\st_{\rho,\sigma}(\varphi_L(R)),\quad \en_{\rho,\sigma}(\varphi(P))= \bar{m}\en_{\rho,\sigma}(\varphi_L(R))\label{eq51}
\shortintertext{and}
&\st_{\rho,\sigma}(\varphi(Q)) = \bar{n}\st_{\rho,\sigma}(\varphi_L(R)),\quad \en_{\rho,\sigma}(\varphi(Q)) = \bar{n} \en_{\rho,\sigma}(\varphi_L(R)),\label{eq52}
\end{align}
and so
\begin{equation}
\st_{\rho,\sigma}(\varphi(P)) = \frac{\bar{m}}{\bar{n}}\st_{\rho,\sigma}(\varphi(Q))\quad\text{and}\quad \en_{\rho,\sigma}(\varphi(P)) = \frac{\bar{m}}{\bar{n}}\en_{\rho,\sigma}(\varphi(Q)).\label{eq53}
\end{equation}
We assert that
\begin{equation}
v_{0,1}(\st_{\rho,\sigma}(\varphi_L(R)))\ge 1.\label{eq47}
\end{equation}
In fact, otherwise $v_{0,1}(\st_{\rho,\sigma}(\varphi_L(R)))=0$, and so
$$
\st_{\rho,\sigma}(\varphi_L(R)) = (h,0)\quad\text{with $h\in \frac{1}{l'}\mathds{Z}$.}
$$
Then
\begin{equation}
v_{\rho,\sigma}(\varphi_L(R))=v_{\rho,\sigma}(\st_{\rho,\sigma}(\varphi_L(R)))= \rho h<0,\label{eqnue14}
\end{equation}
since $\rho>0$ and, by~\eqref{eq50} and~\eqref{eq51},
$$
h = v_{1,-1}(\st_{\rho,\sigma}(\varphi_L(R)))<0.
$$
But, by item~(3) of Proposition~\ref{pr v de un producto}, the second equality in~\eqref{eq46}, and the facts that $\varphi_L$ is $(\rho,\sigma)$-homogeneous, and, by hypothesis, $v_{\rho,\sigma}(P)>0$, we have
$$
v_{\rho,\sigma}(\varphi_L(R))=v_{\rho,\sigma}(R)>0,
$$
which contradicts~\eqref{eqnue14}. Hence inequality~\eqref{eq47} is true. Take now
\begin{align*}
& (\rho',\sigma'):=\max\{(\rho'',\sigma'')\in \ov{\Val}(\varphi(P)):(\rho'',\sigma'')< (\rho,\sigma)\}
\shortintertext{and}
& (\bar{\rho},\bar{\sigma}):=\max\{(\rho'',\sigma'')\in \ov{\Val}(\varphi(Q)):(\rho'',\sigma'')< (\rho,\sigma)\}.
\end{align*}
By Proposition~\ref{le basico}
\begin{equation}
\en_{\rho',\sigma'}(\varphi(P)) = \st_{\rho,\sigma}(\varphi(P)))\quad\text{and}\quad \en_{\bar{\rho},\bar{\sigma}}(\varphi(Q)) = \st_{\rho,\sigma}(\varphi(Q))). \label{eq48}
\end{equation}
Combining the first equality with equality~\eqref{starting de P}, we obtain item~(6). Moreover, by the first equalities in~\eqref{eq49} and~\eqref{eq48},
$$
v_{\rho,\sigma}(\en_{\rho',\sigma'}(\varphi(P))) = v_{\rho,\sigma}(\st_{\rho,\sigma}(\varphi(P)) = v_{\rho,\sigma}(\varphi(P)) =v_{\rho,\sigma}(P)>0,
$$
where the last inequality is true by hypothesis. Consequently
\begin{equation}
\en_{\rho',\sigma'}(\varphi(P))\ne (0,0).\label{eq40}
\end{equation}
We claim that
\begin{equation}
(\rho',\sigma') = (\bar{\rho},\bar{\sigma}).
\label{eqnue3}
\end{equation}
In order to prove this we proceed by contradiction. Assume that $(\rho',\sigma')>(\bar{\rho},\bar{\sigma})$. Then
\begin{equation}
\st_{\rho,\sigma}(\varphi(Q)))= \en_{\rho',\sigma'}(\varphi(Q))= \st_{\rho',\sigma'}(\varphi(Q)), \label{eq37}
\end{equation}
where the first equality follows from Proposition~\ref{le basico}, and the second one, from the fact that $(\rho',\sigma') \not\in\ov{\Val}(\varphi(Q))$. Furthermore
\begin{equation}
\en_{\rho',\sigma'}(\varphi(P))\ne \st_{\rho',\sigma'}(\varphi(P))\label{eq39}
\end{equation}
since $(\rho',\sigma')\in \Val(\varphi(P))$. Now, by~\eqref{eq48}, \eqref{eq53} and~\eqref{eq37},
\begin{equation}
\begin{aligned}
\en_{\rho',\sigma'}(\varphi(P))& = \st_{\rho,\sigma}(\varphi(P))\\
& = \frac{\bar{m}}{\bar{n}} \st_{\rho,\sigma}(\varphi(Q))\\
& = \frac{\bar{m}}{\bar{n}} \en_{\rho',\sigma'}(\varphi(Q))\\
& = \frac{\bar{m}}{\bar{n}} \st_{\rho',\sigma'}(\varphi(Q)).
\end{aligned}\label{eq38}
\end{equation}
We assert that
\begin{equation}
\en_{\rho',\sigma'}(\varphi(P))\nsim \st_{\rho',\sigma'}(\varphi(P)).\label{eq41}
\end{equation}
Otherwise, by the inequalities in~\eqref{eq40} and ~\eqref{eq39} there exists $\mu\in K\setminus\{1\}$ such that
\begin{equation*}
\st_{\rho',\sigma'}(\varphi(P)) = \mu\en_{\rho',\sigma'}(\varphi(P)).
\end{equation*}
which implies that
\begin{equation}
v_{\rho',\sigma'}(\varphi(P)) = \mu v_{\rho',\sigma'} (\varphi(P)).\label{eq42}
\end{equation}
On the other hand, by~\eqref{eq38}
$$
v_{\rho',\sigma'}(\varphi(Q)) = \frac{\bar{n}}{\bar{m}}v_{\rho',\sigma'}(\varphi(P)),
$$
which combined with equality~\eqref{eq42}, gives
$$
v_{\rho',\sigma'}(\varphi(P))=0=v_{\rho',\sigma'}(\varphi(Q)).
$$
But this contradicts Remark~\ref{re v de un conmutador}, since $[\varphi(Q),\varphi(P)] = 1$ and $\rho'+\sigma'>0$. Hence the condition~\eqref{eq41} is fulfilled. Combining this fact with~\eqref{eq38}, we obtain
$$
\st_{\rho',\sigma'}(\varphi(Q)) \nsim \st_{\rho',\sigma'}(\varphi(P)).
$$
Hence $[\varphi(Q),\varphi(P)]_{\rho',\sigma'} \ne 0$, by Corollary~\ref{extremos alineados}.
Then, since $[\varphi(Q),\varphi(P)] = 1$, it follows from item~(1) of Proposition~\ref{extremosnoalineados} that
\begin{equation}
\st_{\rho',\sigma'}(\varphi(P))+\st_{\rho',\sigma'}(\varphi(Q))-(1,1)=\st_{\rho',\sigma'}(1)=(0,0),\label{eq44}
\end{equation}
which implies that
\begin{equation}
v_{0,1}(\ell_{\rho',\sigma'}(\varphi(Q)))\in \{0,1\},\label{eq45}
\end{equation}
because the second coordinates in~\eqref{eq44} are non-negative. But, by~\eqref{eq43}, \eqref{eq47}, \eqref{eq37}, item~(4) of Proposition~\ref{pr v de un producto}, and the fact that $\bar{n}>1$, by the first equality in~\eqref{eq46},
$$
v_{0,1}(\en_{\rho',\sigma'}(\varphi(Q)))= v_{0,1}(\st_{\rho,\sigma}(\varphi(Q)))=\ov n v_{0,1}(\st_{\rho,\sigma} (\varphi_L(R)))> 1,
$$
which contradicts~\eqref{eq45}. Consequently, $(\rho',\sigma') > (\bar{\rho},\bar{\sigma})$ is impossible. Similarly one can prove that $(\rho',\sigma')<(\bar{\rho},\bar{\sigma})$ is also impossible, and so~\eqref{eqnue3} is true.

\smallskip

Using~\eqref{eq53}, \eqref{eq48}, \eqref{eqnue3}, and the fact that $\bar{m}/\bar{n} = v_{\rho,\sigma}(P)/ v_{\rho,\sigma}(Q)$, we obtain
$$
\en_{\rho',\sigma'}(\varphi(Q)) = \st_{\rho,\sigma}(\varphi(Q)))
$$
and
\begin{equation}\label{eq esquinas}
\en_{\rho',\sigma'}(\varphi(P)) = \st_{\rho,\sigma}(\varphi(P)) = \frac{v_{\rho,\sigma}(P)}{v_{\rho,\sigma}(Q)} \en_{\rho',\sigma'}(\varphi(Q)),
\end{equation}
which proves item~(7) and combined with~\eqref{eq50}, also proves item~(2). Hence $(\rho',\sigma')\ne (1,-1)$, since otherwise
$$
v_{1,-1}(\varphi(P))<0\quad\text{and}\quad v_{1,-1}(\varphi(Q))<0,
$$
which is impossible, because it contradicts Remark~\ref{re v de un conmutador}, since $[\varphi(Q),\varphi(P)] = 1$. This concludes the proof of item~(1). Now item~(3) follows, since by~\eqref{eq esquinas},
$$
v_{\rho',\sigma'}(\varphi(P))\le 0\Longleftrightarrow v_{\rho',\sigma'}(\varphi(Q))\le 0,
$$
and so, again by Remark~\ref{re v de un conmutador}, the falseness of item~(3) implies
$$
v_{\rho',\sigma'}(1) = v_{\rho',\sigma'}\bigr([\varphi(Q),\varphi(P)]\bigl) \le v_{\rho',\sigma'}(\varphi(Q)) + v_{\rho',\sigma'}(\varphi(P))-(\rho'+\sigma')<0,
$$
which is impossible. Item~(4) also follows from~\eqref{eq esquinas}, item~(5) from  Proposition~\ref{pr ell por automorfismos} and item~(10) from item~(9) and the facts that $[\varphi(Q),\varphi(P)]= [Q,P]$ and \hbox{$[Q,P]_{\rho,\sigma} = 0$}. Finally we prove Item~(8). Note that by item~(6) and~\eqref{starting de P}
\begin{equation}
\en_{\rho',\sigma'}(\varphi(P)) = \st_{\rho,\sigma}(\varphi(P))=\left(\frac{k}{l'}-m_{\lambda} \frac{\sigma}{\rho},m_{\lambda}\right),\label{eqnue23}
\end{equation}
and so by~\eqref{eqnue22},
$$
v_{0,1}(\en_{\rho',\sigma'}(\varphi(P))) = m_{\lambda}\le N = v_{0,1}(\en_{\rho,\sigma}(P)),
$$
Furthermore, if the equality holds, then by~\eqref{eqnue22} and~\eqref{eqnue23},
$$
\en_{\rho',\sigma'}(\varphi(P))= \en_{\rho,\sigma}(P)\quad\text{and}\quad x^s f^{(l)}_{P,\rho,\sigma}(x) = \mathfrak{f}(x) = a_N (x-\lambda)^N,
$$
where $a_N$ is as in~\eqref{eqnue8}. But $\deg(f^{(l)}_{P,\rho,\sigma})>0$ and $x\nmid f^{(l)}_{P,\rho,\sigma}$, since $(\rho,\sigma)\in \Val(P)$ and $f^{(l)}_{P,\rho,\sigma}(0)\ne 0$. Hence, from the last equality it follows that $\lambda\ne 0$, $s = 0$ and $\st_{\rho,\sigma}(P) = (k/l',0)$. So, by~\eqref{eq58}
$$
\ell_{\rho,\sigma}(P) = x^{\frac{k}{l'}} \mathfrak{f}(x^{-\frac{\sigma} {\rho}}y)= a_N x^{\frac{k}{l'}}(x^{-\frac{\sigma} {\rho}}y-\lambda)^N = \sum_{i=0}^N a_N \binom{N}{i} \lambda^{N-i}x^{\frac{k}{l'}-i\frac{\sigma}{\rho}}y^i.
$$
Consequently,
$$
\Bigl(\frac{k}{l'}- \frac{(N-1)\sigma}{\rho},N-1\Bigr)\in \Supp(P),
$$
since $\lambda\ne 0$. This finishes the proof because
$$
\en_{\rho,\sigma}(P)+\Bigl(\frac{\sigma}{\rho},-1\Bigr) = \Bigl(\frac{k}{l'}- \frac{(N-1)\sigma}{\rho}, N-1\Bigr)
$$
by~\eqref{eqnue22}.
\end{proof}

The following definition generalize~\cite[Def.~2.2]{G-G-V2}.

\begin{definition}\label{def de val} Let $l\in \mathds{N}$. For each $(r,s)\in\frac{1}{l}\mathds{Z}\times \mathds{Z} \setminus \mathds{Z}(1,1)$, we define $\val(r,s)$ to be the unique $(\rho,\sigma)\in \mathfrak{V}$ such that $v_{\rho,\sigma}(r,s)=0$.
\end{definition}

\begin{remark}\label{re val} Note that if $P\in W^{(l)}\setminus\{0\}$ and $(\rho,\sigma)\in \Val(P)$, then
$$
(\rho,\sigma) = \val\bigl(\en_{\rho,\sigma}(P) - \st_{\rho,\sigma}(P)\bigr).
$$
\end{remark}

\begin{proposition}\label{lema general} Let $P,Q\in W^{(l)}$ and let $(\rho,\sigma)\in \mathfrak{V}$ such that conditions~(a), (b), (c), (d) and~(f) of Proposition~\ref{preparatoria} are fulfilled. Assume that $\frac{v_{\rho,\sigma}(Q)}{v_{\rho,\sigma}(P)} = \frac nm$ with $n,m>1$ and $\gcd(n,m)=1$. Then
$$
\frac 1m \en_{\rho,\sigma}(P)\ne \left(\bar{r}-\frac 1l,\bar{r}\right),
$$
for all $\bar{r}\ge 2$.
\end{proposition}

\begin{proof} We will assume that
\begin{equation}
\frac 1m \en_{\rho,\sigma}(P) = \left(\bar{r}-\frac 1l,\bar{r}\right),\label{pepi3}
\end{equation}
for some fixed $\bar{r}\ge 2$ and we will prove successively the following two items:
\begin{enumerate}

\smallskip

\item $[P,Q]_{\rho,\sigma} = 0$, $v_{1,-1}(\en_{\rho,\sigma}(P)) <0$ and $v_{1,-1}(\en_{\rho,\sigma}(Q)) <0$.

\smallskip

\item $\rho|l$ and there exist $\varphi\in \Aut(W^{(l)})$ and $(\rho_1,\sigma_1)\in \mathfrak{V}$ with $(\rho_1,\sigma_1) < (\rho,\sigma)$, such that $P_1:=\varphi(P)$, $Q_1:=\varphi(Q)$ and $(\rho_1,\sigma_1)$ satisfy conditions~(a), (b), (c), (d) and~(f) of Proposition~\ref{preparatoria} (more precisely, these conditions are fulfilled with $(P_1,Q_1)$ playing the role of $(P,Q)$ and $(\rho_1,\sigma_1)$ playing the role of $(\rho,\sigma)$). Furthermore,
$$
\qquad\qquad\frac{v_{\rho_1,\sigma_1}(P_1)}{v_{\rho_1,\sigma_1}(Q_1)}= \frac{v_{\rho,\sigma}(P)} {v_{\rho,\sigma}(Q)}\qquad \text{and}\qquad \frac 1m \en_{\rho_1,\sigma_1}(P_1)= \left(\bar{r}-\frac 1l,\bar{r}\right).
$$

\smallskip

\end{enumerate}
Item~(2) yields an infinite, descending chain of valuations $(\rho_k,\sigma_k)$, such that $\rho_k|l$. But there are only finitely many $\rho_k$'s with $\rho_k|l$. Moreover, $0<-\sigma_k<\rho_k$, so there are only finitely many $(\rho_k,\sigma_k)$ possible, which provide us with the desired contradiction.

\smallskip

We first prove item~(1). Set $A:=\frac 1m \en_{\rho,\sigma}(P)$ and suppose $[P,Q]_{\rho,\sigma}\ne 0$. Since
$$
v_{\rho,\sigma}(P) = v_{\rho,\sigma}(mA),\quad v_{\rho,\sigma}(Q) = v_{\rho,\sigma}(nA)\quad\text{and}\quad [P,Q]=1,
$$
under this assumption we have
$$
v_{\rho,\sigma}\left(mA+nA - (1,1)\right) = v_{\rho,\sigma}(P) + v_{\rho,\sigma}(Q) - v_{\rho,\sigma}(1,1) = 0.
$$
Consequently,
\begin{equation}
v_{\rho,\sigma}(A) = \frac{\rho+\sigma}{m+n}\quad\text{and}\quad \rho\bigl(m\bar{r}l+n\bar{r}l-m-n-l\bigr) = - \sigma\bigl(m\bar{r}l+n\bar{r}l-l\bigr),\label{pepi}
\end{equation}
where for the second equality we use assumption~\eqref{pepi3}. Let
$$
d := \gcd(m\bar{r}l+n\bar{r}l-l,m+n-m\bar{r}l-n\bar{r}l+l) = \gcd(l,m+n).
$$
From the second equality in~\eqref{pepi}, it follows that
\begin{equation}
\rho = \frac{m\bar{r}l+n\bar{r}l-l}d\quad\text{and}\quad \sigma = \frac{m+n-m\bar{r}l-n\bar{r}l+l}{d} = \frac{m+n}{d}-\rho, \label{pepi2}
\end{equation}
and so $\rho+\sigma = (m+n)/d$. Hence, by the first equality in~\eqref{pepi},
$$
v_{\rho,\sigma}(P)=mv_{\rho,\sigma}(A)=\frac {m(\rho+\sigma)}{m+n}=\frac md\quad\text{and}\quad v_{\rho,\sigma}(Q)= nv_{\rho,\sigma}(A)=\frac nd.
$$
We will see that we are lead to
\begin{equation}
v_{1,-1}(P)\le 0\qquad\text{and}\qquad v_{1,-1}(Q)\le 0,\label{pepi1}
\end{equation}
which is impossible, since $[P,Q]=1$. In order to prove~\eqref{pepi1}, it suffices to check that if $(i,j)\in \frac 1l \mathds{Z}\times \mathds{N}_0$ and $i>j$, then $v_{\rho,\sigma}(i,j)>\max\{v_{\rho,\sigma}(P), v_{\rho,\sigma}(Q)\}$. But, writing $(i,j)=(j+\frac sl,j)$ with $s\in \mathds{N}$, we obtain
\begin{align*}
v_{\rho,\sigma}(i,j)=&\rho j+\rho \frac sl+\frac{m+n}{d}j-\rho j &&\text{by~\eqref{pepi2}}\\
=&\frac{s(m\bar{r}+n\bar{r}-1)}{d}+\frac{m+n}{d}j &&\text{by~\eqref{pepi2}} \\
\ge & \frac{(m+n)\bar{r}-1}{d}\\
\ge & \frac{m+n}{d} &&\text{since $\bar{r}\ge 2$ and $m+n\ge 1$.}\\
> & \max\{m/d,n/d\}\\
= & \max\{v_{\rho,\sigma}(P), v_{\rho,\sigma}(Q)\}.
\end{align*}
This concludes the proof that $[P,Q]_{\rho,\sigma}\! =\! 0$. Now, by Corollary~\ref{extremos alineados} and the assump\-tion~\eqref{pepi3}, we have
\begin{align*}
& v_{1,-1}(\en_{\rho,\sigma}(P)) =  mv_{1,-1}\left(\bar{r}-\frac 1l,\bar{r}\right)<0
\shortintertext{and}
& v_{1,-1}(\en_{\rho,\sigma}(Q)) = \frac{n}{m}v_{1,-1}(\en_{\rho,\sigma}(P))<0,
\end{align*}
which finishes the proof of item~(1).

\smallskip

We now prove item~(2). By Item~(1), the hypothesis of Proposition~\ref{preparatoria} are sa\-tisfied. Let $(\rho',\sigma')$ and $\varphi$ be as in its statement. Set
$$
P_1:=\varphi(P),\quad Q_1:=\varphi(Q)\quad\text{and}\quad (\rho_1,\sigma_1):= (\rho',\sigma').
$$
By items~(1), (3) and~(4) of Proposition~\ref{preparatoria}, we know that
$$
\frac{v_{\rho_1,\sigma_1}(P_1)}{v_{\rho_1,\sigma_1}(Q_1)}= \frac{v_{\rho,\sigma}(P)}{v_{\rho,\sigma}(Q)},
$$
and that conditions~(a), (c), (d) and~(f) of that proposition are fulfilled for $P_1$, $Q_1$ and $(\rho_1,\sigma_1)$. Moreover condition (b) follows immediately from the fact that $\varphi$ is an algebra automorphism. It remains to prove that
\begin{equation}
\rho\mid l \quad \text{and}\quad \frac 1m \en_{\rho_1,\sigma_1}(P_1)= \left(\bar{r}-\frac 1l,\bar{r}\right).\label{pepi13}
\end{equation}
By item~(11) of Proposition~\ref{preparatoria}, there is a $(\rho,\sigma)$-homogeneous element $F$, which is not a monomial, such that
\begin{equation}
[P,F]_{\rho,\sigma} = \ell_{\rho,\sigma}(P)\quad\text{and}\quad v_{\rho,\sigma}(F) =\rho+\sigma.\label{pepi6}
\end{equation}
By item~(2) of Proposition~\ref{extremosnoalineados},
$$
\en_{\rho,\sigma}(F) = (1,1)\quad\text{or}\quad \en_{\rho,\sigma}(F) \sim\en_{\rho,\sigma}(P).
$$
By items~(6) and~(11) of Proposition~\ref{preparatoria}, in  the first case we have
$$
\frac{1}{m}\en_{\rho_1,\sigma_1}(P_1) = \frac{1}{m}\st_{\rho,\sigma}(P_1) = \frac{1}{m}\en_{\rho,\sigma}(P) = \left(\bar{r}-\frac 1l,\bar{r}\right).
$$
Hence, by item~(8) of the same proposition,
$$
\en_{\rho,\sigma}(P)+\left(\frac{\sigma}{\rho},-1\right)\in \Supp(P)\subseteq \frac{1}{l}\mathds{Z}\times \mathds{Z}.
$$
Since $\en_{\rho,\sigma}(P)\in \frac{1}{l}\mathds{Z}\times \mathds{N}_0$, this implies that
$$
\left(\frac{\sigma}{\rho},-1\right)\in \frac{1}{l}\mathds{Z}\times \mathds{Z},
$$
and so $\rho|\sigma l$. But then $\rho|l$, since $\gcd(\rho,\sigma)=1$. This finishes the proof of~\eqref{pepi13} when $\en_{\rho,\sigma}(F) = (1,1)$.

\smallskip

Assume now that $\en_{\rho,\sigma}(F) \sim\en_{\rho,\sigma}(P)$. Then, since $\left(\bar{r}-\frac 1l,\bar{r}\right)$ is indivisible in $\frac 1l \mathds{Z}\times \mathds{N}_0$, we have
\begin{equation}
\en_{\rho,\sigma}(F) = \mu \frac 1m \en_{\rho,\sigma}(P) = \mu \left(\bar{r}-\frac 1l,\bar{r}\right) \quad\text{with $\mu\in \mathds{N}$.}\label{pepi5}
\end{equation}
We claim that
$$
\mu=1,\quad \bar{r}=2\quad\text{and}\quad \rho = l.
$$
By item~(2) of Theorem~\ref{f[] en W^{(l)}} there exist $\lambda_P,\lambda_Q\in K^{\times}$ and a $(\rho,\sigma)$-homogeneous polynomial $R\!\in\! L^{(l)}$,  such that
\begin{equation}
\ell_{\rho,\sigma}(P) = \lambda_P R^m\quad \text{and}\quad \ell_{\rho,\sigma}(Q) = \lambda_Q R^n.\label{pepi14}
\end{equation}
Note that $R$ is not a monomial, since $(\rho,\sigma)\in \Val(P)$. By item~(5) of Proposition~\ref{pr v de un producto} and the assumption~\eqref{pepi3}, we have
\begin{equation}
\en_{\rho,\sigma}(R) = \left(\bar{r}-\frac 1l,\bar{r}\right).\label{pepi8}
\end{equation}
Hence, by item~(2) of Proposition~\ref{le basico1},
\begin{equation}
v_{1,-1}(\st_{\rho,\sigma}(R))>v_{1,-1}(\en_{\rho,\sigma}(R))=-\frac 1l.\label{pepi4}
\end{equation}
Since, by item~(2) of Theorem~\ref{th tipo irreducibles} and item~(4) of Proposition~\ref{pr v de un producto},
$$
v_{1,-1}(\st_{\rho,\sigma}(R)) = \frac 1m v_{1,-1}(\st_{\rho,\sigma}(P))\ne 0,
$$
from inequality~\eqref{pepi4} it follows that
\begin{equation}
v_{1,-1}(\st_{\rho,\sigma}(R))> 0.\label{pepi10}
\end{equation}
Moreover, by equality~\eqref{pepi5} and the second equality in~\eqref{pepi6}
$$
v_{\rho,\sigma}\left(\mu\left(\bar{r}-\frac 1l,\bar{r}\right)-(1,1)\right)=0,
$$
which implies that
\begin{equation}
\rho(\mu \bar{r} l - \mu -l) = -\sigma(\mu \bar{r} l - l).\label{pepi7}
\end{equation}
Let
\begin{equation}
d := \gcd(\mu \bar{r} l - \mu -l,\mu \bar{r} l - l) = \gcd(\mu,l).\label{pepi12}
\end{equation}
By equality~\eqref{pepi7}
\begin{equation*}
\rho=\frac{\mu \bar{r} l -l}d\quad\text{and}\quad \sigma=\frac{\mu - \mu \bar{r} l + l}{d} = \frac{\mu}{d}-\rho.
\end{equation*}
Hence
\begin{equation}
\rho=\frac{\mu \bar{r} l -l}d \quad\text{and}\quad \rho+\sigma=\frac \mu d.\label{pepi9}
\end{equation}
So,
\begin{equation}
v_{\rho,\sigma}\left(j+\frac sl,j\right)=\frac{\mu j}{d}+\frac {(\mu \bar{r}-1)s}d \ge \frac{\mu \bar{r}-1}{d}\quad\text{for all $j\in \mathds{N}_0$ and $s\in\mathds{N}$.}\label{pepi11}
\end{equation}
If $\bar{r}>2$ or $\mu>1$, this yields
$$
v_{\rho,\sigma}\left(j+\frac sl,j\right)>\frac 1d = v_{\rho,\sigma}(R),
$$
where the last equality follows from~\eqref{pepi8} and~\eqref{pepi9}. Hence, no $(i,j)\in \frac 1l \mathds{Z}\times \mathds{N}$ with $i>j$ lies in the support of $R$, and so $v_{1,-1}(\st_{\rho,\sigma}(R))\le 0$, which contradicts inequality~\eqref{pepi10}. Thus, necessarily $\bar{r}=2$ and $\mu = 1$, which, by equality~\eqref{pepi12} and the first equality in~\eqref{pepi9}, implies $d=1$ and $\rho=l$. This finishes the proof of the claim. Combining this with~\eqref{pepi5} and~\eqref{pepi8}, we obtain
\begin{equation}
\en_{\rho,\sigma}(F) = \en_{\rho,\sigma}(R) = \left(2-\frac{1}{\rho},2\right).\label{pepi16}
\end{equation}
Now, by~\eqref{eq57}, there exists $\gamma\in\{1,2\}$, such that
$$
\st_{\rho,\sigma}(R) = \left(2-\frac{1}{\rho} + \frac {\gamma\sigma}{\rho},2-\gamma \right)= \left(1 + \frac {(\gamma-1)\sigma}{\rho},2-\gamma \right),
$$
where the last equality follows from the fact that, by the second equality in~\eqref{pepi9}, we have $\rho+\sigma = 1$. But the case $\gamma=1$ is impossible, since it contradicts inequality~\eqref{pepi10}. Thus, necessarily
\begin{equation}
\st_{\rho,\sigma}(R) = \left(1 + \frac{\sigma}{\rho},0\right) = \left(\frac{1}{\rho},0\right).\label{pepi17}
\end{equation}
Note that from equalities~\eqref{pepi16} and~\eqref{pepi17} it follows that $\deg\bigl(f_{R,\rho,\sigma}^{(\rho)}\bigr)=2$. Hence, by Remark~\eqref{f de un producto} and the first equality in~\eqref{pepi14},
$$
f_{P,\rho,\sigma}^{(\rho)} = a(x-\lambda)^{2m}\qquad\text{or}\qquad f_{P,\rho,\sigma}^{(\rho)} = a(x-\lambda)^m(x-\lambda')^m,
$$
where $a,\lambda,\lambda'\in K^{\times}$ and $\lambda'\ne \lambda$. Let $\mathfrak{f}$ be as in Proposition~\ref{preparatoria}. By item~(4) of Proposition~\ref{pr v de un producto}, the first equality in~\eqref{pepi14} and equality~\eqref{pepi17},
$$
\st_{\rho,\sigma}(P) = \left(\frac{m}{\rho},0\right)\qquad\text{and}\qquad \mathfrak{f} = f_{P,\rho,\sigma}^{(\rho)}.
$$
Let $m_{\lambda}$ be the multiplicity of $x-\lambda$ in $\mathfrak{f}$. By item~(6) of Proposition~\ref{preparatoria},
$$
\en_{\rho_1,\sigma_1}(P_1) = \st_{\rho,\sigma}(P_1) = \Bigl(\frac{m}{\rho}-m_{\lambda} \frac{\sigma}{\rho}, m_{\lambda}\Bigr) = \begin{cases} (m,m) &\text{if $m_{\lambda}=m$,}\\ (2m-m/\rho,2m) &\text{if $m_{\lambda}=2m$,}\end{cases}
$$
where for the computation we used that $\rho+\sigma=1$. In order to finish the proof it is enough to show that the case $m_{\lambda}$ is impossible, which will follow from item~(2) of Theorem~\ref{th tipo irreducibles}, if we can show that its hypothesis are satisfied by $P_1$, $Q_1$ and $(\rho,\sigma)$. But this is true by items~(9) and~(10) of Proposition~\ref{preparatoria}.
\end{proof}

\section{Computing lower bounds}\label{computinglowerbounds}

\setcounter{equation}{0}

Our next aim is to determine a lower bound for the value
$$
B:=\min\{\gcd(v_{1,1}(P),v_{1,1}(Q)), \text{ where $(P,Q)$ is an irreducible pair}\}.
$$
More precisely, we will prove that $B>14$. In a forthcoming article we will carry these results over from Dixmier pairs to Jacobian pairs, where this first result already improves the lower bound for the greatest common divisor of the degrees given in~\cite{N1} and~\cite{N2} which is $B>8$. We will also try to raise this lower bound to at least~$52$, which would imply that
$$
\max\{\deg(P),\deg(Q)\}\ge 156,
$$
improving thus the result of~\cite{M}, which says that a Keller map $F$ with $\deg(F)<101$ is invertible.

\smallskip

Let $(P,Q)$ be an irreducible pair and let $(\rho,\sigma)\!\in\! \Val(P)$. Assume that $\sigma\!<\!0$. By~\cite[Prop.~6.3~(1) and~(2)]{G-G-V2}, we know that there exist $\lambda_P,\lambda_Q\in K^{\times}$, $n,m\in \mathds{N}$, a $(\rho,\sigma)$-homogeneous element $R\in L$ and a $(\rho,\sigma)$-homogeneous element $F\in W$, which is not a monomial, such that $n,m>1$, $\gcd(m,n) =1$ and
\begin{xalignat}{2}
&\ell_{\rho,\sigma}(P) = \lambda_P R^m,&&\ell_{\rho,\sigma}(Q) = \lambda_Q R^n,\label{chee}\\
&[F,P]_{\rho,\sigma}= \ell_{\rho,\sigma}(P),&&v_{\rho,\sigma}(F) = \rho+\sigma.\label{chee1}\\
&\en_{\rho,\sigma}(F)\ne (1,1), &&\en_{\rho,\sigma}(F) = \mu \en_{\rho,\sigma}(R) \quad\text{with $0<\mu<1$.}\label{pepa}\\
& \frac{v_{\rho,\sigma}(P)}{v_{\rho,\sigma}(Q)}=\frac{v_{1,1}(P)}{v_{1,1}(Q)}=\frac mn.\label{peepaa}
\end{xalignat}

\begin{proposition}\label{pr no (3,6) ni (4,6)} It is true that
$$
\frac{1}{m}\en_{\rho,\sigma}(P)\notin \{(3,6),(4,6)\}.
$$

\end{proposition}

\begin{proof}
Write
$$
(r,s):=\en_{\rho,\sigma}(R)= \frac{1}{m} \en_{\rho,\sigma}(P)\quad\text{and}\quad (r',s'):= (\mu r,\mu s) = \en_{\rho,\sigma}(F).
$$
Since
$$
\rho r'+\sigma s' = v_{\rho,\sigma}(F) = \rho+\sigma,
$$
we have
\begin{equation}
\rho(r'-1) = -\sigma(s'-1),\label{pepea}
\end{equation}
which determines $(\rho,\sigma)$ as a function of $\en_{\rho,\sigma}(F)$, because $\gcd(\rho,\sigma)=1$, $\sigma<0$ and $\en_{\rho,\sigma}(F)\ne (1,1)$. Note that the equality~\eqref{pepea} means that
$$
(\rho,\sigma) = \val\bigl((r',s')-(1,1)\bigr).
$$
Also, we note that $R$ is not a monomial since $(\rho,\sigma)\in \Val(P)$, and so there exists $\gamma\in \mathds{N}$, such that
\begin{equation}
\st_{\rho,\sigma}(R) = (r+\gamma\sigma,s-\gamma\rho)\in \mathds{N}_0\times \mathds{N}_0.\label{pepesa}
\end{equation}
Next, we prove separately that $(r,s)\ne (3,6)$ and $(r,s)\ne (4,6)$.

\smallskip

\noindent\bf Proof of $\bm{(r,s)\ne (3,6)}$.\rm\enspace Assume by contradiction that $(r,s)=(3,6)$. Then, by the equality in~\eqref{pepa} and~\cite[Prop.~6.3~(4)]{G-G-V2}, necessarily
\begin{equation}
(r',s'):=\en_{\rho,\sigma}(F)=(2,4).\label{eqqq3}
\end{equation}
Hence, by~\cite[Prop.~6.3~(6)]{G-G-V2}, and equalities~\eqref{pepea} and~\eqref{pepesa},
\begin{equation}
(\rho,\sigma)=(3,-1)\quad\text{and}\quad \st_{\rho,\sigma}(R) = (1,0).\label{eqqq4}
\end{equation}
Note now that, by~\eqref{eqqq3} and~\cite[Lemma~5.6]{G-G-V2},
$$
\ell_{3,-1}(F)=\mu_0 xy +\mu_1 x^2y^4\qquad\text{with $\mu_0,\mu_1\in K^{\times}$}.
$$
Hence, by~\cite[Def.~1.20]{G-G-V2},
$$
f_{F,3,-1}=\mu_0+\mu_1x.
$$
Moreover, since, by~\cite[Def.~1.20]{G-G-V2},
$$
\st_{3,-1}(R)=(1,0)\quad\text{and}\quad \en_{3,-1}(R)=(3,6),
$$
we have
$$
\deg(f_{R,3,-1})=2.
$$
So, by the first equalities in~\eqref{chee} and~\eqref{chee1}, and~\cite[Rem.~1.21 and Prop.~4.6]{G-G-V2},
$$
f_{R,3,-1}=\mu(\mu_0+\mu_1 x)^2\quad\text{with $\mu\in K^{\times}$.}
$$
Consequently, by Remark~\ref{f y f^{(l)}},
$$
f_{R,3,-1}^{(1)}=\mu(\mu_0+\mu_1 x^3)^2.
$$
Hence, by the first equality in~\eqref{chee}, the second one in~\eqref{eqqq4}, item~(4) of Proposition~\ref{pr v de un producto} and Remark~\ref{f de un producto},
$$
\st_{3,-1}(P)=(m,0)\quad\text{and}\quad f_{P,3,-1}^{(1)}=\mu^m(\mu_0+\mu_1 x^3)^{2m}.
$$
This implies that the polynomial $\mathfrak{f}$, introduced at the beginning of this Section, e\-quals $f_{P,3,-1}^{(1)}$, and that the multiplicity $m_{\lambda}$ of each linear factor $x-\lambda$ of $\mathfrak{f}$ equals~$2m$. So, if we verify that Conditions~(a)--(g) of Proposition~\ref{preparatoria} are satisfied, we can and will apply it with $\lambda\!\in\! K$ an arbitrary root of $\mathfrak{f}$. Now we proceed to check the a\-bove mentioned conditions. Items~(a)--(b) are trivial and items~(c)--(f) follow from~\cite[Prop.~3.6,~3.7 and Rem.~3.9]{G-G-V2}. Finally, item~(g) is satisfied, since, by~\eqref{chee} and~\cite[Prop~1.9~(5)]{G-G-V2}, we have
\begin{align*}
& v_{1,-1}\bigl(\en_{3,-1}(P)\bigr) = m v_{1,-1}\bigl(\en_{3,-1}(R)\bigr) = -3 m
\shortintertext{and}
& v_{1,-1}\bigl(\en_{3,-1}(Q)\bigr) = n v_{1,-1}\bigl(\en_{3,-1}(R)\bigr) = -3 n.
\end{align*}
Let $\varphi\in \Aut(W^{(3)})$ and $(\rho',\sigma')\in \mathfrak{V}$ be as in Proposition~\ref{preparatoria}. Set
$$
P_1:=\varphi(P),\quad Q_1:=\varphi(Q),\quad \rho_1:=\rho'\quad\text{and}\quad\sigma_1:=\sigma'.
$$
By item~(1), (3), (4) and~(6) of Proposition~\ref{preparatoria} we know that $P_1$, $Q_1$ and $(\rho_1,\sigma_1)$ satisfy conditions~(a), (c), (d) and~(f) in the statement of that proposition, and that
\begin{equation}
\en_{\rho_1,\sigma_1}(P_1) = \st_{3,-1}(P_1) = (5m/3,2m).\label{eqqq8}
\end{equation}
Moreover condition~(b) of Proposition~\ref{preparatoria} is trivially satisfied. This contradicts Proposition~\ref{lema general} for $P_1$, $Q_1$, $l=3$, $(\rho_1,\sigma_1)$ and $\bar{r}=3$, and eliminates so the case ${(r,s)= (3,6)}$.
\smallskip

\noindent\bf Proof of $\bm{(r,s)\ne (4,6)}$.\rm\enspace Assume by contradiction that $(r,s)=(4,6)$. Then, by equality~\eqref{pepa} and~\cite[Prop.~6.3~(4)]{G-G-V2}, necessarily
\begin{equation}
(r',s'):=\en_{\rho,\sigma}(F)=(2,3).\label{eqnn1}
\end{equation}
Hence, by~\cite[Prop.~6.3~(6)]{G-G-V2}, and equalities~\eqref{pepea} and~\eqref{pepesa},
\begin{equation}
(\rho,\sigma)=(2,-1)\quad\text{and}\quad \st_{\rho,\sigma}(R) = (1,0).\label{eqnn2}
\end{equation}
Note now that, by~\eqref{eqqq3} and~\cite[Lemma~5.6]{G-G-V2},
$$
\ell_{2,-1}(F) = \mu_0 xy +\mu_1 x^2y^3\qquad\text{with $\mu_0,\mu_1\in K^{\times}$}.
$$
Hence, by~\cite[Def.~1.20]{G-G-V2},
$$
f_{F,2,-1}=\mu_0+\mu_1x.
$$
Moreover, since, by~\cite[Def.~1.20]{G-G-V2},
$$
\st_{2,-1}(R)=(1,0)\quad\text{and}\quad \en_{2,-1}(R)=(4,6),
$$
we have
$$
\deg(f_{R,2,-1})=3.
$$
So, by equality~\eqref{chee} and~\cite[Rem.~1.21 and Prop.~4.6]{G-G-V2},
$$
f_{R,2,-1}=\mu(\mu_0+\mu_1 x)^3\quad\text{with $\mu\in K^{\times}$.}
$$
Consequently, by Remark~\ref{f y f^{(l)}},
$$
f_{R,3,-1}^{(1)}=\mu(\mu_0+\mu_1 x^2)^3.
$$
Hence, by the first equality in~\eqref{chee}, the second one in~\eqref{eqqq4}, item~(4) of Proposition~\ref{pr v de un producto} and Remark~\ref{f de un producto},
$$
\st_{2,-1}(P)=(m,0)\quad\text{and}\quad f_{P,2,-1}^{(1)}=\mu^m(\mu_0+\mu_1 x^2)^{3m}.
$$
This implies that the polynomial $\mathfrak{f}$, introduced at the beginning of this Section, e\-quals $f_{P,2,-1}^{(1)}$, and that the multiplicity $m_{\lambda}$ of each linear factor $x-\lambda$ of $\mathfrak{f}$ equals~$3m$. So, if we verify that Conditions~(a)--(g) of Proposition~\ref{preparatoria} are satisfied, we can and will apply it with $\lambda\in K$ an arbitrary root of $\mathfrak{f}$. Now we proceed to check the above mentioned conditions. Items~(a)--(b) are trivial and items~(c)--(f) follow from~\cite[Prop.~3.6,~3.7 and Rem.~3.9]{G-G-V2}. Finally, item~(g) is satisfied, since, by~\eqref{chee} and~\cite[Prop~1.9~(5)]{G-G-V2}, we have
\begin{align*}
& v_{1,-1}\bigl(\en_{2,-1}(P)\bigr) = m v_{1,-1}\bigl(\en_{2,-1}(R)\bigr) = -2 m
\shortintertext{and}
& v_{1,-1}\bigl(\en_{2,-1}(Q)\bigr) = n v_{1,-1}\bigl(\en_{2,-1}(R)\bigr) = -2 n.
\end{align*}
Let $\varphi\in \Aut(W^{(2)})$ and $(\rho',\sigma')\in \mathfrak{V}$ be as in Proposition~\ref{preparatoria}. Set
$$
P_1:=\varphi(P),\quad Q_1:=\varphi(Q),\quad \rho_1:=\rho'\quad\text{and}\quad\sigma_1:=\sigma'.
$$
By item~(1), (3), (4) and~(6) of Proposition~\ref{preparatoria} we know that $P_1$, $Q_1$ and $(\rho_1,\sigma_1)$ satisfy conditions~(a), (c), (d) and~(f) in the statement of that proposition, and that
\begin{equation}
\en_{\rho_1,\sigma_1}(P_1) = \st_{2,-1}(P_1) = (5m/2,3m).\label{eqnn3}
\end{equation}
Moreover condition~(b) of Proposition~\ref{preparatoria} is trivially satisfied. This contradicts Proposition~\ref{lema general} for $P_1$, $Q_1$, $l=2$, $(\rho_1,\sigma_1)$ and $\bar{r}=2$, and eliminates so the case ${(r,s)= (4,6)}$.
\end{proof}

\begin{proposition}\label{cota} Let
$$
B:=\min\{\gcd(v_{1,1}(P),v_{1,1}(Q)), \text{ where $(P,Q)$ is an irreducible pair}\}.
$$
We have $B>14$.
\end{proposition}

\begin{proof} Let $(P,Q)$ be an irreducible pair such that $B=\gcd(v_{1,1}(P),v_{1,1}(Q))$. Let $(P_0,Q_0)$ and $(\rho,\sigma)$ be as in~\cite[Prop.~6.2]{G-G-V2}. Thus, $(P_0,Q_0)$ is an irreducible pair, $\sigma<0$ and $(\rho,\sigma)\in\Val(P_0)\cap \Val(Q_0)$. Hence, by~\cite[Prop.~6.3~(1)]{G-G-V2}, there exists $\lambda_{P_0},\lambda_{Q_0}\in K^{\times}$, a $(\rho,\sigma)$-ho\-mo\-geneous element $R\in L\setminus\{0\}$, and $m,n\in \mathds{N}$, with $m,n>1$ and $\gcd(m,n) = 1$, such that
\begin{equation}
\ell_{\rho,\sigma}(P_0) = \lambda_{P_0} R^m,\qquad \ell_{\rho,\sigma}(Q_0) = \lambda_{Q_0} R^n\label{eccua2}
\end{equation}
and
\begin{equation}
\frac{m}{n} = \frac{v_{\rho,\sigma}(P_0)}{v_{\rho,\sigma}(Q_0)} = \frac{v_{1,1}(P_0)}{v_{1,1}(Q_0)}.\label{eccua3}
\end{equation}
From the first equality in~\eqref{eccua2} and~\cite[Prop.~1.9~(5)]{G-G-V2}, it follows that
$$
\frac{1}{m}\en_{\rho,\sigma}(P_0) = \en_{\rho,\sigma}(R),
$$
while, equality~\eqref{eccua3} implies that
$$
\gcd(v_{1,1}(P_0),v_{1,1}(Q_0)) = \frac{1}{m} v_{1,1}(P_0).
$$
On the other hand, by~\cite[Prop.~6.2~(2)]{G-G-V2},
\begin{equation}
v_{1,1}(P_0) = v_{1,1}(P)\quad\text{and}\quad v_{1,1}(Q_0) = v_{1,1}(Q),\label{eccua1}
\end{equation}
and so
$$
B = \gcd(v_{1,1}(P_0),v_{1,1}(Q_0)) = \frac{1}{m} v_{1,1}(P_0)\ge v_{1,1}(\en_{\rho,\sigma}(R)) = r+s,
$$
where $(r,s) := \en_{\rho,\sigma}(R)$. Since, by~\cite[Prop.~6.2~(g)]{G-G-V2}, we have $r<s$, it follows from~\cite[Prop.~6.3~(3) and~(5)]{G-G-V2}, that if if $r+s\le 14$, then
\begin{equation}
(r,s)\in \{(3,9),(4,8),(6,8),(4,10),(3,6),(4,6)\}.\label{eccua4}
\end{equation}
Furthermore, by conditions~\eqref{chee1} and~\eqref{pepa} there exist a $(\rho,\sigma)$-homogeneous element $F\in W\setminus\{0\}$, such that
\begin{xalignat*}{2}
&[F,P_0]_{\rho,\sigma}= \ell_{\rho,\sigma}(P_0),&&v_{\rho,\sigma}(F) = \rho+\sigma,\\
&\en_{\rho,\sigma}(F)\ne (1,1), &&\en_{\rho,\sigma}(F) = \mu \en_{\rho,\sigma}(R) \quad\text{with $0<\mu<1$.}
\end{xalignat*}
Write $(r',s'):=\en_{\rho,\sigma}(F)$, so that
\begin{equation}
(r',s')=\mu(r,s).\label{pep}
\end{equation}
Since
$$
\rho r'+\sigma s' = v_{\rho,\sigma}(F) = \rho+\sigma,
$$
we have
\begin{equation}
\rho(r'-1) = -\sigma(s'-1),\label{pepe}
\end{equation}
which determines $(\rho,\sigma)$ as a function of $(r',s')$ because $\gcd(\rho,\sigma)=1$, $\sigma<0$ and $(r',s')\ne (1,1)$. Finally, since $(\rho,\sigma)\in \Val(P)$, we know that $R$ is not a monomial. Hence there exists $\gamma\in \mathds{N}$, such that
\begin{equation}
\st_{\rho,\sigma}(R) = (r+\gamma\sigma,s-\gamma\rho)\in \mathds{N}_0\times \mathds{N}_0.\label{pepes}
\end{equation}
Now we will analyze each of the possibilities for $(r,s)$ in~\eqref{eccua4} and see that none of them can hold. First of all we note that cases $(r,s)=(3,6)$ and $(r,s)=(4,6)$ are covered by Proposition~\ref{pr no (3,6) ni (4,6)}.

\bigskip

\noindent $\bm{(r,s)=(3,9)}${\bf.}\enspace By~\eqref{pep} and~\cite[Prop.~6.3 (4)]{G-G-V2}, necessarily
$$
(r',s'):=\en_{\rho,\sigma}(F)=(2,6).
$$
Hence, by~\eqref{pepe} and~\eqref{pepes},
\begin{equation*}
(\rho,\sigma)=(5,-1)\quad\text{and}\quad \st_{\rho,\sigma}(R)=(2,4),
\end{equation*}
which contradicts~\cite[Prop.~6.3~(6)]{G-G-V2}.

\medskip

\noindent $\bm{(r,s)=(4,8)}${\bf.}\enspace By~\eqref{pep} and~\cite[Prop.~6.3~(4)]{G-G-V2}, necessarily
$$
(r',s'):=\en_{\rho,\sigma}(F)\in\{(2,4),(3,6)\}.
$$
If $(r',s') = (2,4)$, then by~\eqref{pepe} and~\eqref{pepes},
$$
(\rho,\sigma)=(3,-1)\quad\text{and}\quad \st_{\rho,\sigma}(R)\in\{(3,5),(2,2)\}
$$
and if $(r',s') = (3,6)$, then again by~\eqref{pepe} and~\eqref{pepes},
$$
(\rho,\sigma)=(5,-2)\quad\text{and}\quad \st_{\rho,\sigma}(R)=(2,3).
$$
Both cases are impossible by~\cite[Prop.~6.3~(6)]{G-G-V2}.

\medskip

\noindent $\bm{(r,s)=(6,8)}${\bf.}\enspace By~\eqref{pep} and~\cite[Prop.~6.3~(4)]{G-G-V2}, necessarily
$$
(r',s'):=\en_{\rho,\sigma}(F) = (3,4).
$$
Hence, by~\eqref{pepe} and~\eqref{pepes},
$$
(\rho,\sigma)=(3,-2)\quad\text{and}\quad \st_{\rho,\sigma}(R)\in\{(4,5),(2,2)\},
$$
which also contradicts~\cite[Prop.~6.3~(6)]{G-G-V2}.

\medskip

\noindent $\bm{(r,s)=(4,10)}${\bf.}\enspace By~\eqref{pep} and~\cite[Prop.~6.3~(4)]{G-G-V2}, necessarily
$$
(r',s'):=\en_{\rho,\sigma}(F) = (2,5).
$$
Hence, by~\eqref{pepe} and~\eqref{pepes},
$$
(\rho,\sigma)=(4,-1)\quad\text{and}\quad \st_{\rho,\sigma}(R)\in\{(3,6),(2,2)\}.
$$
Again by~\cite[Prop.~6.3~(6)]{G-G-V2}, the case $\st_{4,-1}(R) = (2,2)$ is impossible. So, we can assume that $\st_{4,-1}(R) = (3,6)$. As above of~\cite[Lemma~2.4]{G-G-V2}, let
$$
\Valinf_{4,-1}(P_0) := \left\{\val\left(\left(i,j\right)-\st \right):\left(i,j\right)\in \Supp(P_0) \text{ and } v_{1,-1}\left(i,j\right)> v_{1,-1}(\st)\right\},
$$
where $\st := \st_{4,-1}(P_0)$. Since, by~\cite[Prop.~1.9~(4) and Prop.~3.6]{G-G-V2}, and the first equality in~\eqref{eccua2}
$$
v_{1,-1}(\st_{4,-1}(P_0)) = v_{1,-1}(\st_{4,-1}(R)) = -3 m < 0\quad\text{and}\quad v_{1,-1}(P_0)>0,
$$
we have $\Valinf_{4,-1}(P_0)\ne \emptyset$. Let
$$
(\rho_1,\sigma_1) := \Pred_{4,-1}(P_0) :=\max(\Valinf_{4,-1}(P_0)).
$$
By~\cite[Lemma~2.7~(2)]{G-G-V2}, we know that
\begin{equation}
(\rho_1,\sigma_1)\in\Val(P_0)\quad\text{and}\quad \en_{\rho_1,\sigma_1}(P_0) = \st_{4,-1}(P_0).\label{eccua5}
\end{equation}
By~\eqref{chee} and~\eqref{peepaa} there exists $\lambda'_{P_0},\lambda'_{Q_0}\in K^{\times}$, a $(\rho_1,\sigma_1)$-homogeneous element $R_1\in L\setminus\{0\}$, and $m_1,n_1\in \mathds{N}$, with $m_1,n_1>1$, $\gcd(m_1,n_1) = 1$, such that
\begin{equation}
\ell_{\rho_1,\sigma_1}(P_0) = \lambda'_{P_0} R_1^{m_1}\quad\text{and}\quad \ell_{\rho_1,\sigma_1}(Q_0) = \lambda'_{Q_0} R_1^{n_1}.\label{eccua6}
\end{equation}
and
\begin{equation}
\frac{m_1}{n_1} = \frac{v_{\rho_1,\sigma_1}(P_0)}{v_{\rho_1,\sigma_1}(Q_0)} = \frac{v_{1,1}(P_0)}{v_{1,1}(Q_0)}.\label{eccua7}
\end{equation}
Combining the last equality with~\eqref{eccua3}, we obtain that $m_1/n_1 = m/n$, which im\-plies $m_1 = m$ and $n_1 = n$, since $\gcd(m_1,n_1) = 1 = \gcd(m,n)$. Consequently, by the equality in~\eqref{eccua5}, \cite[Prop.~1.9~(4) and~(5)]{G-G-V2}, and the first equalities in~\eqref{eccua2} and~\eqref{eccua6},
$$
\en_{\rho_1,\sigma_1}(R_1) = \frac{1}{m}\en_{\rho_1,\sigma_1}(P_0) = \frac{1}{m} \st_{4,-1}(P_0) = \st_{4,-1}(R_0) = (3,6),
$$
which is impossible by Proposition~\ref{pr no (3,6) ni (4,6)}.
\end{proof}

\section{Compatible complete chains}

\setcounter{equation}{0}

In this last section we construct a sequence of pairs $(P_j,Q_j)$ in $W^{(l_j)}$ using Proposition~\ref{preparatoria}, and prove that the sequence is finite. Associated with this sequence are certain triples $A_j= \bigl(A_j,(\rho_j,\sigma_j),l_j\bigr)$, with $A_j\in \frac 1{l_j}\mathds{Z}\times \mathds{N}$, which form a compatible complete chain of corners of the support of the last pair.

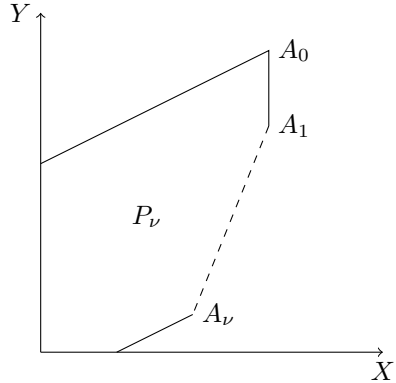
\begin{figure}[H]
\centering

\begin{tikzpicture}
\draw [->] (0,0) -- (0,4.5) node[anchor=east]{$Y$};
\draw [->] (0,0) -- (4.5,0) node[anchor=north]{$X$};
\draw (0,2.5) -- (3,4)node[anchor=west]{$A_0$} -- (3, 3)node[anchor=west]{$A_1$};
\draw (1,0) -- (2,0.5)node[anchor=west]{$A_{\nu}$};
\draw[dashed] (3,3) -- (2,0.5);
\draw (1.4,1.8) node {$P_{\nu}$};
\end{tikzpicture}
\caption{A complete chain $A_0,\dots, A_{\nu}$}
\end{figure}

The smallest of such compatible complete chains, i.e., with $v_{1,1}(A_0)$ minimal, will give a lower bound for $B$, improving Proposition~\ref{cota}. Until now, the smallest compatible chain found starts at $A_0=(9,21)$.

\smallskip

Let $(P,Q)$ be an irreducible pair in $W$. By~\cite[Prop.~3.6 and~3.8]{G-G-V2} there exist $m,n\in \mathds{N}$, such that $m,n>1$, $\gcd(m,n) = 1$ and $\frac{v_{1,1}(P)}{v_{1,1}(Q)} = \frac{m}{n}$.

\begin{theorem}\label{familia} There exist $\nu\in \mathds{N}$, $\psi\in \Aut(W)$ and families
\begin{equation}
\bigl((P_j,Q_j),(\rho_j,\sigma_j),l_j\bigr)_{0\le j\le \nu}\quad\text{and}\quad (\varphi_j)_{1\le j\le \nu}, \label{eqnue10}
\end{equation}
with $P_0 = \psi(P)$, $Q_0 = \psi(Q)$, $l_0=1$, such that
\begin{equation}
(1,0) > (\rho_0,\sigma_0),\quad v_{1,1}(P_0) = v_{1,1}(P),\quad v_{1,1}(Q_0) = v_{1,1}(Q),\label{vimos}
\end{equation}
and
\begin{equation}
\begin{alignedat}{3}
& l_j = \lcm(l_{j-1},\rho_{j-1}),\qquad &&P_j,Q_j\in W^{(l_j)},\qquad &&(\rho_j,\sigma_j)\in \mathfrak{V},\\
&\varphi_j\in \Aut(W^{(l_j)}),\qquad && P_j=\varphi_j(P_{j-1}),\qquad && Q_j=\varphi_j(Q_{j-1}),
\end{alignedat}\label{veamos}
\end{equation}
for all $j\ge 1$. Furthermore they fulfill:

\begin{enumerate}

\smallskip

\item $[Q_j,P_j]=1$ for all $j$,

\smallskip

\item $v_{1,-1}(\en_{\rho_j,\sigma_j}(P_j))<0$ and $v_{1,-1}(\en_{\rho_j,\sigma_j}(Q_j))<0$ for all $j$,

\smallskip

\item $(\rho_j,\sigma_j)\in \Val(P_j)\cap \Val(Q_j)$ for all $j\ge 0$,

\smallskip

\item $v_{\rho_j,\sigma_j}(P_j)>0$ and $v_{\rho_j,\sigma_j}(Q_j)>0$ for all $j$,

\smallskip

\item $\frac{v_{\rho_j,\sigma_j}(P_j)}{v_{\rho_j,\sigma_j}(Q_j)} = \frac{m}{n}$ for all $j$,

\smallskip

\item $\en_{\rho_j,\sigma_j}(P_j) = \frac{m}{n} \en_{\rho_j,\sigma_j}(Q_j)$ for all $j$,

\smallskip

\item $(\rho_{j-1},\sigma_{j-1})>(\rho_j,\sigma_j)$ for $j=1,\dots,\nu$,

\smallskip

\item The equalities
$$
\qquad\qquad v_{\rho_{j-1},\sigma_{j-1}}(P_j)\! =\! v_{\rho_{j-1},\sigma_{j-1}}(P_{j-1})\! \quad\text{and}\!\quad v_{\rho_{j-1},\sigma_{j-1}}(Q_j)\! =\! v_{\rho_{j-1},\sigma_{j-1}}(Q_{j-1})
$$
hold for $j=1,\dots,\nu$,

\smallskip

\item The equalities
$$
\ell_{\rho'',\sigma''}(P_j)=\ell_{\rho'',\sigma''}(P_{j-1})\quad\text{and}\quad \ell_{\rho'',\sigma''}(Q_j)=\ell_{\rho'',\sigma''}(Q_{j-1})
$$
hold for $j=1,\dots,\nu$ and $(\rho_{j-1},\sigma_{j-1}) <(\rho'',\sigma'') < (-1,1)$,

\smallskip

\item $\en_{\rho_j,\sigma_j}(P_j) = \st_{\rho_{j-1},\sigma_{j-1}}(P_j)$ for $j=1,\dots,\nu$,

\smallskip

\item For each $j=1,\dots,\nu$ there exists a $(\rho_{j-1},\sigma_{j-1})$-homogeneous element $F_{j-1}$ in $W^{(l_{j-1})}$, which is not a monomial, such that
\begin{align*}
& v_{\rho_{j-1},\sigma_{j-1}}(F_{j-1}) = \rho_{j-1} + \sigma_{j-1}
\shortintertext{and}
& [P_{j-1},F_{j-1}]_{\rho_{j-1},\sigma_{j-1}} = \ell_{\rho_{j-1},\sigma_{j-1}}(P_{j-1}).
\end{align*}
Furthermore,
$$
\qquad \en_{\rho_{j-1},\sigma_{j-1}}(F_{j-1}) = (1,1) \Longrightarrow \st_{\rho_{j-1},\sigma_{j-1}}(P_j) = \en_{\rho_{j-1},\sigma_{j-1}}(P_{j-1}),
$$

\smallskip

\item $[Q_j,P_j]_{\rho_j,\sigma_j}=0$ for $j=0,\dots, \nu-1$,

\smallskip

\item $[Q_{\nu},P_{\nu}]_{\rho_{\nu},\sigma_{\nu}}=1$.

\smallskip

\end{enumerate}
For the sake of simplicity in the sequel we will write $T_j:= \bigl((P_j,Q_j),(\rho_j,\sigma_j),l_j\bigr)$.

\end{theorem}

\begin{proof} Let $\psi\in \Aut(W)$ and $(P_0,Q_0)$ be as in~\cite[Prop.~6.2]{G-G-V2}. Set
$$
\rho_0:=\rho,\quad \sigma_0:=\sigma,\quad\text{and}\quad l_0:=1.
$$
The assertions in conditions~\eqref{vimos} follow from~\cite[Prop.~6.2~(2) and~(a)]{G-G-V2}. Furthermore, items~(b)--(g) of~\cite[Prop.~6.2]{G-G-V2} imply items~(1)--(5) and~(12) for $T_0$. Finally item~(6) is consequence of item~(5) and~\cite[Rem.~3.11]{G-G-V2} (Note that items~(7)--(11) and~(13) only make sense for $j>0$).

\smallskip

Assume that we have $T_0,\dots,T_{j_0}$ and $\varphi_1,\dots,\varphi_{j_0}$ such that conditions~\eqref{veamos} and items~(1)--(12) are fulfilled for $j\!<\!j_0$, and that conditions~\eqref{veamos} and items~\mbox{(1)--(11)} are fulfilled for $T_{j_0}$ and $\varphi_{j_0}$. If $[Q_{j_0},P_{j_0}]_{ \rho_{j_0},\sigma_{j_0}}\ne 0$, then we set $\nu:=j_0$. Clearly, since $[Q_{\nu},P_{\nu}]=1$, item~(13) is true. If $[Q_{j_0},P_{j_0}]_{\rho_{j_0}, \sigma_{j_0}}=0$, then $(P_{j_0},Q_{j_0})$ and $(\rho_{j_0},\sigma_{j_0})$ fulfill the conditions required to $(P,Q)$ and $(\rho,\sigma)$ in the hypothesis of Proposition~\ref{preparatoria} with $l\!:=\!l_{j_0}$.
Applying that proposition we obtain

\begin{itemize}

\smallskip

\item[-] $(\rho',\sigma')\in \mathfrak{V}$,

\smallskip

\item[-] a $(\rho,\sigma)$-homogeneous element $F$ of $W^{(l)}$,

\smallskip

\item[-] $\varphi\in \Aut(W^{(l')})$, such that
$$
\varphi(X^{\frac{1}{l'}}) = X^{\frac{1}{l'}}\quad\text{and}\quad \varphi(Y) = Y + \lambda X^{\frac{\sigma}{\rho}},
$$
in which $l':=\lcm(\rho_{j_0},l)$ and $\lambda\in K$ is any element such that the multiplicity of $x-\lambda$ in $x^{s_{j_0}}f^{(l)}_{P_{j_0},\rho_{j_0},\sigma_{j_0}}(x)$ is maximum, where $s_{j_0}$ is the second coordinate of $\st_{\rho_{j_0},\sigma_{j_0}}(P_{j_0})$.

\end{itemize}
which enjoy the properties established there, in items~(1)--(11). We set
$$
l_{j_0+1}:=l',\quad \varphi_{j_0+1}:= \varphi,\quad P_{j_0+1}:=\varphi_{j_0+1}(P_{j_0})\quad\text{and}\quad Q_{j_0+1}:=\varphi_{j_0+1}(Q_{j_0}).
$$
Now it is clear that items~(1)--(11) are fulfilled for $j=j_0+1$.

\smallskip

Next we will prove that this process is finite. By item~(8) of Proposition~\ref{preparatoria} we know that
$$
v_{0,1}(\en_{\rho_{j+1},\sigma_{j+1}}(P_{j+1})) \le v_{0,1} (\en_{\rho_j,\sigma_j}(P_j)).
$$
Hence, since $v_{0,1}(\en_{\rho_j,\sigma_j}(P_j))\in \mathds{N}_0$ for all $j$, it suffices to prove that for each $j$ there are only finitely many $k\ge 0$ such that
\begin{equation}\label{esquinasbajan}
v_{0,1}(\en_{\rho_{j+k},\sigma_{j+k}}(P_{j+k})) = v_{0,1}(\en_{\rho_j,\sigma_j}(P_j)).
\end{equation}
We claim that if \eqref{esquinasbajan} is fulfilled for $k=1$ and $j$, then $\rho_j|l_j$, and therefore $l_{j+1}=l_j$. In fact, again by item~(8) of Proposition~\ref{preparatoria}, in this case
$$
\en_{\rho_j,\sigma_j}(P_j)+\Bigl(\frac{\sigma_j}{\rho_j},-1\Bigr)\in \Supp(P_j)\subseteq \frac{1}{l_j}\mathds{Z}\times \mathds{Z}.
$$
Since $\en_{\rho_j,\sigma_j}(P_j)\in \frac{1}{l_j}\mathds{Z}\times \mathds{Z}$, we obtain that $(\sigma_j/\rho_j,-1)\in \frac{1}{l_j}\mathds{Z}\times \mathds{Z}$, i.e.,
$$
\frac{\sigma_j}{\rho_j}=\frac{h}{l_j},
$$
for some $h\in \mathds{Z}$. But then $\rho_j|\sigma_j l_j$, and so $\rho_j|l_j$, since $\gcd(\rho_j,\sigma_j)=1$. This proves the claim.

Now, If~\eqref{esquinasbajan} is fulfilled for $k=0,\dots,k_0$, then $\rho_{j+k}|l_j$ for $k=0,\dots,k_0$. So there are only finitely many $\rho_{j+k}$ possible. But $0<-\sigma_{j+k}<\rho_{j+k}$, so there are only finitely many $(\rho_{j+k},\sigma_{j+k})$ possible. Since $(\rho_{j+k+1},\sigma_{j+k+1})< (\rho_{j+k},\sigma_{j+k})$, we have proved that for each $j$ there are only finitely many $k\ge 0$ such that~\eqref{esquinasbajan} is fulfilled, which concludes the proof of the theorem.
\end{proof}

\smallskip

To each triple $T_j$ as in Theorem~\ref{familia} we associate the triple $S_j := (A_j,(\rho_j,\sigma_j),l_j)$, where $A_j:=\frac{1}{m} \en_{\rho_j,\sigma_j}(P_j)$.

\begin{proposition}\label{chains} Let $(S_j)_{j=0,\dots,\nu}$ be a family associated with the irreducible pair $(P,Q)$, according to Theorem~\ref{familia}. The following facts hold:
\begin{enumerate}

\smallskip

\item $l_0=1$ and $l_j=\lcm(\rho_{j-1},l_{j-1})$ for $j=1,\dots,\nu$,

\smallskip

\item $A_j\in \frac{1}{l_j}\mathds{N}\times \mathds{N}$ for all $j$,

\smallskip

\item $v_{1,-1}(A_j)<0$ and $v_{\rho_j,\sigma_j}(A_j)>0$ for all $j$,

\smallskip

\item $(1,0)>(\rho_0,\sigma_0)$ and $(\rho_{j-1},\sigma_{j-1})>(\rho_j,\sigma_j)$ for $j=1,\dots,\nu$.

\smallskip

\item $v_{\rho_{j-1},\sigma_{j-1}}(A_j)=v_{\rho_{j-1},\sigma_{j-1}}(A_{j-1})$ for $j=1,\dots,\nu$.

\smallskip

\item For all $j$, there exist $A_j'\in \frac{1}{l_j}\mathds{N}\times \mathds{N}$ such  that
$$
v_{\rho_j,\sigma_j}(A'_j)=v_{\rho_j,\sigma_j}(A_j)\quad\text{and}\quad v_{1,-1}(A'_j)>v_{1,-1}(A_j).
$$

\smallskip

\item If $A_j\ne A_{j+1}$, then $\frac{\rho_j+\sigma_j}{v_{\rho_j,\sigma_j}(A_j)} A_j\in \frac{1}{l_j}\mathds{N}\times \mathds{N}$.

\smallskip

\item $v_{\rho_{\nu},\sigma_{\nu}}(A_{\nu})=\frac{\rho_{\nu}+\sigma_{\nu}}{n+m}$.

\smallskip

\end{enumerate}
\end{proposition}

\begin{proof} Item~(1) is true by the unnumbered conclusions in Theorem~\ref{familia}, item~(3) by items~(2) and~(4) of Theorem~\ref{familia}, and item~(4) by item~(7) of the same theorem and the first condition in~\eqref{vimos}. We now prove item~(2). By the definition of $A_j$ and item~(6) of Theorem~\ref{familia},
$$
m A_j = \en_{\rho_j,\sigma_j}(P_j)\in \frac{1}{l_j} \mathds{Z}\times \mathds{N}_0\quad\text{and}\quad n A_j=\en_{\rho_j,\sigma_j}(Q_j)\in \frac{1}{l_j}\mathds{Z}\times \mathds{N}_0.
$$
Since $\gcd(m,n) = 1$ this implies that $A_j\in \frac{1}{l_j} \mathds{Z}\times \mathds{N}_0$. Write $A_j = (c,d)$. It remains to see that $c,d>0$. But this follows from the fact that, by item~(3)
$$
c<d\quad\text{and}\quad \rho_j c >- \sigma_jd.
$$
In fact, since $\rho_j>-\sigma_j>0$, from the second inequality it follows that if $c\le 0$, then $d<0$, which is impossible since $d\in \mathds{N}_0$. Consequently $c>0$, which, for the first inequality, implies that $d>0$.

Item~(5) follows immediately from the fact that, by items~(8) and~(10) of Theorem~\ref{familia}
$$
v_{\rho_{j-1},\sigma_{j-1}}(P_{j-1}) = v_{\rho_{j-1},\sigma_{j-1}}(P_j) =v_{\rho_{j-1},\sigma_{j-1}} (\en_{\rho_j,\sigma_j}(P_j)).
$$
By item~(3) of Theorem~\ref{familia}, to prove item~(6) it suffices to take $A'_j:=\frac{1}{m} \st_{\rho_j,\sigma_j}(P_j)$. Now, let us note that
$$
v_{\rho_{\nu},\sigma_{\nu}}(P_{\nu})+v_{\rho_{\nu},\sigma_{\nu}}(Q_{\nu})-(\rho_{\nu}+\sigma_{\nu})=v_{\rho_{\nu},\sigma_{\nu}}(1)=0,
$$
because $[Q_{\nu},P_{\nu}]_{\rho_{\nu},\sigma_{\nu}}=1$. Hence, since
$$
m v_{\rho_{\nu},\sigma_{\nu}}(A_{\nu}) =v_{\rho_{\nu},\sigma_{\nu}}(P_{\nu})
$$
and, by item~(5) of Theorem~\ref{familia}, we have
$$
n v_{\rho_{\nu},\sigma_{\nu}}(A_{\nu}) = v_{\rho_{\nu},\sigma_{\nu}}(Q_{\nu}),
$$
item~(8) is true. It remains to prove item~(7). Let $F_j\in W^{(l_j)}$ be as in item~(11) of Theorem~\ref{familia}. If $\en_{\rho_j,\sigma_j}(F_j) = (1,1)$, then items~(10) and~(11) of Theorem~\ref{familia} yield
$$
\en_{\rho_{j+1},\sigma_{j+1}}(P_{j+1})) = \en_{\rho_j,\sigma_j}(P_j),
$$
and so $A_j = A_{j+1}$. Thus we can assume $\en_{\rho_j,\sigma_j}(F_j) \sim\en_{\rho_j,\sigma_j}(P_j)$. But then, there exists $\lambda\in \mathds{Q}$ such that
$$
\lambda A_j = \en_{\rho_j,\sigma_j}(F_j) \in \frac{1}{l_j}\mathds{Z}\times \mathds{N}_0,
$$
and applying $v_{\rho_j,\sigma_j}$ we obtain
$$
\rho_j+\sigma_j=v_{\rho_j,\sigma_j}(F_j)=\lambda v_{\rho_j,\sigma_j}(A_j).
$$
So,
$$
\lambda = \frac{\rho_j+\sigma_j}{v_{\rho_j,\sigma_j}(A_j)}.
$$
Write $\lambda A_j = (c,d)$. In order to finish the proof we must check that $c,d>0$. But, this follows immediately from the fact that $A_j\in \frac{1}{l_j}\mathds{N}\times \mathds{N}$ by item~(2), and $\lambda>0$ by item~(3).
\end{proof}

\begin{remark}\label{dato} Assume that $(S_j)_{j=0,\dots,\nu}$ is a family that satisfies Conditions~(1)--(7) of Proposition~\ref{chains}. Then in order that Condition~(8) be also fulfilled it must be
$$
m+n = \frac{\rho_{\nu}+\sigma_{\nu}}{v_{\rho_{\nu},\sigma_{\nu}}(A_{\nu})}.
$$
Note that, there is only a finite number of pairs $(m,n)$ satisfying this equality, such that $m,n>1$ and $\gcd(m,n)=1$.
\end{remark}

\begin{example} We next give some examples of families $(S_j)_{j=0,1,2}$ which fulfill items~(1)-(8) of Proposition~\ref{chains}.

\begin{enumerate}

\item The first family $(S_j)_{j=0,1,2}$ is
\begin{align*}
S_0 & =((9,21),(3,-1),1),\\
S_1 & =((13/3,7),(5,-3),3),\\
S_2 & =((11/15,1),(3,-2),15).
\end{align*}
By Remark~\ref{dato}, we know that $n+m = 5$. Consequently $(m,n) = (2,3)$ or $(m,n) = (3,2)$.  Assume that the family $(S_j)_{j=0,1,2}$ is constructed from a irreducible pair $(P,Q)$, according to Proposition~\ref{chains}. If $(m,n) = (2,3)$, then by the definition of $A_0$ and item~(6) of Theorem~\ref{familia},
$$
\en_{3,-1}(P_0) = (18,42)\quad\text{and}\quad \en_{3,-1}(Q_0) = (27,63).
$$
Similarly, if $(m,n) = (3,2)$, then
$$
\en_{3,-1}(P_0) = (27,63)\quad\text{and}\quad \en_{3,-1}(Q_0) = (18,42).
$$

\item The second family $(S_j)_{j=0,1,2}$ is
\begin{align*}
S_0 & =((6,30),(6,-1),1),\\
S_1 & =((3/2,3),(9,-4),6),\\
S_2 & =((11/18,1),(9,-5),18).
\end{align*}
By Remark~\ref{dato}, we know that $n+m = 8$. Consequently $(m,n) = (3,5)$ or $(m,n) = (5,3)$. Assume that the family $(S_j)_{j=0,1,2}$ is constructed from a irreducible pair $(P,Q)$, according to Proposition~\ref{chains}. If $(m,n) = (3,5)$, then by the definition of $A_0$ and item~(6) of Theorem~\ref{familia},
$$
\en_{6,-1}(P_0) = (18,90)\quad\text{and}\quad \en_{6,-1}(Q_0) = (30,150).
$$

\item The third family $(S_j)_{j=0,1,2}$ is
\begin{align*}
S_0 & =((9,36),(9,-2),1),\\
S_1 & =((5/3,3),(2,-1),9),\\
S_2 & =((2/3,1),(18,-11),18).
\end{align*}
By Remark~\ref{dato}, we know that $n+m = 7$. Consequently $(m,n) = (2,5)$, $(m,n) = (3,4)$, $(m,n) = (4,3)$ or $(m,n) = (5,2)$. Assume that the family $(S_j)_{j=0,1,2}$ is constructed from a irreducible pair $(P,Q)$, according to Proposition~\ref{chains}. If $(m,n) = (2,5)$, then by the definition of $A_0$ and item~(6) of Theorem~\ref{familia},
$$
\en_{9,-2}(P_0) = (18,72)\quad\text{and}\quad \en_{9,-2}(Q_0) = (35,180).
$$
Similarly, if $(m,n) = (3,4)$, then
$$
\en_{9,-2}(P_0) = (27,108)\quad\text{and}\quad \en_{9,-2}(Q_0) = (36,144).
$$

\item The fourth family $(S_j)_{j=0,1,2}$ is
\begin{align*}
S_0 & =((14,42),(4,-1),1),\\
S_1 & =((6,10),(7,-4),4),\\
S_2 & =((6/7,1),(28,-23),28).
\end{align*}
By Remark~\ref{dato}, we know that $n+m = 5$. Consequently $(m,n) = (2,3)$ or $(m,n) = (3,2)$. Assume that the family $(S_j)_{j=0,1,2}$ is constructed from a irreducible pair $(P,Q)$, according to Proposition~\ref{chains}. If $(m,n) = (2,3)$, then by the definition of $A_0$ and item~(6) of Theorem~\ref{familia},
$$
\en_{4,-1}(P_0) = (28,84)\quad\text{and}\quad \en_{4,-1}(Q_0) = (42,126).
$$

\item The fifth family $(S_j)_{j=0,1,2}$ is
\begin{align*}
S_0 & =((17,85),(17,-3),1),\\
S_1 & =((46/17,4),(17,-11),17),\\
S_2 & =((13/17,1),(17,-12),17).
\end{align*}
By Remark~\ref{dato}, we know that $n+m = 5$. Consequently $(m,n) = (2,3)$ or $(m,n) = (3,2)$. Assume that the family $(S_j)_{j=0,1,2}$ is constructed from a irreducible pair $(P,Q)$, according to Proposition~\ref{chains}. If $(m,n) = (2,3)$, then by the definition of $A_0$ and item~(6) of Theorem~\ref{familia},
$$
\en_{17,-3}(P_0) = (34,170)\quad\text{and}\quad \en_{17,-3}(Q_0) = (51,255).
$$

\smallskip

\end{enumerate}
\end{example}

\begin{remark} After we set the first version of~\cite{G-G-V2} on arXiv, Yucai Su draw our attention to~\cite{S}, where some properties of so called Dixmier pairs are studied. In particular Theorem~5.2 of~\cite{S} states that a certain shape of Dixmier pairs can be achieved. Via automorphisms we have brought an irreducible pair into the shape of $(P_{\nu},Q_{\nu})$, this shape can be further be brought into the same shape stated in~\cite{S}, by the automorphism of $W^{(l)}$ given by
$$
X^{1/l}\mapsto \left(\frac{\rho+\sigma}{\rho}\right)^{1/l}X^{\rho/(l(\rho+\sigma))}\quad\text{and}\quad
Y\mapsto X^{1-\rho/(\rho+\sigma)}Y,
$$
where $\rho=\rho_{\nu}$, $\sigma=\sigma_{\nu}$ and $l=\lcm(\rho_r+\sigma_r,l_r)$.
\end{remark}

\end{document}